\documentclass[12pt, a4paper, oneside]{Thesismain} % Paper size, default font size and one-sided paper
\usepackage{graphicx}
\graphicspath{{Graphics}} % Specifies the directory where pictures are stored
\usepackage{epstopdf}
\usepackage{setspace}
\hypersetup{urlcolor=black, linkcolor=black, citecolor=black, hypertexnames=true} % Colors hyperlinks in blue - change to black if annoying
\title{\ttitle} % Defines the thesis title - don't touch this

\usepackage{makeidx} \makeindex

\usepackage[refpage,noprefix]{nomencl} \makenomenclature

\usepackage{etoolbox}
\patchcmd{\theindex}{\thispagestyle{plain}}{}{}{}

\usepackage{changepage}

\usepackage{amstext,amsmath,amsthm,amsfonts,amssymb,enumerate,amscd,cite,fancyhdr,bm,mathrsfs,mathtools,graphicx,tikz,thmtools,upgreek}
\usepackage[alphabetic]{amsrefs}
\usetikzlibrary{decorations.pathreplacing,angles,quotes,arrows}
\usepackage{xcolor}

\usepackage{subfigure}
\usepackage{float}
\usepackage{mathtools}
\usepackage{tikz-cd}
\usetikzlibrary{positioning}
\tikzset{node distance=1.5cm, auto}
\usepackage{quiver}
\usepackage{url}

\tikzset{no body/.style={/tikz/dash pattern=on 0 off 1mm}}

\newcommand{\ZZ}{\ensuremath{\mathbb{Z}}}

\newcommand{\C}{C_2}
\newcommand{\Rdelta}{{\mathbb{R}^\delta}}
\newcommand{\R}{\mathbb{R}}
\newcommand{\RC}{\mathbb{R}[\C]}
\newcommand{\CL}{C_2\mathcal{L}}
\newcommand{\Jmor}[3]{C_2\mathcal{J}_{#1}(#2,#3)}
\newcommand{\Gmor}[3]{C_2\mathcal{\gamma}_{#1}(#2,#3)}
\newcommand{\CTop}{\C\Top}
\newcommand{\COTop}{(O(p,q)\rtimes\C)\Top}
\newcommand{\Jpq}{C_2\mathcal{J}_{p,q}}
\newcommand{\Jzero}{C_2\mathcal{J}_{0,0}}
\newcommand{\Jzerobar}{\overline{C_2\mathcal{J}_{0,0}}}
\newcommand{\Epq}{C_2\mathcal{E}_{p,q}}
\newcommand{\Ezero}{C_2\mathcal{E}_{0,0}}
\newcommand{\LL}{ \mathcal{L} }
\newcommand{\OEpq}{O(p,q)C_2\mathcal{E}_{p,q}}
\newcommand{\OEpql}{O(p,q)C_2\mathcal{E}_{p,q}^{l}}
\newcommand{\CJ}[2]{C_2\mathcal{J}_{#1,#2}}
\newcommand{\CE}[2]{C_2\mathcal{E}_{#1,#2}}
\newcommand{\OCE}[2]{O(#1,#2)C_2\mathcal{E}_{#1,#2}}
\newcommand{\taupq}{\tau_{p,q}}
\newcommand{\Tpq}{T_{p,q}}
\newcommand{\pqs}{(p,q)\mathbb{S}}
\newcommand{\Ilevel}{I_{level}}
\newcommand{\Jlevel}{J_{level}}
\newcommand{\Jstable}{J_{stable}}

\DeclareMathOperator{\Top}{Top}
\DeclareMathOperator{\ind}{ind}
\DeclareMathOperator{\Nat}{Nat}
\DeclareMathOperator{\res}{res}
\DeclareMathOperator{\colim}{colim}
\DeclareMathOperator{\hocolim}{hocolim}
\DeclareMathOperator{\Ev}{Ev}
\DeclareMathOperator{\hofibre}{hofibre}
\DeclareMathOperator{\holim}{holim}
\DeclareMathOperator{\Dim}{dim}
\DeclareMathOperator{\Min}{min}
\DeclareMathOperator{\conn}{conn}
\DeclareMathOperator{\Ho}{Ho}
\DeclareMathOperator{\CI}{CI}
\DeclareMathOperator{\proj}{proj}
\DeclareMathOperator{\poly}{-poly-}
\DeclareMathOperator{\homog}{-homog-}
\DeclareMathOperator{\id}{id}
\DeclareMathOperator{\Id}{Id}
\DeclareMathOperator{\inc}{inc}
\DeclareMathOperator{\Fun}{Fun}

\begin{document}

\frontmatter % Use roman page numbering style (i, ii, iii, iv...) for the pre-content pages

\setstretch{1.3} % Line spacing of 1.3

% Define the page headers using the FancyHdr package and set up for one-sided printing
\fancyhead{} % Clears all page headers and footers
\rhead{\thepage} % Sets the right side header to show the page number
\lhead{} % Clears the left side page header

\pagestyle{fancy} % Finally, use the "fancy" page style to implement the FancyHdr headers
%The standard headers are given at the start of chapters; in the Results chapter I have given an example of how to manually change the header to reflect the current section.

\newcommand{\HRule}{\rule{\linewidth}{0.5mm}} % New command to make the lines in the title page

% PDF meta-data
\hypersetup{pdftitle={\ttitle}}
\hypersetup{pdfsubject=\subjectname}
\hypersetup{pdfauthor=\authornames}
\hypersetup{pdfkeywords=\keywordnames}

%----------------------------------------------------------------------------------------
%	TITLE PAGE
%----------------------------------------------------------------------------------------

\begin{titlepage}

\pagenumbering{Roman}

\begin{center}

\textsc{\LARGE \univname}\\[1.5cm] % University name
\textsc{\Large Doctoral Thesis}\\[0.5cm] % Thesis type

\HRule \\[0.4cm] % Horizontal line
{\huge \bfseries \ttitle}\\[0.4cm] % Thesis title
\HRule \\[1.5cm] % Horizontal line

\begin{minipage}{0.4\textwidth}
\begin{flushleft} \large
\emph{Author:}\\
{\authornames} % Author name 
\end{flushleft}
\end{minipage}
\begin{minipage}{0.4\textwidth}
\begin{flushright} \large
\emph{Supervisor:} \\
{\supname} % Supervisor name 
\end{flushright}
\end{minipage}\\[3cm]

\large \textit{A thesis submitted for the degree of}\\[0.1cm] \textrm{\degreename}\\[0.3cm] % University requirement text
\textit{in the}\\[0.4cm]
\groupname\\\deptname\\[1cm] % Research group name and department name

{\large \today}\\[1cm] % Date
\includegraphics[width=2cm]{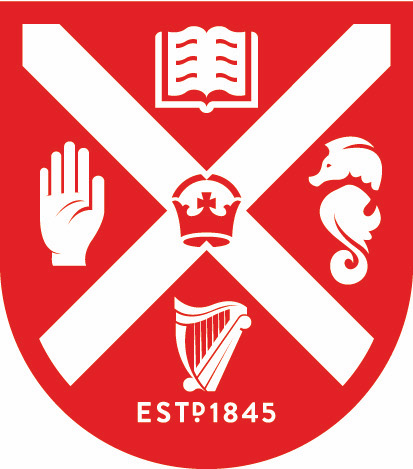} % University/department logo - uncomment to place it

\vfill
\end{center}

\end{titlepage}

\pagenumbering{roman}

%-------------------------------------------
%	ABSTRACT PAGE
%----------------------------------------------------------------------------------------

\addtotoc{Abstract} % Add the "Abstract" page entry to the Contents

\abstract{\addtocontents{toc}{\vspace{1em}} % Add a gap in the Contents, for aesthetics

In this thesis, we construct a new version of orthogonal calculus for functors $F$ from $\C$-representations to $\C$-spaces, where $\C$ is the cyclic group of order 2. For example, the functor $BO(-):V\mapsto BO(V)$, where $BO(V)$ is the classifying space of the orthogonal group $O(V)$, which has a $\C$-action induced by the action on the $\C$-representation $V$. We obtain a bigraded sequence of approximations to $F$, called the strongly $(p,q)$-polynomial approximations $T_{p,q}F$. The bigrading arises from the bigrading on $\C$-representations. The homotopy fibre $D_{p,q}F$ of the map $T_{p+1,q}T_{p,q+1}F\rightarrow \Tpq F$ is such that the approximation $T_{p+1,q}T_{p,q+1}D_{p,q}F$ is equivalent to the functor $D_{p,q}F$ itself and the approximation $T_{p,q}D_{p,q}F$ is trivial. A functor with these properties is called $(p,q)$-homogeneous. Via a zig-zag of Quillen equivalences, we prove that $(p,q)$-homogeneous functors are fully determined by orthogonal spectra with a genuine action of $\C$ and a naive action of the orthogonal group $O(p,q):=O(\mathbb{R}^{p+q\delta})$. The notation $\delta$ is used to represent the sign $\C$-representation, and $\C$ acts on $O(p,q)$ by conjugation. Hence, the fibres $D_{p,q}F$ are $(O(p,q)\rtimes \C)$-spectra over a non-trivial incomplete universe.

}

\clearpage % Start a new page

%----------------------------------------------------------------------------------------
%	ACKNOWLEDGEMENTS
%----------------------------------------------------------------------------------------

\setstretch{1.5} % Reset the line-spacing to 1.3 for body text (if it has changed)

\acknowledgements{\addtocontents{toc}{\vspace{1em}} % Add a gap in the Contents, for aesthetics

I would like to express my gratitude to my supervisor, David Barnes, whose support, guidance and expertise have been instrumental in the completion of this thesis. My thanks also go to my second supervisor Lisa McFetridge, for her support and encouragement throughout my postgraduate studies. A special thanks also goes to Niall Taggart, whose knowledge, advice and time have been invaluable. 

%I would like to take this opportunity to place on record my thanks to several people who played a part, big and small, in this thesis coming together. The first is my supervisor Dr Florian Pausinger, whose support has been fundamental, generous and highly appreciated over the last three years. Thanks should also go to Dr Ying-Fen Lin who took a keen interest in my studies.

%My sincere thanks go to Dr David Barnes, our Postgraduate Coordinator and Head of Centre for his time and assistance. A word of thanks should also go to Mrs Naoimh Mackel and Miss Emma Finnegan for their advice and expertise in the administrative processes throughout each stage of my PhD.

%I want to thank Jenna, who has been constantly supportive and motivating when it has been greatly needed. I also want to wish thanks to my parents Mark and Patricia, who encouraged me to take the opportunity to further my education from a young age.

%Last but not least I want to thank my PhD colleagues; Gareth, Emel, Olivia and Melanie to name a few for providing some lighthearted relief from mathematics both inside and outside of the office. I wish all of them well in the rest of their studies and for the future.
}

\clearpage % Start a new page

%----------------------------------------------------------------------------------------
%	LIST OF CONTENTS/FIGURES/TABLES PAGES
%----------------------------------------------------------------------------------------

\pagestyle{fancy} % The page style headers have been "empty" all this time, now use the "fancy" headers as defined before to bring them back
\setstretch{1.5}
\fancyhead[C]{Contents} % Set the centre header to "Contents"
\tableofcontents % Write out the Table of Contents

%\lhead{\emph{List of Figures}} % Set the left side page header to "List of Figures" (uncomment to include)
%\listoffigures % Write out the List of Figures

%\lhead{\emph{List of Tables}} % Set the left side page header to "List of Tables" (uncomment to include)
%\listoftables % Write out the List of Tables

%----------------------------------------------------------------------------------------
%	ABBREVIATIONS (again uncomment to use) probably not necessary
%----------------------------------------------------------------------------------------

%\clearpage % Start a new page

%\setstretch{1.5} % Set the line spacing to 1.5, this makes the following tables easier to read

%\lhead{\emph{Abbreviations}} % Set the left side page header to "Abbreviations"
%\listofsymbols{ll} % Include a list of Abbreviations (a table of two columns)
{
%\textbf{LAH} & \textbf{L}ist \textbf{A}bbreviations \textbf{H}ere \\
%\textbf{Acronym} & \textbf{W}hat (it) \textbf{S}tands \textbf{F}or \\
}

%\clearpage % Start a new page

%\lhead{\emph{Symbols}} % Set the left side page header to "Symbols"

%\listofnomenclature{lll} % Include a list of Symbols (a three column table)
%{
%$\vee$&\\
%$\wedge$&\\
% Symbol & Name & Unit \\

%& & \\ % Gap to separate the Roman symbols from the Greek

%$\omega$ & angular frequency & rads$^{-1}$ \\
% Symbol & Name & Unit \\
%}

%----------------------------------------------------------------------------------------
%	DEDICATION (uncomment to include)
%----------------------------------------------------------------------------------------

%\setstretch{1.5} % Return the line spacing back to 1.3

%\pagestyle{empty} % Page style needs to be empty for this page

%\dedicatory{

%} % Dedication text

%\addtocontents{toc}{\vspace{2em}} % Add a gap in the Contents, for aesthetics

%----------------------------------------------------------------------------------------
%	THESIS CONTENT - CHAPTERS
%----------------------------------------------------------------------------------------
\setstretch{1.5} %line spacing required by university regulations
\mainmatter % Begin numeric (1,2,3...) page numbering

\pagestyle{fancy} % Return the page headers back to the "fancy" style

% Include the chapters of the thesis as separate files from the Chapters folder, rename them as you want
% Uncomment the lines as you write the chapters

\chapter{Introduction} %the * means that the introduction will not be numbered as a chapter - change if wanted
\label{ch:Introduction}

\fancyhf{}
\fancyhead[C]{\rightmark}
\fancyhead[R]{\thepage}

\section{Context}
Taylor's Theorem for real functions is well known. Given a real function $f:\R\rightarrow\R$, Taylor's theorem describes how $f$ can be approximated by a sequence of polynomial functions, which are built using the derivatives of $f$. As a result, in order to study complex functions, it suffices to study polynomial functions, which are well understood. 

This concept of `breaking' a function into more manageable pieces is one that can be seen throughout other areas of mathematics. A well known example of this in algebraic topology is the Postnikov tower of a CW space.  Given a CW-complex X, one can construct a tower (inverse system) of spaces $\{P_nX\}_{n\geq 0}$,
% https://q.uiver.app/#q=WzAsNyxbMiwzLCIqIl0sWzIsMiwiUF8xWCJdLFsyLDEsIlBfMlgiXSxbMCwzLCJYIl0sWzQsMiwiRl8xWCJdLFsyLDAsIlxcdmRvdHMiXSxbNCwxLCJGXzJYIl0sWzMsMF0sWzQsMV0sWzIsMV0sWzEsMF0sWzMsMSwiIiwyLHsiY3VydmUiOi0xfV0sWzUsMl0sWzMsMiwiIiwxLHsiY3VydmUiOi0yfV0sWzYsMl1d
\[\begin{tikzcd}
	&& \vdots \\
	&& {P_2X} && {F_2X} \\
	&& {P_1X} && {F_1X} \\
	X && {*}
	\arrow[from=4-1, to=4-3]
	\arrow[from=3-5, to=3-3]
	\arrow[from=2-3, to=3-3]
	\arrow[from=3-3, to=4-3]
	\arrow[curve={height=-6pt}, from=4-1, to=3-3]
	\arrow[from=1-3, to=2-3]
	\arrow[curve={height=-12pt}, from=4-1, to=2-3]
	\arrow[from=2-5, to=2-3]
\end{tikzcd}\]
where each map $P_{n+1}X\rightarrow P_n X$ is a fibration and each fibre $F_nX$ is an Eilenberg-Maclane space $K(\pi_n(X),n)$. Since Eilenberg-MacLane spaces are well understood, one can study the space $X$ by studying the layers of it's Postnikov tower. 

Many objects studied in algebraic topology can be realised as functors. Functor calculus is a method by which one can approximate a given functor by a sequence of functors with `nice' properties, which we call polynomial functors. The resulting sequence of functors is a similar concept to that of a Postnikov tower. Polynomial functors have properties that mimic those of the polynomial functions used in differential calculus. For example, an $n$-polynomial functor is also $(n+1)$-polynomial and the $(n+1)$-st derivative of an $n$-polynomial functor is trivial. 

There are many different branches of functor calculus designed to study different categories of functors. Goodwillie calculus, originally constructed by Goodwillie \cite{Goo90, Goo92, Goo03}, is used to study endofunctors on the category of topological spaces. The fibres of the tower produced by Goodwillie calculus are classified by spectra with an action of the symmetric group $\Sigma_n$. The main focus of this thesis is the orthogonal homotopy calculus first constructed by Weiss in \cite{Wei95}. It is the branch of functor calculus involving the study of functors from the category of finite dimensional real vector spaces to the category of pointed topological spaces. The tower for a functor $F$ produced by orthogonal calculus looks as follows,   
% https://q.uiver.app/#q=WzAsNyxbMiwzLCJGKFxcUl5cXGluZnR5KSJdLFsyLDIsIlRfMUYoVikiXSxbMiwxLCJUXzJGKFYpIl0sWzIsMCwiXFx2ZG90cyJdLFszLDIsIlxcT21lZ2FeXFxpbmZ0eSBcXGxlZnRbXFxsZWZ0KFNeezFWfSBcXHdlZGdlIFxcVGhldGFfRl4xXFxyaWdodClfe2hPKDEpfVxccmlnaHRdIl0sWzAsMywiRihWKSJdLFszLDEsIlxcT21lZ2FeXFxpbmZ0eSBcXGxlZnRbXFxsZWZ0KFNeezJWfSBcXHdlZGdlIFxcVGhldGFfRl4yXFxyaWdodClfe2hPKDIpfVxccmlnaHRdIl0sWzMsMl0sWzIsMV0sWzEsMF0sWzQsMV0sWzUsMF0sWzYsMl0sWzUsMSwiIiwwLHsiY3VydmUiOi0xfV0sWzUsMiwiIiwwLHsiY3VydmUiOi0yfV1d
\[\begin{tikzcd}
	&& \vdots \\
	&& {T_2F(V)} & {\Omega^\infty \left[\left(S^{2V} \wedge \Theta_F^2\right)_{hO(2)}\right]} \\
	&& {T_1F(V)} & {\Omega^\infty \left[\left(S^{1V} \wedge \Theta_F^1\right)_{hO(1)}\right]} \\
	{F(V)} && {F(\R^\infty)}
	\arrow[from=1-3, to=2-3]
	\arrow[from=2-3, to=3-3]
	\arrow[from=3-3, to=4-3]
	\arrow[from=3-4, to=3-3]
	\arrow[from=4-1, to=4-3]
	\arrow[from=2-4, to=2-3]
	\arrow[curve={height=-6pt}, from=4-1, to=3-3]
	\arrow[curve={height=-12pt}, from=4-1, to=2-3]
\end{tikzcd}\]
where each functor $T_nF$ is $n$-polynomial and the $n^\text{th}$ fibre of the tower is fully determined by an orthogonal spectrum $\Theta_F^n$ with an action of the orthogonal group $O(n)$.

Classic examples of functors studied using orthogonal calculus include:
\begin{itemize}
    \item $BO(-):V\mapsto BO(V)$
    \item $B\Top(-):V\mapsto B\Top(V)$
    \item $B\text{Diff}^b(-):V\mapsto B\text{Diff}^b(M\times V)$
\end{itemize}
where $BO(V)$ is the classifying space of the space of linear isometries on V, $B\Top(V)$ is the classifying space of the space of homeomorphisms on V and, for a smooth compact manifold $M$, $B\text{Diff}^b(M\times V)$ is the classifying space of the space of bounded diffeomorphisms on $M\times V$. 

There exist functors, similar to those above, that have group actions. For example, the functor $BO(-):V\rightarrow BO(V)$ that sends a $G$-representation to its classifying space. As such, there is a natural motivation to construct functor calculi that study functors with a group action. An equivariant orthogonal calculus of this type could have applications in many different areas, such as the study of equivariant diffeomorphisms of $G$-manifolds. Extensive research focused on equivariance in the Goodwillie calculus setting has been carried out by Dotto \cite{Dot16a,Dot16b,Dot17} and Dotto and Moi \cite{DM16}. 

It is very difficult to produce variations of orthogonal calculus due to the nature of its construction, see Section \ref{general G future work}. Two successful variations are the unitary calculus and calculus with reality constructed by Taggart in \cite{Tag22unit, Tag22real}. In these calculi, real vector spaces are replaced by complex vector spaces, and in the calculus with reality one takes into consideration the $\C$-action on complex vector spaces given by complex conjugation. The fibres of the towers produced are classified by spectra with an action of the unitary group $U(n)$ for unitary calculus and spectra with an action of $\C\ltimes U(n)$ for calculus with reality. Taggart's calculus with reality provides a great insight of what can be expected from a genuine $C_2$-equivariant orthogonal calculus, and a number of proofs in this thesis were inspired by the extensions of Taggart.

\section{$\C$-equivariant orthogonal calculus}\label{sec: intro c2 calc}

The $\C$-equivariant orthogonal calculus gives a method for studying functors from $\C$-representations to the category of $\C$-spaces. For example, the functor $$BO(-):V\mapsto BO(V),$$where $BO(V)$ is the classifying space of the orthogonal group $O(V)$, which has a $\C$-action induced by the action on the $\C$-representation $V$. Details of this functor, along with it's derivatives, are discussed in Section \ref{sec: BO}.

The main result of the thesis, Theorem \ref{thm A}, is the classification of $(p,q)$-homogeneous functors (defined in Section \ref{sec: homog functors}), which are the $\C$-equivariant analogue of functors that are $n$-homogeneous in orthogonal calculus (functors with polynomial approximations concentrated in degree $n$). We show that $(p,q)$-homogeneous functors are fully determined by genuine orthogonal $\C$-spectra with an action of the orthogonal group $O(p,q):=O(\R^{p+q\delta})$, which has a specified $\C$-action given in Definition \ref{def:O(p,q) and matrix A}. That is, orthogonal spectra with a genuine action of $\C$ and a naive action of $O(p,q)$, denoted $\C Sp^{\mathcal{O}}[O(p,q)]$. In this way, we get a richer equivariant structure compared to that of calculus with reality \cite{Tag22real}, in which the classification is in terms of spectra with a naive action of $\C\ltimes U(n)$.  

\begin{xxthm}[{Theorem \ref{weissclassification}}]\label{thm A}
Let $p,q\geq 1$. If $F$ is a $(p,q)$-homogeneous functor, then $F$ is objectwise weakly equivalent to 
\begin{equation*}
    V\mapsto \Omega^\infty[(S^{(p,q)V}\wedge\Theta_F^{p,q})_{hO(p,q)}],
\end{equation*}
where $\Theta_F^{p,q}\in \C Sp^{\mathcal{O}}[O(p,q)]$ and $(-)_{hO(p,q)}$ denotes homotopy orbits. 

Conversely, every functor of the form 
\begin{equation*}
    V\mapsto \Omega^\infty[(S^{(p,q)V}\wedge\Theta)_{hO(p,q)}],
\end{equation*}
where $\Theta\in \C Sp^{\mathcal{O}}[O(p,q)]$, is $(p,q)$-homogeneous. 
\end{xxthm}

The classification can alternatively be stated in terms of the following zig-zag of Quillen equivalences between the calculus input category $\Ezero$ (of functors from the category of finite dimensional $\C$-representations with inner product and linear isometries to $\C\Top_*$) and the category of genuine orthogonal $\C$-spectra with an action of $O(p,q)$. 

\begin{xxthm}[{Theorem \ref{boclassification} and Theorem \ref{QEstabletospectra}}]\label{thm B}
For all $p,q\geq 1$, there exist Quillen equivalences
\[\begin{tikzcd}
	{(p,q)\homog C_2\mathcal{E}_{0,0}} && {O(p,q)C_2\mathcal{E}_{p,q}^s} && {C_2Sp^O[O(p,q)]}
	\arrow["{\ind_{0,0}^{p,q}\varepsilon^*}"', shift right=2, from=1-1, to=1-3]
	\arrow["{\res_{0,0}^{p,q}/O(p,q)}"', shift right=2, from=1-3, to=1-1]
	\arrow["{(\alpha_{p,q})_!}", shift left=2, from=1-3, to=1-5]
	\arrow["{\alpha_{p,q}^*}", shift left=2, from=1-5, to=1-3]
\end{tikzcd}\]
\end{xxthm}
Here $(p,q)\homog C_2\mathcal{E}_{0,0}$ denotes the $(p,q)$-homogeneous model structure on the input category. This model structure captures the structure of $(p,q)$-homogeneous functors, in that the cofibrant-fibrant objects are exactly the projectively cofibrant $(p,q)$-homogeneous functors. This model structure is detailed in Section \ref{sec:homog ms}. The zig-zag of equivalences is made up of two Quillen equivalences. Differentiation (also called induction) forms a Quillen functor from the $(p,q)$-homogeneous model structure to an intermediate category of functors $O(p,q)\C\mathcal{E}_{p,q}^s$, which is in turn Quillen equivalent to the category of genuine orthogonal $\C$-spectra with an action of $O(p,q)$.

In comparison to the underlying calculus which is indexed over $\mathbb{N}$, $\C$-equivariant orthogonal calculus is bi-indexed over $\mathbb{N}\times\mathbb{N}$. As a result, we can define differentiation in two directions (the $p$-direction and the $q$-direction). These different derivatives act like partial derivatives in differential calculus; in particular, they commute. In Conjecture \ref{ex: BO}, we predict that the two first derivatives of $BO(-)$ are the orthogonal sphere spectrum $\mathbb{S}:V\mapsto S^V$ (in the $p$-direction) and the shifted orthogonal sphere spectrum $\mathbb{S}^{\otimes\delta}:V\mapsto S^{V\otimes \Rdelta}$ (in the $q$-direction), where $\Rdelta$ is the sign representation of $\C$ (See Example \ref{Ex: trivial and sign representation}). 

A key difference between the underlying and $\C$-equivariant orthogonal calculi is an indexing shift, caused by this bi-indexing. In particular, $\tau_n$ in the underlying calculus is defined using the poset of non-zero subspaces $\{0\neq U\subseteq \R^{n+1}\}$ and $\taupq$ in the $\C$-calculus is defined using the poset of non-zero subspaces $\{0\neq U\subseteq \R^{p,q}\}$. To keep notation consistent, the author introduced the new term strongly $(p,q)$-polynomial, see Definition \ref{def: polynomial}. A functor in the input category for $\C$-equivariant orthogonal calculus is then called $(p,q)$-polynomial if and only if it is both strongly $(p+1,q)$-polynomial and strongly $(p,q+1)$-polynomial. In particular, we define the strongly $(p,q)$-polynomial approximation functor $\Tpq$, and the $(p,q)$-polynomial approximation functor is the composition $T_{p+1,q}T_{p,q+1}$. A functor $X$ is $(p,q)$-homogeneous if it is $(p,q)$-polynomial and the strongly $(p,q)$-polynomial approximation $T_{p,q}X$ is trivial. 

\begin{xxthm}[{Theorem \ref{thm: DYpq is homog}}]
The homotopy fibres of the maps 
\begin{equation*}
    T_{p+1,q}T_{p,q+1}F\rightarrow \Tpq F
\end{equation*}
are $(p,q)$-homogeneous, and can therefore be described in terms of genuine orthogonal $\C$-spectra with an action of $O(p,q)$, by the classification given in Theorem \ref{thm A}. 
\end{xxthm}

A key result of the calculus, which makes the classification work, is the existence of the following $\C$-homotopy cofibre sequences. These cofibre sequences tell us that derivatives in $\C$-equivariant orthogonal calculus are well behaved. The notation $\C\mathcal{J}_{p,q}:=\C\mathcal{J}_{\mathbb{R}^{p,q}}$ denotes the $(p,q)$-th jet category whose objects are $\C$-representations and morphisms are given by $\Jmor{p,q}{U}{V}$, which is a $\C$-space (see Definition \ref{def: p,q-jet cat}).

\begin{xxprop}[{Proposition \ref{cofibseq}}]\label{prop D}
For all $U,V,W$ in $\Jzero$, the homotopy cofibre of the map
\begin{equation*}
\Jmor{W}{U\oplus X}{V}\wedge S^{W\otimes X} \rightarrow \Jmor{W}{U}{V}
\end{equation*}
is $\C$-homeomorphic to $\Jmor{W\oplus X}{U}{V}$, where $X=\R$ or $X=\Rdelta$ (see Example \ref{Ex: trivial and sign representation}), and $S^{W\otimes X}$ denotes the one point compactification of $W\otimes X$.
\end{xxprop}  
To replace $X$ with something of higher dimension would mean taking some kind of iteration of cofibre sequences. This indicates that a potentially more involved approach may be needed if one wants to construct this kind of result in a $G$-equivariant orthogonal calculus, for an arbitrary group $G$ which may have irreducible representations of dimension greater than one. As a result, it is not obvious how derivatives should behave for the arbitrary $G$ setting, see Section \ref{general G future work}.

\section{Future work}

In this chapter we provide a brief overview of some future work that naturally follows from the content of this thesis. This acknowledges the importance of having a good working model of $\C$-equivariant orthogonal calculus as a key step towards understanding a more general equivariant orthogonal calculus. 

\subsection{Comparison to orthogonal calculus}\label{sec: forget c2}
As mentioned throughout this thesis, the $\C$-equivariant orthogonal calculus has been constructed in such a way that orthogonal calculus might be recovered via a forgetful functor. Comparisons made between existing functor calculi have depended on the inclusion of sub-automorphism groups. For example, comparisons made by Taggart \cite{Tag21,Tag22comp ,Tag23}   between orthogonal calculus, unitary calculus and calculus with reality, relied on the subgroup inclusions 
\begin{equation*}
    O(n)\hookrightarrow U(n)\hookrightarrow U(n)\rtimes \C.
\end{equation*}
Similarly, one could expect that the comparison between $\C$-equivariant orthogonal calculus and orthogonal calculus may depend on the inclusion 
\begin{equation*}
    O(p,q)\hookrightarrow O(p,q)\rtimes \C.
\end{equation*}

Orthogonal calculus is indexed on the universe $\bigoplus\limits_{i=1}^{\infty} \R$, meaning that functors in orthogonal calculus take finite dimensional real inner product spaces as their input. The regular representation, $\RC$, of $\C$ is $\C$-isomorphic to the direct sum of the trivial and sign representations of $\C$. That is, $\RC\cong \R\oplus\Rdelta$. We index $\C$-equivariant orthogonal calculus on the universe $\bigoplus\limits_{i=1}^{\infty} \RC$, which means that functors in the $\C$-equivariant calculus take finite dimensional real subrepresentations of $\bigoplus\limits_{i=1}^{\infty} \RC$ with an inner product as their input. These subrepresentations have the form $\R^{p}\oplus \R^{q\delta}$ (see Section \ref{sect: input functors}). We define $\Ezero$ to be the category of $\CTop_*$-enriched functors from the category of such representations with linear isometries to $\CTop_*$, and $\C$-equivariant natural transformations (see Definition \ref{jzero and ezero def}). We call $\Ezero$ the input category for $\C$-equivariant orthogonal calculus. 

Much like the calculus with reality case \cite{Tag23}, given a functor in the equivariant input category $\Ezero$, forgetting the $\C$-action on the target space and changing the universe from $\bigoplus\limits_{i=1}^{\infty} \RC$ to $\bigoplus\limits_{i=1}^{\infty} \R$ gives a functor in the input category for orthogonal calculus $\mathcal{E}_0$ (of functors from the category of finite dimensional real inner product spaces with linear isometries to pointed topological spaces). That is, there is a functor from the equivariant input category $\Ezero$ to the underlying input category $\mathcal{E}_0$. It is not clear how this functor behaves, since we choose the fine model structure on $\C\Top_*$, rather than the coarse model structure used in \cite{Tag22real} (see Proposition \ref{finemodelstructure} and Proposition \ref{prop: coarse ms}). 

In future work, these functors should be explored in more detail. In particular, the most interesting comparisons will be made between the functors $\tau_{p,q}$ and $T_{p,q}$ and their non-equivariant equivalents $\tau_n$ and $T_n$. We expect, from comparing the indexing, that the equivariant functors $\tau_{p,q}$ and $T_{p,q}$ should correspond to the non-equivariant functors $\tau_{p+q-1}$ and $T_{p+q-1}$. This is interesting, since this would mean that all functors $\Tpq$ with $p+q=n$ correspond non-equivariantly to the same functor $T_{n-1}$.  

\subsection{$G$-equivariant orthogonal calculus}\label{general G future work}

The main difficulty in generalising the $\C$-equivariant orthogonal calculus to $G$-equivariant orthogonal calculus, for an arbitrary group $G$, is the cofibre sequence of Proposition \ref{prop D}. The cofibre sequence holds for $X=\R$ and $X=\Rdelta$, since $\R$ and $\Rdelta$ are one-dimensional irreducible $\C$-representations. Replacing $X$ with a representation of dimension greater than one would require an iterated cofibre sequence. In particular, this indicates that replicating this type of cofibre sequence for a general group $G$ could be difficult, since $G$ might have irreducible representations with dimension greater than one. 

This difficulty can be avoided if one restricts to abelian groups and the complex setting, since every irreducible representation of a finite abelian group over $\mathbb{C}$ is one-dimensional. Therefore, it should be possible to construct a $G$-equivariant unitary calculus, for $G$ a finite abelian group. It might be possible to then recover $G$-equivariant orthogonal calculus via a complexification functor, similar to that used by Taggart \cite{Tag21}, however one should be careful to check that this preserves the $G$-equivariance. 

Naturally, one might then want to consider a change of group. Given a group homomorphism $G\rightarrow G'$, one could attempt to construct comparison functors between $G$-equivariant unitary calculus and $G'$-equivariant unitary calculus, and then also in the orthogonal setting.

\subsection{Global equivariant orthogonal calculus}\label{sec:global}
It is well known that equivariant stable homotopy theory can be recovered from global stable homotopy theory, see \cite{Sch18}. This raises the following idea: If one could construct a `global orthogonal calculus', then it might be possible to recover $G$-equivariant orthogonal calculus, for an arbitrary group $G$. 

To construct global stable homotopy theory in \cite[Chapter 4]{Sch18}, one starts by defining what global equivalences and global fibrations of orthogonal spectra are. In particular, this allows one to define a global model structure on the category of orthogonal spectra, in which the weak equivalences are the global equivalences and the fibrations are the global fibrations. 

One could imagine replicating these definitions in the category of orthogonal functors $\mathcal{J}_0\rightarrow\Top_*$. That is, we could define a global model structure on the input category for orthogonal calculus, which could be used in place of the projective model structure. Repeating the same constructions of orthogonal calculus with this new model structure should in theory produce a `global orthogonal calculus'. In future work, these definitions should be made precise, and all of the constructions that follow should be checked carefully.

\subsection{Comparing the equivariant orthogonal and Goodwillie calculi}\label{sec: barnes eldred}

Comparisons between the non-equivariant orthogonal and Goodwillie calculi have been made by Barnes and Eldred \cite{BE16}. Extensive research focused on equivariance in the Goodwillie calculus setting has been carried out by Dotto \cite{Dot16a,Dot16b, Dot17} and Dotto and Moi \cite{DM16}. Differences in the implementation of equivariance make the equivariant calculi difficult to compare. In particular, there is no obvious way to compare the notions of $(p,q)$-polynomial in the $\C$-equivariant orthogonal calculus sense and $J$-excisive in the equivariant Goodwillie sense, for a $\C$-set $J$. Never the less, we make the following conjecture, which is analogous to \cite[Theorem 3.5]{BE16}. 

Let $F:G\Top_*\rightarrow G\Top_*$ and $\mathbb{R}[G]$ be the regular representation of $G$. Define the restriction of $F$ by 
\begin{equation*}
    \res F: V\mapsto F(S^V)
\end{equation*}
for $V$ a finite dimensional sub-representation of $\oplus^{\infty} \mathbb{R}[G]$ with an inner product. 

\begin{conjecture}
    The $G$-equivariant Goodwillie calculus on a functor $F$ is equivalent to $G$-equivariant orthogonal calculus on the functor $\res F$. 
\end{conjecture}

\chapter{Preliminaries}
\label{ch:Preliminaries}

\fancyhf{}
\fancyhead[C]{\rightmark} 
\fancyhead[R]{\thepage}

In this chapter we gather any necessary preliminary material that will aid in reading the main text. We only provide a very concise summary of the material, and provide references of where greater detail can be found in literature. 

\section{Group actions}

The primary goal of the main text is to construct a $\C$-equivariant orthogonal calculus. The functors of interest will take $\C$-representations as their input and output $\C$-spaces. As such, we will rely heavily on standard notation and results from equivariant homotopy theory. In this section, we recall various equivariant constructions and results that will be used throughout the main text. The main resources for this section are Mandell and May \cite{MM02} and May \cite{May96}. 

\subsection{Topological $G$-spaces}\label{sec:Gspace}

In this section, we cover basic definitions and results of $G$-spaces for a compact topological group $G$. This material has been detailed by Mandell and May in \cite[Section 3.1]{MM02}. 

We will use $\Top$ to denote the category of compactly generated weak Hausdorff spaces, and $\Top_*$ to denote the category of pointed compact generated weak Hausdorff spaces with the Quillen model structure below. 

\begin{proposition}\label{prop: Quillen ms}\index{$\Top_*$}
$\Top_*$ is a cofibrantly generated model category with weak equivalences and fibrations given by weak homotopy equivalences and Serre fibrations. The generating cofibrations $I_{\Top_*}$ and acyclic cofibrations $J_{\Top_*}$ are below, where the map $S^{n-1}\rightarrow D^n$ is inclusion as the boundary.
\begin{align*}\index{$I_{\Top_*}$}\index{$J_{\Top_*}$}
    &I_{\Top_*} =\{S^{n-1}_+\rightarrow D^n_+ :n\in \mathbb{N}\}\\
    &J_{\Top_*} =\{D^n_+ \rightarrow (D^n\times I)_+ :n\in \mathbb{N}\}
\end{align*}
\end{proposition}

\begin{definition}
A \emph{$G$-space} is a topological space $Y$ with a continuous left group action of $G$
\begin{equation*}
    G\times Y\rightarrow Y, \quad (g,y)\mapsto gy
\end{equation*}
such that $ey=y$ and $g_1(g_2y)=(g_1g_2)y$.

A \emph{pointed $G$-space} is a pointed topological space $X$ with a continuous left group action of $G$, such that the basepoint $x_0$ of $X$ is fixed (i.e. $g x_0=x_0$ for all $g\in G$). 

An \emph{equivariant map} of (pointed) $G$-spaces is a continuous (pointed) map $f:X\rightarrow Y$ such that $g\circ f=f\circ g$ for all $g\in G$.

Denote the category of pointed $G$-spaces and $G$-equivariant maps by $G\Top_*$. 
\end{definition}

\begin{examples}The following are examples of pointed $G$-spaces.
\begin{enumerate}
    \item The trivial $G$-action on a (pointed) space $X$ is given by $gx=x$ for all $g\in G,x\in X$. 
    \item Let $X,Y$ be pointed $G$-spaces. There is a $G$-action on the space of pointed continuous maps $\Top_*(X,Y)$ called conjugation, which is given by 
\begin{equation*}
    (g*f)(x)=(g\circ f\circ g^{-1})(x)
\end{equation*}
    for all $x\in X$, $f\in \Top_*(X,Y)$ and $g\in G$.
\end{enumerate}
\end{examples}

One can construct a topological space from a $G$-space by taking fixed points. Taking fixed points is the most important construction on $G$-spaces, and can be described as a functor. 

\begin{definition}\label{def: fixed point of G-space}\index{$(-)^G$}
    Let $H$ be a closed subgroup of $G$. Define \emph{the $H$-fixed point functor} $(-)^H:G\Top_*\rightarrow \Top_*$ by $X\mapsto X^H$, where 
    \begin{equation*}
        X^H=\{x\in X: hx=x ,\space\forall h\in H\}.
    \end{equation*}
\end{definition}

The fixed point functor $(-)^G$ has a left adjoint, which is equipping a topological space with the trivial $G$-action, see \cite[Result 3.1.4]{MM02}. The quotient map $$\varepsilon: G\rightarrow G/N,$$ for a normal subgroup $N\unlhd G$, induces the inflation functor $$\varepsilon^*:G/N \Top_* \rightarrow G\Top_*,$$ which sends a $G/N$-space $X$ to the underlying space $X$ with $G$-action given by $$g x= (\varepsilon (g)) x.$$ The left adjoint to $(-)^G$ is exactly the inflation functor for $N=G$. Therefore, for a $G$-space $X$ and a space $K$
\begin{equation*}
G\Top_* (\varepsilon^* K,X)\cong \Top_*(K,X^G). 
\end{equation*}
With the conjugation group action on $\Top_*(X,Y)$, $G\Top_*(X,Y)$ is exactly the subspace $(\Top_*(X,Y))^G$ of $\Top_*(X,Y)$. Hence, we topologise the set of $G$-equivariant maps $G\Top_* (X,Y)$ as a subspace of the space of continuous maps $\Top_*(X,Y)$ with the compact-open topology, via fixed points. 

The category $G\Top_*$ is a closed symmetric monoidal category. The monoidal product is given by the smash product $X\wedge Y$, equipped with the diagonal $G$-action defined by 
\begin{equation*}\index{$\wedge$}
    g(x\wedge y):=gx\wedge gy.
\end{equation*}

The internal hom is given by the space of continuous maps $\Top_*(X,Y)$ with the conjugation action, and there exists a $G$-equivariant homeomorphism 
\begin{equation*}
    \Top_*(X, \Top_*(Y,Z))\overset{\cong}{\rightarrow} \Top_* (X\wedge Y,Z).
\end{equation*}

Let $G_+$\index{$G_+$} be the group $G$ with a disjoint basepoint, and let $Y$ be an $H$-space. One can define an equivalence relation $\sim$ on the space $G_+\wedge Y$ such that $gh\wedge y\sim g\wedge hy$ for all $g\in G$, $h\in H$ and $y\in Y$. We denote the quotient space $G_+\wedge Y/\sim$ by $G_+\wedge_H Y$\index{$\wedge_G$}, which has a $G$-action given by $g'(g\wedge y)=g'g\wedge y$ for all $g, g'\in G$ and $y\in Y$. Moreover, if $Y$ is a $G$-space, then $G_+\wedge_H Y \cong (G/H)_+\wedge Y$. The subgroup inclusion map $i: H\rightarrow G$ induces the restriction functor $i^*:G \Top_* \rightarrow H\Top_*$, which sends a $G$-space $X$ to the underlying space $X$ with $H$-action given by $h x= (i(h)) x$.

\begin{proposition}[{\cite[Results 3.1.2 and 3.1.3]{MM02}}]\label{prop: mm3.1.2}
Let $H$ be a closed subgroup of $G$. For a $G$-space $X$ and an $H$-space $Y$
\begin{equation*}
    G\Top_* (G_+\wedge_H Y,X)\cong H\Top_*(Y,i^*X).
\end{equation*}
\end{proposition}

In particular, together with the fixed point adjunction discussed previously, this gives \cite[III.1.5]{MM02}. That is, for $G$-spaces $X$ and $Y$
\begin{equation*}
    G\Top_*((G/H)_+\wedge Y,X)\cong \Top_*(Y,X^H).
\end{equation*}

We now give two model structures on the category $G\Top_*$. They are called the coarse and fine model structures respectively, and are related by a Quillen adjunction, where the underlying functor is the identity.

\begin{proposition}\label{prop: coarse ms}\index{$\Top_*[G]$} (Coarse model structure) $G\Top_*$ is a compactly generated, proper model category with weak equivalences and fibrations defined as follows. A map $f:X\rightarrow Y$ is a coarse weak equivalence or coarse fibration of pointed $G$-spaces if after  forgetting the $G$-action the underlying map $f:X\rightarrow Y$ is a weak homotopy equivalence or Serre fibration of pointed spaces. Denote this model structure by $\Top_*[G]$.
The generating cofibrations and acyclic cofibrations are respectively
\begin{align*}
    &\{G_+\wedge i:i\in I_{\Top_*}\}\\
    &\{G_+\wedge j:j\in J_{\Top_*}\}.
\end{align*}
\end{proposition}

\begin{proposition}[{\cite[Theorem 3.1.8]{MM02}}]\label{finemodelstructure}\index{$G\Top_*$}  (Fine model structure) $G\Top_*$ is a compactly generated, proper model category with weak equivalences and fibrations defined as follows. A map $f:X\rightarrow Y$ is a fine weak equivalence or fine fibration of pointed $G$-spaces if $f^H:X^H\rightarrow Y^H$ is a weak homotopy equivalence or Serre fibration of pointed spaces for each closed subgroup $H\leq G$. Denote this model structure by $G\Top_*$.
The generating cofibrations and acyclic cofibrations are respectively 
\begin{align*}\index{$I_{G}$}\index{$J_{G}$}
    &I_{G}=\{G/H_+\wedge i:i\in I_{\Top_*}, H\leq G \textnormal{ a closed subgroup}\}\\
    &J_{G}=\{G/H_+\wedge j:j\in J_{\Top_*}, H\leq G \textnormal{ a closed subgroup}\}.
\end{align*}
\end{proposition}

Moreover, the coarse model structure is monoidal, and so is the fine model structure provided that $G$ is a compact Lie group. 

\begin{remark}\index{$\Top_*[O(n)]$}\index{$\Top_*[O(p,q)]$}\index{$\C\Top_*$}
For the remainder of this document we will use the coarse model structure on $O(n)$-spaces and $O(p,q)$-spaces, which we will denote by $\Top_*[O(n)]$ and $\Top_*[O(p,q)]$ respectively. We will use the fine model structure on $\C$-spaces, which we denote by $\C\Top_*$. The word fine will often be dropped from notation, and a fine weak equivalence/fibration of pointed $\C$-spaces will simply be called a weak equivalence/fibration of $\C$-spaces. We let $[-,-]_{\C}$ denote maps in the homotopy category of $\C\Top_*$. 
\end{remark}

\subsection{$G$-representations}

Orthogonal calculus is indexed on the universe $\mathbb{R}^\infty$. That is, the functors considered in orthogonal calculus take finite dimensional inner product spaces as their input. A natural replacement for these inner product spaces in the $\C$-equivariant setting will be elements from a universe of $\C$-inner product spaces. As such, in this section we recall basic definitions from representation theory for groups, see for example the work of Howe \cite{How22}. Throughout the main text we will always be working over $\mathbb{R}$, so that all vector spaces are real. 

\begin{definition}
    A representation of a group $G$ on an inner product space $V$ is a map $\Phi:G\times V\rightarrow V$ such that for all $g,g_1,g_2\in G$, $u,v\in V$ and $e$ the identity element of $G$
    \begin{itemize}
        \item $\Phi(g):V\rightarrow V$, $v\mapsto \Phi(g,v)$ is linear
        \item $\Phi(e,v)=v$
        \item $\Phi (g_1,\Phi (g_2,v))=\Phi (g_1g_2,v)$
    \end{itemize}

and the inner product on $V$ is $G$-invariant ($\langle \Phi(g,u),\Phi(g,v\rangle = \langle u,v\rangle$). The inner product space $V$ is called a \emph{$G$-representation}. This can also be defined in terms of a map from $G$ to $GL(V)$. 
\end{definition}

\begin{remark}\label{rem: all reps have inner prod}
In general, a $G$-representation need not be an inner product space, however we will assume that all representations have an inner product throughout this thesis. 
\end{remark}

\begin{example}\label{Ex: trivial and sign representation}\index{$\R$}\index{$\Rdelta$}
Let $G=\C=\{e,\sigma\}$. The trivial representation $\R$ and the sign representation $\Rdelta$ of $\C$ have the $\C$-actions defined below. 
\begin{align*}
    \sigma(x)&=x\quad (x\in\R)\\
    \sigma(y)&=-y\quad (y\in\Rdelta)
\end{align*}
These are the indecomposable $\C$-representations, in that they cannot be decomposed as direct sums of subrepresentations, see Definition \ref{def: subrep}.  
\end{example}

\begin{example}\label{ex: regular c2 rep}\index{$\RC$}
The regular representation of $\C=\{e,\sigma\}$ is defined as the following vector space\begin{equation*}
    \RC=\{\lambda_1 \underline{e}+\lambda_2 \underline{\sigma}:\lambda_1,\lambda_2\in\R\}
\end{equation*}
with basis elements $\underline{e}, \underline{\sigma}$. One can decompose 
\begin{equation*}\RC= \R\langle\underline{e}+\underline{\sigma}\rangle \oplus \R\langle\underline{e}-\underline{\sigma}\rangle.
\end{equation*}
There is a $\C$-equivariant isomorphism $\R\rightarrow \R\langle\underline{e}+\underline{\sigma}\rangle$ defined by $x\mapsto x(\underline{e}+\underline{\sigma})$, and a $\C$-equivariant isomorphism $\Rdelta\rightarrow \R\langle\underline{e}-\underline{\sigma}\rangle$ defined by $y\mapsto y(\underline{e}-\underline{\sigma})$. Therefore, $\R[\C]$ is $\C$-isomorphic to $\R\oplus\Rdelta$, which is given the diagonal $\C$-action, since it is a direct sum of $\C$-representations (see below). 
\end{example}

\begin{example}
    The direct sum of two $G$-representations is a $G$-representation with the diagonal $G$-action. Similarly, the tensor product of two $G$-representations is a $G$-representation, whose underlying vector space is the usual tensor product of vector spaces and $G$-action is the diagonal $G$-action, e.g. $\R\otimes\Rdelta=\Rdelta$ and $\Rdelta\otimes\Rdelta=\R$. 

\end{example}

Much like with $G$-spaces, we can also take fixed points of $G$-representations (see Definition \ref{def: fixed point of G-space}). 

\begin{definition}\index{$V^{\C}$}
    Let $H$ be a closed subgroup of $G$, and $V$ be a $G$-representation. Define the $H$-fixed points of $V$ by 
    \begin{equation*}
        V^H=\{ v\in V: hv=v, \forall h\in H\}.
    \end{equation*}
\end{definition}
\begin{remark}\label{rem: ^h^perp notation}\index{$(V^{\C})^\perp$}
We will use the notation $(V^H)^\perp$ to denote the orthogonal complement of $V^H$ in $V$. For example, if $V$ is the regular representation of $\C$ (see Example \ref{ex: regular c2 rep}), then $V^{\C}=\R$ and $(V^{\C})^{\perp}=\Rdelta$.
\end{remark}

One can talk about subspaces of a vector space. This notion extends when the vector space is given a group action. 
\begin{definition}\label{def: subrep}
    Let $(V,\Phi)$ be a $G$-representation. A \emph{subrepresentation} of $(V,\Phi)$ is a linear subspace $W\subseteq V$ that is preserved by the $G$-action. That is, $\Phi|_{G\times W}$ defines a $G$-action on W. 
\end{definition}

Lastly, we recall the notion of a $G$-universe. Choosing a $G$-universe allows one to specify which $G$-representations are to be considered. 

\begin{definition}
    A \emph{$G$-universe} $U$ is a countable direct sum of $G$-representations such that $U$ contains:
    \begin{itemize}
        \item the trivial $G$-representation,
        \item countably many copies of each of its subrepresentations.
    \end{itemize}
    A $G$-universe is called complete if is contains every irreducible representation of $G$.
\end{definition}

\begin{example}
    For $G=\C$. The countable direct sum of the regular representation $\bigoplus\limits_{i=1}^{\infty} \RC$ forms a complete $\C$-universe. 
\end{example}

\subsection{$G$-vector bundles}

Vector bundles are used to define the $n$-th jet categories $\mathcal{J}_n$ in orthogonal calculus, see Definition \ref{def: gamma n bundle}. The analogous $(p,q)$-th jet categories $\Jpq$ in the $\C$-calculus are defined using $\C$-vector bundles, see Definition \ref{def: gamma pq bundle}. 

We now briefly introduce for later convenience the definition of a $G$-vector bundle. This is an equivariant vector bundle with respect to an action of a compact Lie group $G$. These bundles are discussed in more detail by May in \cite[Section 14.1]{May96}. Again, we will be working over $\R$. 

\begin{definition}
 Let $G$ be a compact Lie group and $X$ be a $G$-space. A \emph{$G$-vector bundle} over $X$ is a vector bundle map $\rho:E\rightarrow X$ and a $G$-action on the total space $E$ such that: 
 \begin{itemize}
    \item $\rho$ is a $G$-equivariant map,
    \item if $g\in G$, then $\rho^{-1}(x)\rightarrow \rho^{-1}(gx)$ is a linear map for all $x\in X$.
\end{itemize}
\end{definition}

That is, $G$ acts linearly on the fibres.

\begin{example} \cite[Section 14.2]{May96}
A $G$-vector bundle over a point is precisely the information of a $G$-representation. 
\end{example}

An important example of a $G$-vector bundle is the bundle $E(V,V' )$ used by Mandell and May in \cite{MM02} to describe $G$-spectra. These are defined analogously to the bundles $\gamma_1(U,V)$ used in orthogonal calculus, see Definition \ref{def: gamma n bundle}. 

\begin{example} \cite[Definition 2.4.1]{MM02}
Let $V,V'$ be elements of some $G$-universe of $G$-representations. Let $\mathcal{J} (V,V')$ be the $G$-space of linear isometries from $V$ to $V'$ with $G$-action given by conjugation. Let $E(V,V')$ be the subbundle of the product $G$-bundle $\mathcal{J} (V,V')\times V'$ over $\mathcal{J} (V,V')$ defined by 
\begin{equation*}
        E(V,V')= \{ (f,x): f\in\mathcal{J} (V,V'), x\in f(V)^\perp \}
\end{equation*}
where $f(V)^\perp$ denotes the orthogonal complement of the image of $f$. 
The group $G$ acts on $E(V,V')$ via the diagonal action. 
\end{example}

We now recall the definition for the Thom space of a vector bundle. 

\begin{definition} \cite[Section 23.5]{May99}\index{$T(-)$}
    Let $f:E\rightarrow B$ be a real vector bundle. Let $D(E)$ and $S(E)$ denote the unit disk and sphere fibre bundles of $f$ respectively, with respect to any choice of metric. The \emph{Thom space} of the vector bundle $f$ is the pointed topological quotient space $T(E)=D(E)/S(E)$. The basepoint of $T(E)$ is the image of $S(E)$ under the quotient.  
\end{definition}

If the base space $B$ is compact, $T(E)$ is homeomorphic to the one point compactification of $E$. The point at infinity is exactly the basepoint given by $S(E)$ under the quotient.  

\begin{remark}
In particular, if $f:E\rightarrow B$ is a $G$-vector bundle, then $T(E)$ inherits a $G$-action from the $G$-action on $E$, as $G$ acts through isometries, and $G$ acts as the identity on the point at infinity.       
\end{remark}

\subsection{The equivariant Freudenthal suspension theorem}\label{sec: EFST}

Many of the key results of orthogonal calculus are proven using connectivity arguments. As such, the Freudenthal suspension theorem is heavily relied upon. Analogously, the $\C$-calculus depends on the equivariant Freudenthal suspension theorem. Since the equivariant Freudenthal suspension theorem does not follow from the non-equivariant theorem in an obvious way, we state the theorem here. The content of this section can be found in \cite[Chapters 9 and 11]{May96}. %and also Greenlees May - ESHT

Given a $G$-representation $V$, one can define disks and spheres (since $V$ has an inner product, see Remark \ref{rem: all reps have inner prod}). Let $S(V)$ \index{$S(V)$}denote the unit sphere of $V$ and $D(V)$ \index{$D(V)$}denote the unit disk of $V$. Let $S^V$ \index{$S^V$}be the one point compactification of $V$, with basepoint given by the point at infinity. The space $S^V$ inherits a $G$-action from the action on $V$, and the point at infinity has trivial $G$-action. Alternatively, $S^V$ is homeomorphic to the quotient $D(V)/S(V)$, where the image of $S(V)$ under the quotient corresponds to the point at infinity. In particular, when $V$ is the trivial $n$-dimensional $G$-representation this gives the usual $n$-sphere $S^n$. 

We now define two functors which give a notion of equivariant suspension and loops. They are analogous to the standard $\Sigma^n X=S^n\wedge X$\index{$\Sigma^n$} and $\Omega^n X=\Top_*(S^n,X)$\index{$\Omega^n$} functors for pointed spaces. 

\begin{definition}\index{$\Sigma^V$}\index{$\Omega^V$}
Let $X$ be a $G$-space and $V$ be a $G$-representation, then the \emph{$V$-th suspension functor} and \emph{$V$-th loop space functor} are defined by 
\begin{align*}
    \Sigma^V X&=S^V \wedge X,\\
    \Omega^V X&=\Top_*(S^V,X).
\end{align*}
As expected from Section \ref{sec:Gspace}, $G$ acts on $\Sigma^V X$ diagonally and on $\Omega^V X$ by conjugation.    
\end{definition}
 
As with the standard suspension and loops functors, these functors form an adjunction on $G$-spaces
\begin{equation*}
    \Sigma^V:G\Top_*\rightleftarrows G\Top_*:\Omega^V.
\end{equation*}

Now we turn to discussing the connectivity of a $G$-space. 

Given a $G$-space $X$, one can take homotopy groups of the space $X^H$ for each closed subgroup $H$ of $G$. Therefore, the connectivity of $X$ can be described as a dimension function, see \cite[Definition 11.2.1]{May96}. 

\begin{definition}\label{def: dim funtion}
    A \emph{dimension function} $v$ is a function from the set of conjugacy classes of closed subgroups of $G$ to the integers. 
    A $G$-space $X$ is \emph{$v$-connected}, if each $X^H$ is a $v(H)$-connected space. 
\end{definition}

\begin{examples}\label{ex: dim function examples} \cite[Definition 11.2.1]{May96}
The following are examples of dimension functions, where $G$ is a group and $H$ is a closed subgroup of $G$:
\begin{enumerate}
        \item For $n\in\ZZ$, $n^*:H\mapsto n$.
        \item For $V$ a $G$-representation, $|V^*| :H\mapsto \Dim(V^H)$.
        \item For $X$ a $G$-space, $c^*(X):H\mapsto c(X^H)$, where $c(-)$ is the connectivity of the space. If $X^H$ is not path connected and non-empty, then set $c^H(X)$ equal to $-1$.
\end{enumerate}
\end{examples}

One can also use dimension functions to describe the connectivity of a map between $G$-spaces. 

\begin{definition}
    Let $X$ and $Y$ be $G$-spaces. A map $f:X\rightarrow Y$ is \emph{$v$-connected}, if each $f^H:X^H\rightarrow Y^H$ is $v(H)$-connected. 
\end{definition}

That is, for each closed subgroup $H$, $f^H$ is injective on homotopy groups of degree less than $v(H)$ and surjective on homotopy groups of degree less than or equal to $v(H)$. 

We can now state the equivariant Freudenthal suspension theorem, a proof of which by May can be found in \cite[Section 11.5]{May96}.

\begin{theorem}[{\cite[Theorem 9.1.4]{May96}}] (\textit{Equivariant Freudenthal suspension theorem})
Let $Y$ be a pointed $G$-space, and $V$ be a $G$-representation. Then the unit map $\eta: Y\mapsto \Omega^V\Sigma^V Y$, of the adjunction $\Sigma^V:G\Top_*\rightleftarrows G\Top_*:\Omega^V$, is $v$-connected for any $v$ satisfying
\begin{enumerate}
    \item $v(H)\leq 2c^H (Y) +1, \quad \forall$ subgroups $H$ such that $V^H\neq 0$.
    \item $v(H)\leq c^K(Y), \quad \forall$ subgroup pairs $K< H$ such that $V^K\neq V^H$.
\end{enumerate}
Therefore, the suspension map
\begin{equation*}
    \Sigma^V:[X,Y]_G\rightarrow [\Sigma^V X,\Sigma^V Y]_G
\end{equation*}
is surjective if $\Dim(X^H)\leq v(H)$, and bijective if $\Dim(X^H)\leq v(H)-1$ for all closed subgroups $H$ of $G$, where $[-,-]_G$ denotes the set of homotopy classes of based $G$-maps. 
\end{theorem}

Note that when $G=e$, this is exactly the standard non-equivariant Freudenthal suspension theorem, since condition 2 is empty. 

For $G=\C$, the theorem becomes much easier to interpret, since there are only a maximum of three conditions on the dimension function $v$. 

\begin{theorem}($\C$-equivariant Freudenthal suspension theorem)
Let $Y$ be a pointed $\C$-space and $V$ be a $\C$-representation. The map $\eta: Y\mapsto \Omega^V\Sigma^V Y$ is $v$-connected for any $v$ satisfying
\begin{enumerate}
    \item $v(e)\leq 2c^e (Y) +1$, if $V\neq 0$.
    \item $v(\C) \leq 2c^{\C} (Y) +1$, if $V^{\C}\neq 0$.
    \item $v(\C)\leq c^{e}(Y)$, if $V\neq V^{\C}$.
\end{enumerate}
\end{theorem}

Given the non-equivariant Freudenthal suspension theorem, see \cite[Theorem 1.1.10]{BR20}, one would expect the $\C$-equivariant suspension theorem to look like: 

The map $\eta: Y\mapsto \Omega^V\Sigma^V Y$ is $v$-connected for any $v$ satisfying
\begin{enumerate}
    \item $v(e)\leq 2c^e (Y) +1$.
    \item $v(\C) \leq 2c^{\C} (Y) +1$.
\end{enumerate}

However, there are counter examples that demonstrate why this is false. These are discussed by May in \cite[Section 11.2]{May96}, and we just state one of them here. 

If $V$ is the sign representation $\Rdelta$ of $\C$ and $n\geq 3$, then the above `expected theorem' says that the map 
\begin{equation*}
    \Sigma^V:[S^n, S^n]_G\rightarrow [\Sigma^V S^n, \Sigma^V S^n]_G
\end{equation*}
is an isomorphism. However, $[S^n, S^n]_G=\mathbb{Z}$ and $[\Sigma^V S^n, \Sigma^V S^n]_G=\mathbb{Z}^2$. Hence, the map $\Sigma^V$ is not surjective. 

In this way, the third condition of the $\C$-equivariant Freudenthal suspension theorem can be thought of as the extra condition needed to make the `expected theorem' work.

\section{Orthogonal calculus}\label{sec: orth calc}

The aim of the main text is to extend the theory of orthogonal calculus to the $\C$-equivariant context. To do this, we will work through the same constructions of the underlying calculus, while adjusting the categories and results to account for the newly introduced group actions. As such, in this section we provide an overview of these constructions for comparison. The main details of these constructions were originally by Weiss \cite{Wei95}, however towards the classification theorems we choose to follow the model categorical approach of Barnes and Oman \cite{BO13}. Since the proofs of many results in the $\C$-equivariant setting mimic those in the underlying calculus, we chose to omit them here. 

Many objects studied in algebraic topology can be realised as functors. Using functor calculus, one can approximate a given functor by a sequence of functors with `nice' properties. In doing so, we make analysis of the original functor much more simple, by 'breaking' it into smaller, more manageable parts. The resulting sequence of functors is a similar concept to that of a Postnikov tower, in which we approximate a CW-complex $X$ by a tower of fibrations, with fibres Eilenberg-MacLane spaces. Language from differential calculus is adopted, since this is comparable to the Taylor series for a function, in which the function is approximated by polynomial functions. Hence, we can justify labelling this a form of calculus. 

Orthogonal homotopy calculus is the branch of functor calculus involving the study of functors from the category of finite dimensional real vector spaces to the category of pointed topological spaces. Given an input functor $F$, the end result is a Taylor tower of approximations, built from polynomial functors $T_n F$ and spectra $\Theta_F^n$ with an action of $O(n)$, for $n\geq 1$. 
% https://q.uiver.app/#q=WzAsNyxbMiwzLCJGKFxcUl5cXGluZnR5KSJdLFsyLDIsIlRfMUYoVikiXSxbMiwxLCJUXzJGKFYpIl0sWzIsMCwiXFx2ZG90cyJdLFszLDIsIlxcT21lZ2FeXFxpbmZ0eSBcXGxlZnRbXFxsZWZ0KFNeezFWfSBcXHdlZGdlIFxcVGhldGFfRl4xXFxyaWdodClfe2hPKDEpfVxccmlnaHRdIl0sWzAsMywiRihWKSJdLFszLDEsIlxcT21lZ2FeXFxpbmZ0eSBcXGxlZnRbXFxsZWZ0KFNeezJWfSBcXHdlZGdlIFxcVGhldGFfRl4yXFxyaWdodClfe2hPKDIpfVxccmlnaHRdIl0sWzMsMl0sWzIsMV0sWzEsMF0sWzQsMV0sWzUsMF0sWzYsMl0sWzUsMSwiIiwwLHsiY3VydmUiOi0xfV0sWzUsMiwiIiwwLHsiY3VydmUiOi0yfV1d
\[\begin{tikzcd}
	&& \vdots \\
	&& {T_2F(V)} & {\Omega^\infty \left[\left(S^{2V} \wedge \Theta_F^2\right)_{hO(2)}\right]} \\
	&& {T_1F(V)} & {\Omega^\infty \left[\left(S^{1V} \wedge \Theta_F^1\right)_{hO(1)}\right]} \\
	{F(V)} && {F(\R^\infty)}
	\arrow[from=1-3, to=2-3]
	\arrow[from=2-3, to=3-3]
	\arrow[from=3-3, to=4-3]
	\arrow[from=3-4, to=3-3]
	\arrow[from=4-1, to=4-3]
	\arrow[from=2-4, to=2-3]
	\arrow[curve={height=-6pt}, from=4-1, to=3-3]
	\arrow[curve={height=-12pt}, from=4-1, to=2-3]
\end{tikzcd}\]
The fibres of the tower are homogeneous functors. The main result of the calculus is known as the classification theorem. It allows us to characterise $n$-homogeneous functors as functors that are completely determined by orthogonal spectra with an action of $O(n)$. Therefore, giving a means to obtain unstable information from stable data. It can be stated as a zig-zag of equivalences between the stable model structure on orthogonal spectra with an action of $O(n)$ and the $n$-homogeneous model structure. 
% https://q.uiver.app/#q=WzAsMyxbMCwwLCJuXFx0ZXh0ey1ob21vZy19XFxtYXRoY2Fse0V9XzAiXSxbNCwwLCJTcF5cXG1hdGhjYWx7T31bTyhuKV0iXSxbMiwwLCJTcF5cXG1hdGhjYWx7T31bTyhuKV0iXSxbMCwyLCJcXHRleHR7aW5kfV8wXm5cXHZhcmVwc2lsb25eKiIsMix7Im9mZnNldCI6Mn1dLFsyLDAsIlxcdGV4dHtyZXN9XzBebi9PKG4pIiwyLHsib2Zmc2V0IjoyfV0sWzIsMSwiKFxcYWxwaGFfbilfISIsMCx7Im9mZnNldCI6LTJ9XSxbMSwyLCJcXGFscGhhX25eKiIsMCx7Im9mZnNldCI6LTJ9XV0=
\[\begin{tikzcd}
	{n\homog\mathcal{E}_0} && {O(n)\mathcal{E}_n} && {Sp^O[O(n)]}
	\arrow["{\ind_0^n\varepsilon^*}"', shift right=2, from=1-1, to=1-3]
	\arrow["{\res_0^n/O(n)}"', shift right=2, from=1-3, to=1-1]
	\arrow["{(\alpha_n)_!}", shift left=2, from=1-3, to=1-5]
	\arrow["{\alpha_n^*}", shift left=2, from=1-5, to=1-3]
\end{tikzcd}\]

\subsection{The functor categories}

Let $\mathcal{J}$ \index{$\mathcal{J}$}denote the category of finite dimensional real subspaces of $\R^\infty$ with an inner product and linear isometries. We now define a sequence of $\Top_*$-enriched categories $\mathcal{J}_n$ for $n\geqslant 0$, which we will use to define intermediate categories $O(n)\mathcal{E}_n$ of enriched functors. We begin by defining the following vector bundle $\gamma_n(U,V)$, that will be used to build $\mathcal{J}_n$. 

\begin{definition}\label{def: gamma n bundle}\index{$\gamma_n(U,V)$}
For $U,V\in \mathcal{J}$, define the \emph{$n$-th complement bundle} $\gamma_n(U,V)$ to be the vector bundle on $\mathcal{J}(U,V)$, whose total space is given by
\begin{equation*}
    \gamma_n(U,V)=\{(f,x):f\in \mathcal{J}(U,V), x\in \mathbb{R}^n\otimes f(U)^\perp\}
\end{equation*}
where $f(U)^\perp$ denotes the orthogonal complement of the image of $f$.
\end{definition}

This is a sub-bundle of the product bundle whose total space is $\mathcal{J}(U,V)\times (\mathbb{R}^n \otimes V)$.

Denote the Thom space of the bundle $\gamma_n(U,V)$ by $\mathcal{J}_n(U,V)$\index{$\mathcal{J}_n(U,V)$}. This Thom space is the one point compactification of $\gamma_n(U,V)$, since $\mathcal{J}(U,V)$ is compact. Each $\mathcal{J}_n(U,V)$ is a pointed space. 

\begin{remark}
In particular, $\mathcal{J}_0(U,V)=T(\gamma_0(U,V))$ is equal to $\mathcal{J}(U,V)_+$.
\end{remark}

There is a composition rule induced by the vector bundle map 
\begin{align*}
\gamma_n(V,W)\times \gamma_n(U,V)&\rightarrow \gamma_n(U,W)\\
((f,x),(g,y))&\mapsto (fg,x+(\id\otimes f)(y)).
\end{align*}Passing to Thom spaces then yields the composition law 
\begin{equation*}
    \mathcal{J}_n(V,W)\wedge \mathcal{J}_n(U,V)\rightarrow \mathcal{J}_n(U,W).
 \end{equation*}

One can check that this composition is a continuous map. Moreover, the composition maps are unital and associative. 

We can now define the categories $\mathcal{J}_n$.
\begin{definition}\label{def: nth jet cat}\index{$\mathcal{J}_n$}
For each $n\geqslant 0$, let the \emph{$n$-th jet category} $\mathcal{J}_n$ be the $\Top_*[O(n)]$-enriched category whose objects are finite dimensional real inner product spaces, and whose morphism spaces are given by $\mathcal{J}_n(U,V)=T(\gamma_n(U,V))$. Composition in $\mathcal{J}_n$ is defined as above.  
\end{definition}

The action of $O(n)$ on the space of morphisms $\mathcal{J}_n(U,V)$ is induced by the $O(n)$-action on $\R^n$.

The following theorem by Weiss demonstrates that it is possible to construct $\mathcal{J}_{n+1}$ from $\mathcal{J}_n$. Note that there is a functor $\mathcal{J}_n\rightarrow\mathcal{J}_{n+1}$ induced by the standard inclusion $\R^n\rightarrow \R^{n+1}$ as the first $n$ coordinates (see Section \ref{sec: orth calc derivatives}). 

\begin{proposition}[{\cite[Theorem 1.2]{Wei95}}]\label{prop: wei 1.2}
For all $U,V$ in $\mathcal{J}_0$ and for all $n\geqslant 0$,
\begin{equation*}
    \mathcal{J}_n(\mathbb{R}\oplus U,V)\wedge S^n\rightarrow \mathcal{J}_n(U,V)\rightarrow \mathcal{J}_{n+1}(U,V).
\end{equation*}
is a homotopy cofibre sequence. 
\end{proposition}

We are now ready to define the intermediate functor categories $O(n)\mathcal{E}_n$. 

\begin{definition}\index{$\mathcal{E}_n$}\index{$O(n)\mathcal{E}_n$}
Define $\mathcal{E}_n$ to be the category of $\Top_*$-enriched functors $\mathcal{J}_n\rightarrow \Top_*$ and natural transformations. 
Define the \emph{$n$-th intermediate category} $O(n)\mathcal{E}_n$ to be the category of $\Top_*[O(n)]$-enriched functors $\mathcal{J}_n\rightarrow \Top_*[O(n)]$ and $O(n)$-equivariant natural transformations. 
\end{definition}

\begin{remark}\label{rem: ontop has coarse ms and E1 is spectra}
Note that $\Top_*[O(n)]$ is equipped with the coarse model structure, see Proposition \ref{prop: coarse ms}. The category $\mathcal{E}_1$ is equivalent to the category of orthogonal spectra $Sp^O$ by definition. 
\end{remark}

The input functors for orthogonal calculus are those objects from the category $\mathcal{E}_0$. Some examples of such functors are listed below. 

\begin{examples}
$\quad$
\begin{itemize}
    \item $\mathcal{J}_n(U,-):V\mapsto \mathcal{J}_n(U,V)_+$
    \item $O(-):V\mapsto O(V)_+ $ 
    \item $\Top(-):V\mapsto \Top(V)_+$
\end{itemize}
where $O(V)$ is the space of linear isometries on $V$, and $\Top(V)$ is the space of homeomorphisms on $V$. Other examples can be produced by replacing $O(V)$ and $\Top(V)$ by their classifying spaces $BO(V)_+$ and $B\Top(V)_+$.
\end{examples}

\begin{remark}
We denote the set of natural transformations between $E,F\in \mathcal{E}_n$ by $\Nat_n(E,F)$\index{$\Nat_n(E,F)$}. There is a natural topology that one can define on $\Nat_n(E,F)$, by expressing $\Nat_n(E,F)$ as the enriched end below, see \cite[Section 2.2]{Kel05}. 

\begin{equation*}
    \Nat_n(E,F)= \int_{V\in\mathcal{J}_n} \Top_*(E(V),F(V))\subseteq \prod_{V\in\mathcal{J}_n}\Top_*(E(V),F(V)) 
\end{equation*}
Hence, we can equip $\Nat_n(E,F)$ with the appropriate subspace topology of the product space. 

In a similar way, we can describe a functor $E\in \mathcal{E}_n$ in terms of an enriched coend, by the enriched Yoneda lemma (see for example \cite[Section 3.10]{Kel05}). 
\begin{equation*}
    \int^{W\in\mathcal{J}_n} E(W) \wedge \mathcal{J}_n(W,-)\cong E.
\end{equation*}
\end{remark}

There exists a projective model structure on $\mathcal{E}_0$, by \cite[Theorem 6.5]{MMSS01}, in which weak equivalences and fibrations are defined objectwise. 

\begin{proposition}[{\cite[Lemma 6.1]{BO13}}]\label{prop: E0 proj model structure}\index{$\mathcal{E}_0$}
There exists a proper, cellular model structure on $\mathcal{E}_0$, where $f:E\rightarrow F$ is a weak equivalence (resp. fibration), if $f(V):E(V)\rightarrow F(V)$ is a weak homotopy equivalence (resp. Serre fibration) for all $V\in\mathcal{J}_0$. We call this model structure the projective model structure on $\mathcal{E}_0$ and denote it by $\mathcal{E}_0$. It is cofibrantly generated by the following sets of generating cofibrations and generating acyclic cofibrations respectively
\begin{align*}
    &\{\mathcal{J}_0 (V,-)\wedge i : i\in I_{\Top_*}\}\\
    &\{\mathcal{J}_0 (V,-)\wedge j : j\in J_{\Top_*}\},
\end{align*}
where $V\in \mathcal{J}_0$, and $I_{\Top_*}$, $J_{\Top_*}$ are the generating cofibrations and acyclic cofibrations of the standard Quillen model structure on $\Top_*$ respectively (see Proposition \ref{prop: Quillen ms}). 
\end{proposition}

\subsection{The intermediate categories as spectra}\label{sec: En as spectra}

The intermediate category $O(n)\mathcal{E}_n$ is Quillen equivalent to the category of orthogonal spectra with an action of $O(n)$, when it is equipped with the $n$-stable model structure. This Quillen equivalence forms half of the zig-zag of equivalences that gives the classification of $n$-homogeneous functors, see \cite[Theorem 10.1]{BO13}. In this section, we summarise the construction this Quillen equivalence. 

First we give the $n$-stable model structure on the intermediate category $O(n)\mathcal{E}_n$. It is a modification of the stable model structure on orthogonal spectra, see Barnes and Roitzheim \cite[Section 2.3]{BR20}. We begin by first defining homotopy groups on objects of $O(n)\mathcal{E}_n$. These homotopy groups detect the weak equivalences of the $n$-stable model structure.  

\begin{definition}\index{$n\pi_k (X)$}
Define the \emph{$n$-homotopy groups} of $X\in O(n)\mathcal{E}_n$ by 
\begin{equation*}
n\pi_k (X) =\colim_l \pi_{nl+k} X(\R^l),
\end{equation*} 
where $k\in\mathbb{Z}$ and the colimit runs over the diagram
% https://q.uiver.app/#q=WzAsNSxbMSwwLCJcXHBpX3tubCtrfVgoXFxSXmwpIl0sWzIsMCwiXFxwaV97bmwraytufShYKFxcUl5sKVxcd2VkZ2UgU15uKSJdLFszLDAsIlxccGlfe25sK2srbn1YKFxcUl57bCsxfSkiXSxbMCwwLCJcXGRvdHMiXSxbNCwwLCJcXGRvdHMiXSxbMCwxLCIoLSlcXHdlZGdlIFNebCJdLFsxLDIsIlxcc2lnbWFfWCJdLFszLDAsIlxcc2lnbWFfWCJdLFsyLDQsIigtKVxcd2VkZ2UgU15sIl1d
\[\begin{tikzcd}
	\dots & {\pi_{nl+k}X(\R^l)} & {\pi_{nl+k+n}(X(\R^l)\wedge S^n)} & {\pi_{nl+k+n}X(\R^{l+1})} & \dots
	\arrow["{\sigma_X}", from=1-1, to=1-2]
	\arrow["{(-)\wedge S^l}", from=1-2, to=1-3]
	\arrow["{\sigma_X}", from=1-3, to=1-4]
	\arrow["{(-)\wedge S^l}", from=1-4, to=1-5]
\end{tikzcd}\]
in which $\sigma_X$ is the structure map of $X$ (see Definition \ref{def: structure maps in orth}).

Define a map $f:X\rightarrow Y$ in $O(n)\mathcal{E}_n$ to be an $n\pi_*$-equivalence if the induced map on $n$-homotopy groups $n\pi_k f: n\pi_k X\rightarrow n\pi_k Y$ is an isomorphism for all $k$. 
\end{definition}

Now we identify the fibrant objects of the $n$-stable model structure. These are the $n\Omega$-spectra, see \cite[Definition 7.9]{BO13}.

\begin{definition}\label{def: structure maps in orth}
An object $X$ of $O(n)\mathcal{E}_n$ has structure maps $$\sigma_X:S^{nV}\wedge X(W)\rightarrow X(W\oplus V).$$The object $X$ is called an \emph{$n\Omega$-spectrum} if its adjoint structure maps $$\tilde{\sigma}_X:X(W)\rightarrow \Omega^{nV} X(W\oplus V)$$ are weak homotopy equivalences, for all $V,W\in \mathcal{J}_0$. 
\end{definition}

Recall from Remark \ref{rem: ontop has coarse ms and E1 is spectra} that $\Top_*[O(n)]$ is equipped with the coarse model structure.

\begin{proposition}[{\cite[Proposition 7.14]{BO13}}]\label{prop: n stable ms}\index{$O(n)\mathcal{E}_n^s$}
There is a cofibrantly generated, proper, cellular model structure on $O(n)\mathcal{E}_n$ called the $n$-stable model structure. The weak equivalences are the $n\pi_*$-equivalences. The fibrations are the maps $f:X\rightarrow Y$ such that $f(V)$ is a Serre fibration for each $V\in\mathcal{J}_0$ and such that the diagram 
% https://q.uiver.app/#q=WzAsNCxbMCwwLCJYKFYpIl0sWzAsMSwiWShWKSJdLFsxLDAsIlxcT21lZ2Fee25XfVgoVlxcb3BsdXMgVykiXSxbMSwxLCJcXE9tZWdhXntuV31ZKFZcXG9wbHVzIFcpIl0sWzAsMl0sWzAsMV0sWzEsM10sWzIsM11d
\[\begin{tikzcd}
	{X(V)} & {\Omega^{nW}X(V\oplus W)} \\
	{Y(V)} & {\Omega^{nW}Y(V\oplus W)}
	\arrow[from=1-1, to=1-2]
	\arrow[from=1-1, to=2-1]
	\arrow[from=2-1, to=2-2]
	\arrow[from=1-2, to=2-2]
\end{tikzcd}\]
is a homotopy pullback for all $V,W\in\mathcal{J}_0$. The fibrant objects are the $n\Omega$-spectra. Denote this model category by $O(n)\mathcal{E}_n^s$.
\end{proposition}

This model structure is constructed as a left Bousfield localisation of the levelwise model structure $O(n)\mathcal{E}_n^l$ (\cite[Theorem 6.5]{MMSS01}), where fibrations and weak equivalences are defined levelwise. 

To relate the intermediate categories to orthogonal spectra with an action of $O(n)$, $Sp^O[O(n)]$\index{$Sp^O[O(n)]$}, we define explicitly a functor $\alpha_n^*$. The construction of this functor is detailed by Barnes and Oman in \cite[Section 8]{BO13}, and it forms a Quillen equivalence between these categories. Recall that the category of orthogonal spectra $Sp^O$ is equivalent to the category $\mathcal{E}_1$ by definition (see Remark \ref{rem: ontop has coarse ms and E1 is spectra}). Therefore, the categories $Sp^O[O(n)]$ and $\mathcal{E}_1[O(n)]$ are also equivalent. The category $Sp^O[O(n)]$ has a stable model structure, where weak equivalences and fibrations are defined by the underlying stable model structure on $Sp^O$.

\begin{definition}\index{$\alpha_n^*$}
Define a $\Top_*$-enriched functor $\alpha_{n}:\mathcal{J}_n\rightarrow\mathcal{J}_1$ by 
\begin{equation*}
    U\mapsto \mathbb{R}^{n}\otimes U := nU
\end{equation*}
on objects, and 
\begin{align*}
\mathcal{J}_n(U,V)&\rightarrow \mathcal{J}_1(nU,nV)\\
    (f,x)&\mapsto(\mathbb{R}^{n}\otimes f, x)
\end{align*}
on morphism spaces. 
\end{definition}

This induces a well defined functor 
% https://q.uiver.app/#q=WzAsNCxbMCwwLCJTcF5cXG1hdGhjYWx7T31bTyhuKV0iXSxbMSwwLCJPKG4pXFxtYXRoY2Fse0V9X24iXSxbMCwxLCJcXHRleHR7RnVufV97XFx0ZXh0e1RvcH1fKn0oXFxtYXRoY2Fse0p9XzEsTyhuKVxcdGV4dHtUb3B9XyopIl0sWzEsMSwiXFx0ZXh0e0Z1bn1fe08obilcXHRleHR7VG9wfV8qfShcXG1hdGhjYWx7Sn1fbixPKG4pXFx0ZXh0e1RvcH1fKikiXSxbMCwxLCJcXGFscGhhX25eKiJdLFswLDIsIiIsMix7ImxldmVsIjoyLCJzdHlsZSI6eyJoZWFkIjp7Im5hbWUiOiJub25lIn19fV0sWzEsMywiIiwwLHsibGV2ZWwiOjIsInN0eWxlIjp7ImhlYWQiOnsibmFtZSI6Im5vbmUifX19XV0=
\[\begin{tikzcd}
	{Sp^O[O(n)]} & {O(n)\mathcal{E}_n} \\
	{\text{Fun}_{\Top_*}(\mathcal{J}_1,\Top_*[O(n)])} & {\text{Fun}_{\Top_*[O(n)]}(\mathcal{J}_n,\Top_*[O(n)])}
	\arrow["{\alpha_n^*}", from=1-1, to=1-2]
	\arrow[Rightarrow, no head, from=1-1, to=2-1]
	\arrow[Rightarrow, no head, from=1-2, to=2-2]
\end{tikzcd}\]
which is precomposition with $\alpha_{n}$ on objects, where $\text{Fun}_C(-,-)$ denotes the category of $C$-enriched functors.

Given $\Theta \in Sp^O[O(n)]$, there are two $O(n)$-actions on the image $\alpha_n^*\Theta(V)=\Theta(nV)$. The first action is the internal action on $\Theta(nV)$ induced by the action on $nV$, which is denoted by $ \Theta(g \otimes V)$ for $g\in O(n)$. The second is the external action from $\Theta(nV)$ being an $O(n)$-space, which is denoted by $g_{\Theta(nV)}$ for $g\in O(n)$. These two actions commute by construction. Since we want the functor $\alpha^*\Theta$ to be $\Top_*[O(n)]$-enriched, we define the action of $O(n)$ on $\Theta(nV)$ as the composition
\begin{equation*}
 \Theta(g\otimes V)\circ g_{\Theta(nV)}.
\end{equation*}

The left Kan extension along $\alpha_{n}$ forms a left adjoint to $\alpha_{n}^*$. Using the notation $(\alpha_{n})_!$ to denote taking the left Kan extension along $\alpha_{n}$, this can be described by the $\Top_*[O(n)]$-enriched coend 
\begin{equation*}
    ((\alpha_{n})_! (\Theta))(V)=\int\limits^{U\in \mathcal{J}_n} \mathcal{J}_1 (nU,V)\wedge \Theta(U).
\end{equation*}

\begin{theorem}[{\cite[Proposition 8.3]{BO13}}]
The adjoint pair 
\begin{equation*}
    (\alpha_{n})_{!}: O(n)\mathcal{E}_n^s \rightleftarrows Sp^O[O(n)] : \alpha_{n}^*
\end{equation*}
is a Quillen equivalence, where $Sp^O[O(n)]$ is equipped with the stable model structure. 
\end{theorem}

\subsection{Derivatives}\label{sec: orth calc derivatives}

As suggested by the name orthogonal calculus, one can expect to define a notion of differentiation for functors. The derivatives of input functors play a key role in the structure of the tower of approximations produced by orthogonal calculus. We construct these derivatives by defining an adjunction between the intermediate categories. In particular, by the previous section, this says that the derivatives of an input functor are given by spectra with an action of $O(n)$. 

Let $n\geq m$. Consider the map 
\begin{align*}
    (i^n_m)_{V,W}:\gamma_m(V,W)&\rightarrow\gamma_n(V,W)\\
    (f,x)&\mapsto (f,(i^n_m\otimes \id) (x))
\end{align*}
where $i^n_m:\mathbb{R}^m\rightarrow \mathbb{R}^n$\index{$i_m^n$} is the standard inclusion as the first $m$ entries. 
Passing to Thom spaces yields a sequence of enriched functors 
$$\mathcal{J}_0\xrightarrow{i_0^1}\mathcal{J}_1\xrightarrow{i_1^2}\mathcal{J}_2\xrightarrow{i_2^3}\mathcal{J}_3\rightarrow\dots$$

\begin{definition}\label{def: res and ind for weiss calc}\index{$\res_m^n$}\index{$\ind_m^n$}
Let $m\leqslant n$. Define the \emph{restriction functor} $\res_m^n:\mathcal{E}_n\rightarrow \mathcal{E}_m$ as precomposition with $i_m^n:\mathcal{J}_m\rightarrow\mathcal{J}_n$.

Define the \emph{induction functor} $\ind_m^n:\mathcal{E}_m\rightarrow \mathcal{E}_n$ by 
\begin{equation*}
\left(\ind_m^nF\right)(U)=\Nat_m(\mathcal{J}_n(U,-),F).
\end{equation*}\end{definition}

\begin{remark}
Note that the restriction functor $\res_m^n$ is often omitted from notation. For a functor $X\in \mathcal{E}_m$, we call $\ind_m^n X$ the $(n-m)$-derivative of $X$, denoted by $X^{(n-m)}$. In particular, $X^{(n)}=\ind_0^n X$ is the $n$-th derivative of an input functor $X\in \mathcal{E}_0$. 
\end{remark}

With the addition of inflation-orbit change-of-group functors for spaces, see \cite[Section 4]{BO13}, this defines an adjunction
\begin{equation*}
    \res_0^n /O(n) : O(n)\mathcal{E}_n \rightleftarrows \mathcal{E}_0 : \ind_0^n \varepsilon^* 
\end{equation*}\index{$\res_0^n /O(n) : O(n)$}
which is upgraded to a Quillen equivalence when the categories are equipped with the stable and homogeneous model structures, see \cite[Theorem 10.1]{BO13}. 

The following proposition defines induction iteratively as a homotopy fibre, and acts as a tool for calculating the derivatives. 

\begin{proposition}[{\cite[Theorem 2.2]{Wei95}}]
For all $U\in\mathcal{J}_0$ and for all $F\in\mathcal{E}_n$, there exists a homotopy fibre sequence
\begin{equation*}
\ind_{n}^{n+1} F(U)\rightarrow F(U)\rightarrow \Omega^{n}F(U\oplus \R)
\end{equation*}
where $\Omega^{n}Y$ is the space of pointed maps $S^{n}\rightarrow Y$, for a pointed topological space $Y$.
\end{proposition}

\begin{proof}
From Proposition \ref{prop: wei 1.2} there exists a homotopy cofibre sequence  
\begin{equation*}
    \mathcal{J}_n(\R\oplus U,-)\wedge S^{n} \rightarrow \mathcal{J}_n(U,-)\rightarrow \mathcal{J}_{n+1}(U,-).
\end{equation*}
Pick $F\in\mathcal{E}_n$ and apply the contravariant functor $\Nat_{n}(-,F)$ to the cofibre sequence above. This yields a homotopy fibre sequence
\begin{align*}
\Nat_{n}\left(\mathcal{J}_n(\R\oplus U,-)\wedge S^{n},F\right)\leftarrow \Nat_{n}&\left(\mathcal{J}_n(U,-),F\right)
\leftarrow \Nat_{n}\left(\mathcal{J}_{n+1}(U,-),F\right).
\end{align*}
Application of the Yoneda Lemma and the definition of  $\ind_{n}^{n+1}$ gives the desired fibre sequence.
\end{proof}

\begin{example}\cite[Example 2.7]{Wei95}
The first derivative of $BO(-)$ is the orthogonal sphere spectrum $\mathbb{S}:V\mapsto S^V$. The second derivative of $BO(-)$ is a shifted sphere spectrum, and the third derivative of $BO(-)$ is a shifted $\mathbb{Z}/3$-Moore spectrum. 
\end{example}

\subsection{Polynomial functors}

In differential calculus, a real function is approximated by polynomial functions, which are the partial sums of the associated Taylor series. Analogously, polynomial functors are a crucial ingredient in orthogonal calculus. In particular, the $n$-polynomial approximation functors form the layers of the tower of approximations. As suggested by the name, these functors have similar properties to polynomial functions. They are discussed in detail by Weiss \cite[Section 5]{Wei95} and Barnes and Oman \cite[Sections 5 and 6]{BO13}. 

\begin{definition}\index{$\tau_n$}
Let $X\in \mathcal{E}_0$. Define the functor $\tau_n X\in\mathcal{E}_0$ by 
\begin{equation*}
    \tau_n X(V)=\Nat_0(S\gamma_{n+1}(V,-)_+,X).
\end{equation*}
\end{definition}

The functor $\tau_n$ can be defined in a different way using \cite[Proposition 4.2]{Wei95}. 
\begin{equation*}
    \tau_n X(V)=\underset{0\neq U \subseteq \mathbb{R}^{n+1}}{\holim}X(U\oplus V).
\end{equation*}
This homotopy limit is taken over the poset of non-zero subspaces of $\R^{n+1}$ and is constructed to take into account that this poset is internal to $\Top_*$. This is discussed in more detail by Weiss in \cite{Wei98}. 

Now we define what it means for a functor to be $n$-polynomial. This is exactly Weiss' definition of polynomial of degree less than or equal to $n$ \cite[Definition 5.1]{Wei95}.

\begin{definition}
A functor $X\in\mathcal{E}_0$ is defined to be \emph{$n$-polynomial} if and only if the map $$(\rho_n)_X:X(V)\rightarrow\tau_n X(V)$$ is a weak homotopy equivalence, for all $V\in\mathcal{J}_0$. 
\end{definition}

In differential calculus, an $n$-polynomial function is also $(n+1)$-polynomial. The same property holds for polynomial functors. This is \cite[Proposition 5.4]{Wei95}, which is also found in \cite[Proposition 6.7]{BO13}. 
\begin{proposition}[{\cite[Proposition 5.4]{Wei95}}]
Let $X\in\mathcal{E}_0$. If $X$ is $n$-polynomial, then it is also $(n+1)$-polynomial. 
\end{proposition}

The fibre of the map $X\rightarrow \tau_n X$ determines how far a functor $X$ is from being $n$-polynomial, and this fibre is exactly the $(n+1)$-derivative of $X$. The following Proposition \cite[Proposition 5.3]{Wei95} describes how this fibre can be calculated. 

\begin{proposition}[{\cite[Proposition 5.3]{Wei95}}]\label{3.4}
For all $X\in\mathcal{E}_0$ and for all $V\in\mathcal{J}_0$, there exists a natural fibration sequence 
\begin{equation*}
    X^{(n+1)}(V)\rightarrow X(V)\xrightarrow{(\rho_n)_X} \tau_n X(V).
\end{equation*}
\end{proposition}

In particular, Proposition \ref{3.4} describes the relation between differentiation and polynomial functors. From this fibration sequence, one can see that if a functor $X$ is $n$-polynomial, then $X^{(n+1)}$ is contractible. This is analogous to an $n$-polynomial function having zero $(n+1)$-st derivative.

\begin{definition}\index{$T_n$}
Let $X\in \mathcal{E}_0$. Define the \emph{$n$-polynomial approximation functor} of $X$ by
\begin{equation*}
    T_nX=\hocolim \left(X\xrightarrow{\rho_X}\tau_n X\xrightarrow{\tau_n\rho_X} \tau_n^2X\xrightarrow{\tau_n^2\rho_X} \dots\right).
\end{equation*}
There is a natural transformation $\eta_n :X\rightarrow T_nX$, which is the map of homotopy colimits 
% https://q.uiver.app/#q=WzAsMixbMCwwLCJcXHRleHR7aG9jb2xpbX0oRVxccmlnaHRhcnJvdyBFXFxyaWdodGFycm93IEVcXHJpZ2h0YXJyb3cgXFxkb3RzKSJdLFswLDEsIlxcdGV4dHtob2NvbGltfVxcbGVmdChFXFxyaWdodGFycm93XFx0YXVfe3AscX0gRVxccmlnaHRhcnJvd1xcdGF1X3twLHF9XjIgRVxccmlnaHRhcnJvd1xcZG90c1xccmlnaHQpIl0sWzAsMSwiXFxldGEiXV0=
\[\begin{tikzcd}
	{\hocolim(X\rightarrow X\rightarrow X\rightarrow \dots)} \\
	{\hocolim\left(X\rightarrow\tau_{p,q} X\rightarrow\tau_{p,q}^2 X\rightarrow\dots\right)}
	\arrow["\eta", from=1-1, to=2-1]
\end{tikzcd}\]
There are also maps $T_nX\rightarrow T_{n-1}X$, induced by the inclusions $\mathbb{R}^{n-1}\rightarrow  \mathbb{R}^{n} $.
\end{definition}

As the name would suggest, $T_n X\in\mathcal{E}_0$ is an $n$-polynomial functor. This can be proven using relations between homotopy limits and sequential homotopy colimits, see \cite[Theorem 6.3]{Wei95} and \cite{Wei98}. It is also true that if the functor $X$ is already $n$-polynomial, then $X\simeq T_nX$, see \cite[Theorem 6.3]{Wei95}. In particular, this implies that $T_mT_nX\simeq T_nX$ for any $m\geq n$. 

\begin{example}
If $X$ is $0$-polynomial, then $X(V)\simeq X(V\oplus \R)$. That is, $X$ is homotopically constant. In particular, the $0$-polynomial approximation of a functor $X$ is the constant functor which takes value 
\begin{equation*}
    T_0X(V)=X(\R^\infty):=\hocolim_n X(\R^n)
\end{equation*}
for each $V\in \mathcal{J}_0$.
\end{example}

To better understand polynomial functors, we construct a model structure on the input category $\mathcal{E}_0$, whose fibrant objects are the $n$-polynomial functors. We call this the $n$-polynomial model structure. 

\begin{proposition}[{\cite[Proposition 6.5]{BO13}}]\index{$n \poly \mathcal{E}_0$}
There is a proper model structure on $\mathcal{E}_0$ such that a morphism $f$ is a weak equivalence if and only if $T_n f$ is a levelwise weak equivalence. The fibrant objects are the $n$-polynomial functors. A morphism $f$ is a fibration if and only if it is an objectwise fibration and the diagram 
% https://q.uiver.app/#q=WzAsNCxbMCwwLCJYIl0sWzEsMCwiWSJdLFswLDEsIlRfblgiXSxbMSwxLCJUX25ZIl0sWzAsMSwiZiJdLFsxLDMsIlxcZXRhIl0sWzAsMiwiXFxldGEiLDJdLFsyLDMsIlRfbmYiLDJdXQ==
\[\begin{tikzcd}
	X & Y \\
	{T_nX} & {T_nY}
	\arrow["f", from=1-1, to=1-2]
	\arrow["\eta", from=1-2, to=2-2]
	\arrow["\eta"', from=1-1, to=2-1]
	\arrow["{T_nf}"', from=2-1, to=2-2]
\end{tikzcd}\]
is a homotopy pullback square in $\mathcal{E}_0$. We call this the $n$-polynomial model structure, and denote it by $n \poly \mathcal{E}_0$. 
\end{proposition}

\begin{proof}
The model structure can be constructed as the Bousfield-Friedlander localisation of the projective model structure at the functor $T_n$, see Proposition \ref{prop: E0 proj model structure}. This proves the existence of the model structure, and guarantees that it is proper. Alternatively, one can construct the $n$-polynomial model structure as the left Bousfield localisation of the projective model structure with respect to the set of maps 
\begin{equation*}
    S_n=\{S\gamma_{n+1}(V,-)_+\rightarrow \mathcal{J}_0(V,-):V\in\mathcal{J}_0\}.
\end{equation*}
This construction guarantees that the model structure is also cellular. Since Bousfield-Friedlander and left Bousfield localisations do not change cofibrations, and the fibrant objects of both model structures are the same, these two constructions do indeed agree. 
\end{proof}

\subsection{Homogeneous functors}

The fibre of the map $T_n X\rightarrow T_{n-1} X$ is $n$-polynomial and has trivial $(n-1)$-polynomial approximation. Functors of this type are called $n$-homogeneous. The main result of the calculus is that $n$-homogeneous functors are completely determined by orthogonal spectra with an action of $O(n)$. Therefore, these fibres are much more computable. 

\begin{definition}
Let $X\in\mathcal{E}_0$. $X$ is defined to be \emph{$n$-homogeneous} if it is $n$-polynomial and $T_{n-1}X(V)$ is contractible for all $V\in\mathcal{J}_0$. 
\end{definition}

The following example is \cite[Example 5.7]{Wei95}. It forms one half of the classification theorem. 
\begin{example}\label{ex: 5.7 orth calc}
Let $\Theta\in Sp^O[O(n)]$. The functor in $\mathcal{E}_0$ defined by 
\begin{equation*}
    V\mapsto \Omega^{\infty}[(S^{nV}\wedge \Theta)_{hO(n)}]
\end{equation*}
is $n$-homogeneous. 
\end{example}

\begin{proposition}\index{$D_n$}
The homotopy fibre $D_nX=\hofibre[T_nX\rightarrow T_{n-1}X]$ is an $n$-homogeneous functor.
\end{proposition}

\begin{proof}
One can prove, using an application of the Five Lemma, that the homotopy fibre of a map between two $n$-polynomial functors is also $n$-polynomial. We know already that $T_nX $ is $n$-polynomial. $T_{n-1}X$ is also $n$-polynomial, given that it is $(n-1)$-polynomial and using \cite[Proposition 5.4]{Wei95}. Hence, $D_nX$ is $n$-polynomial. 

Using that $T_n$ and $T_{n-1}$ commute with homotopy fibres, and in particular that $$T_{n-1}T_nX\simeq T_nT_{n-1}X\simeq T_{n-1}X,$$we get
\begin{align*}
T_{n-1}D_nX&=T_{n-1}\hofibre[T_nX\rightarrow T_{n-1}X]\\
&\simeq \hofibre[T_{n-1}T_nX\rightarrow T_{n-1}^2 X]\\
&\simeq \hofibre[T_nT_{n-1}X\rightarrow T_{n-1}X]\\
&\simeq \hofibre[T_{n-1}X\rightarrow T_{n-1}X]\\
&\simeq \ast
\end{align*}
Thus, $D_nX$ is $n$-homogeneous.
\end{proof}

As with $n$-polynomial functors, there is a model structure on the input category $\mathcal{E}_0$ that captures the structure of the $n$-homogeneous functors. 

\begin{proposition}[{\cite[Proposition 6.9]{BO13}}]\index{$n\homog\mathcal{E}_0$}
There exists a model structure on $\mathcal{E}_0$ whose cofibrant-fibrant objects are the $n$-homogeneous functors that are cofibrant in the projective model structure on $\mathcal{E}_0$. Fibrations are the same as for the $n$-polynomial model structure and weak equivalences are morphisms $f$ such that $\res_0^{n}\ind_{0}^{n}T_n f$ is an objectwise weak equivalence. 
We call this the $n$-homogeneous model structure on $\mathcal{E}_0$ and denote it by $n\homog\mathcal{E}_0$. 

\noindent There is a Quillen adjunction 
\begin{equation*}
    \Id: n\homog\mathcal{E}_0\rightleftarrows n\poly\mathcal{E}_0:\Id.
\end{equation*}
\end{proposition}

\begin{proof}
Right Bousfield localisation of the $n$-polynomial model structure with respect to the set 
\begin{equation*}
    K_n=\{\mathcal{J}_n(V,-):V\in \mathcal{J}_0\}
\end{equation*}
yields the desired model structure and adjunction. 
\end{proof}

\subsection{The classification of $n$-homogeneous functors}\label{sec: classification n-homog}

To recap, so far we have constructed a tower of $n$-polynomial approximation functors $T_nX$, and the fibres of the maps between them $D_nX$ are $n$-homogeneous functors. 

% https://q.uiver.app/#q=WzAsMTAsWzAsMiwiXFxkb3RzIl0sWzEsMiwiVF9uWChWKSJdLFsyLDIsIlRfe24tMX1YKFYpIl0sWzMsMiwiXFxkb3RzIl0sWzQsMiwiVF8xWChWKSJdLFs1LDIsIlgoXFxSXlxcaW5mdHkpIl0sWzUsMCwiWChWKSJdLFs0LDMsIkRfMVgoVikiXSxbMiwzLCJEX3tuLTF9WChWKSJdLFsxLDMsIkRfe259WChWKSJdLFswLDFdLFsxLDJdLFsyLDNdLFszLDRdLFs0LDVdLFs3LDRdLFs4LDJdLFs5LDFdLFs2LDVdLFs2LDQsIiIsMix7ImN1cnZlIjoxfV0sWzYsMiwiIiwyLHsiY3VydmUiOjJ9XSxbNiwxLCIiLDIseyJjdXJ2ZSI6M31dXQ==
\[\begin{tikzcd}
	&&&&& {X(V)} \\
	\\
	\dots & {T_nX(V)} & {T_{n-1}X(V)} & \dots & {T_1X(V)} & {X(\R^\infty)} \\
	& {D_{n}X(V)} & {D_{n-1}X(V)} && {D_1X(V)}
	\arrow[from=3-1, to=3-2]
	\arrow[from=3-2, to=3-3]
	\arrow[from=3-3, to=3-4]
	\arrow[from=3-4, to=3-5]
	\arrow[from=3-5, to=3-6]
	\arrow[from=4-5, to=3-5]
	\arrow[from=4-3, to=3-3]
	\arrow[from=4-2, to=3-2]
	\arrow[from=1-6, to=3-6]
	\arrow[curve={height=6pt}, from=1-6, to=3-5]
	\arrow[curve={height=12pt}, from=1-6, to=3-3]
	\arrow[curve={height=18pt}, from=1-6, to=3-2]
\end{tikzcd}\]

We now wish to characterise the fibres $D_nX$ as a functors built from spectra. This process is outlined by Barnes and Oman in \cite[Sections 9 and 10]{BO13}. The model categories and their relations discussed in the previous sections are summarised by the following diagram. 

% https://q.uiver.app/#q=WzAsNixbMCwwLCJPKG4pXFxtYXRoY2Fse0V9X25ebCJdLFswLDEsIk8obilcXG1hdGhjYWx7RX1fbl5zIl0sWzIsMCwiXFxtYXRoY2Fse0V9XzAiXSxbNCwwLCJuXFx0ZXh0ey1wb2x5LX1cXG1hdGhjYWx7RX1fMCJdLFs0LDEsIm5cXHRleHR7LWhvbW9nLX1cXG1hdGhjYWx7RX1fMCJdLFswLDIsIlNwXlxcbWF0aGNhbHtPfVtPKG4pXSJdLFswLDIsIlxcdGV4dHtyZXN9XzBebi9PKG4pIiwwLHsib2Zmc2V0IjotMn1dLFsyLDAsIlxcdGV4dHtpbmR9XzBeblxcdmFyZXBzaWxvbl4qIiwwLHsib2Zmc2V0IjotMn1dLFszLDIsIlxcdGV4dHtpZH0iLDAseyJvZmZzZXQiOi0yfV0sWzIsMywiXFx0ZXh0e2lkfSIsMCx7Im9mZnNldCI6LTJ9XSxbMyw0LCJcXHRleHR7aWR9IiwwLHsib2Zmc2V0IjotMn1dLFs0LDMsIlxcdGV4dHtpZH0iLDAseyJvZmZzZXQiOi0yfV0sWzEsNCwiXFx0ZXh0e3Jlc31fMF5uL08obikiLDAseyJvZmZzZXQiOi0yfV0sWzQsMSwiXFx0ZXh0e2luZH1fMF5uXFx2YXJlcHNpbG9uXioiLDAseyJvZmZzZXQiOi0yfV0sWzAsMSwiXFx0ZXh0e2lkfSIsMix7Im9mZnNldCI6Mn1dLFsxLDAsIlxcdGV4dHtpZH0iLDIseyJvZmZzZXQiOjJ9XSxbMSw1LCIoXFxhbHBoYV9uKV8hIiwyLHsib2Zmc2V0IjoyfV0sWzUsMSwiXFxhbHBoYV9uXioiLDIseyJvZmZzZXQiOjJ9XV0=
\[\begin{tikzcd}
	{O(n)\mathcal{E}_n^l} && {\mathcal{E}_0} && {n\poly\mathcal{E}_0} \\
	{O(n)\mathcal{E}_n^s} &&&& {n\homog\mathcal{E}_0} \\
	{Sp^O[O(n)]}
	\arrow["{\res_0^n/O(n)}", shift left=2, from=1-1, to=1-3]
	\arrow["{\ind_0^n\varepsilon^*}", shift left=2, from=1-3, to=1-1]
	\arrow["{\Id}", shift left=2, from=1-5, to=1-3]
	\arrow["{\Id}", shift left=2, from=1-3, to=1-5]
	\arrow["{\Id}", shift left=2, from=1-5, to=2-5]
	\arrow["{\Id}", shift left=2, from=2-5, to=1-5]
	\arrow["{\res_0^n/O(n)}", shift left=2, from=2-1, to=2-5]
	\arrow["{\ind_0^n\varepsilon^*}", shift left=2, from=2-5, to=2-1]
	\arrow["{\Id}"', shift right=2, from=1-1, to=2-1]
	\arrow["{\Id}"', shift right=2, from=2-1, to=1-1]
	\arrow["{(\alpha_n)_!}"', shift right=2, from=2-1, to=3-1]
	\arrow["{\alpha_n^*}"', shift right=2, from=3-1, to=2-1]
\end{tikzcd}\]

In \cite[Theorem 10.1]{BO13}, the bottom horizontal Quillen adjunction is proven to be a Quillen equivalence. 
\begin{theorem}[{\cite[Theorem 10.1]{BO13}}]
There exists a Quillen equivalence 
\begin{equation*}
    \res_0^n/O(n): O(n)\mathcal{E}_n^s \rightleftarrows n\homog\mathcal{E}_0: \ind_0^n\varepsilon^*.
\end{equation*}
\end{theorem}

Combined with the bottom left vertical Quillen equivalence, see Section \ref{sec: En as spectra}, this results in a zig-zag of equivalences between the stable model structure on spectra with an action of $O(n)$ and the $n$-homogeneous model structure on $\mathcal{E}_0$. 

Under this zig-zag of equivalences, we denote the spectrum which is the image of an input functor $X$ by $\Theta_X^n$. That is, $\Theta_X^n=\mathbb{L}(\alpha_n)_!\mathbb{R}\ind_0^n\varepsilon^* X$, where $\mathbb{L}$ and $\mathbb{R}$ denote taking the left and right derived functors respectively. This is weakly equivalent to the spectrum $\Theta X^{(n)}$ constructed in \cite[Section 2]{Wei95}. The classification theorem states that an $n$-homogeneous functor $X$ is levelwise weakly equivalent to a functor built from the spectrum $\Theta_X^n$. 

\begin{theorem}[{\cite[Theorem 7.3]{Wei95}}]\label{thm: classification weiss calc}
Let $X$ be an $n$-homogeneous functor in $\mathcal{E}_0$, for $n>0$. Then $X$ is levelwise weakly equivalent to the functor defined by 
\begin{equation*}
    V\mapsto \Omega^\infty [(S^{nV}\wedge \Theta_X^n)_{hO(n)}].
\end{equation*}
Conversely, any functor of the form 
\begin{equation*}
    V\mapsto \Omega^\infty [(S^{nV}\wedge \Theta)_{hO(n)}]
\end{equation*}
for $\Theta\in Sp^O[O(n)]$, is $n$-homogeneous. 
\end{theorem}

Note that the converse statement is exactly Example \ref{ex: 5.7 orth calc}. 

In particular, applied to the fibres of the maps $T_nX\rightarrow T_{n-1}X$, this gives a description of the fibres of the orthogonal tower in terms of spectra with an action of $O(n)$. 

\begin{theorem}[{\cite[Theorem 9.1]{Wei95}}]
For each $X\in \mathcal{E}_0$, $n>0$, $V\in\mathcal{J}_0$, there exists a homotopy fibre sequence
\begin{equation*}
    \Omega^\infty[(S^{nV}\wedge \Theta_X^n)_{hO(n)}]\rightarrow T_nX(V)\rightarrow T_{n-1}X(V).
\end{equation*}
\end{theorem}

\chapter{Equivariant functor categories}
\label{ch:EquivFunct}

\fancyhf{}
\fancyhead[C]{\rightmark}
\fancyhead[R]{\thepage}

The primary objects of study in calculus of real functions are derivatives. In orthogonal calculus, one constructs derivatives of input functors via combinations of restriction functors and the inflation-orbit change-of-group functors for spaces (see Definition \ref{def: res and ind for weiss calc}). These derivatives play a key role in the classification of $n$-homogeneous functors, as they form part of the zig-zag of Quillen Equivalences used to derive the classification Theorem (Theorem \ref{thm: classification weiss calc}). 

We will extend this notion to the $\C$-equivariant setting by defining new functor categories and adjunctions analogous to those used in the underlying calculus. We begin by choosing a new indexing category that will induce the $\C$-actions used throughout the calculus. With the $\C$-setting fixed, we can construct the jet categories $\Jpq$ as well as the intermediate categories $\OEpq$. The relationships between these categories (which include derivatives), and in particular their various model structures, are heavily relied upon later in the homotopical part of the calculus; the classification of $(p,q)$-homogeneous functors as a category of orthogonal $\C$-spectra with an action of $O(p,q)$. In particular, objects in this category of orthogonal spectra have a genuine action of $\C$ and a naive action of $O(p,q)$.

\section{The input functors}\label{sect: input functors}

The input functors to $\C$-equivariant orthogonal calculus are continuous functors from a category of finite dimensional $\C$-representations (with inner product, see Remark \ref{rem: all reps have inner prod}) to the category of pointed $\C$-spaces. For example, the functor $$BO(-):V\mapsto BO(V)_+,$$see Examples \ref{Examples}. We call the category of such functors $\Ezero$. In this section we give a more detailed description of this category, by defining a new indexing category $\CL$. 

First, recall that $\C$\index{$\C$} is the cyclic group of order two, which we write as $\C=\{e,\sigma\}$. The category of $\C$-spaces is assumed to be equipped with its fine model structure throughout this main text, see Proposition \ref{finemodelstructure}. In particular, a map $f:X\rightarrow Y$ in $\CTop_*$ is a weak equivalence if and only if $f^{\C}:X^{\C}\rightarrow Y^{\C}$ and $f^e:X^e\rightarrow Y^e$ are weak homotopy equivalences. 

Orthogonal calculus is indexed on the universe $\mathbb{R}^\infty$, which makes the input category of functors enriched over topological spaces. To guarantee that the category of input functors for $\C$-equivariant orthogonal calculus is enriched over pointed $\C$-spaces, we must specify a new universe which is closed under $\C$-action. 

 Recall from Example \ref{ex: regular c2 rep} that the regular representation of $\C=\{e,\sigma\}$ is defined as the following vector space\begin{equation*}
    \RC=\{\lambda_1 \underline{e}+\lambda_2 \underline{\sigma}:\lambda_1,\lambda_2\in\R\}
\end{equation*}
with basis elements $\underline{e}, \underline{\sigma}$. To better understand the $\C$-action on the vector space $\RC$, we can decompose \begin{equation*}\RC=\R\langle\underline{e}+\underline{\sigma}\rangle\oplus\R\langle\underline{e}-\underline{\sigma}\rangle.\end{equation*}This direct sum is $\C$-isomorphic to $\R\oplus \Rdelta$, where $\R$ and $\Rdelta$ are the trivial and sign $\C$-representations respectively with $\C$-actions defined below
\begin{align*}
    \sigma(x)&=x\quad (x\in\R),\\
    \sigma(y)&=-y\quad (y\in\Rdelta).
\end{align*}

Choosing the universe $\bigoplus\limits_{i=1}^{\infty} \RC$, we define a new indexing category as follows. 
\begin{definition}\index{$\CL$}\index{$\LL(U,V)$}
The \emph{equivariant indexing category} $\CL$ is a $\CTop_*$-enriched category whose objects are the finite dimensional subrepresentations of $\bigoplus\limits_{i=1}^{\infty} \RC$ with inner product. Let $\LL(U,V)$ denote the space of (not necessarily $\C$-equivariant) linear isometries with the $\C$-action that is conjugation. That is, for $f\in \LL(U,V)$, $e*f=f$ and $\sigma * f = \sigma f \sigma : U\rightarrow V$. The hom-object of morphisms $U\rightarrow V$ is the pointed $\C$-space $\LL(U,V)_+$. 
\end{definition}

From the discussion above on the decomposition of $\RC$, we can see that an object in the category $\CL$ is isomorphic to an object of the form $\R^{p+q\delta}=\R^p\oplus \R^{q\delta}$, for some $p,q\in \mathbb{N}$. That is, $p$ copies of the trivial representation and $q$ copies of the sign representation.  

\begin{remark}\index{$\R^{p,q}$}\index{$(p,q)V$}
We will use the notation $\R^{p,q}$ to mean $\R^{p+q\delta}$. We will also use the notation $(p,q)V$ to mean $\R^{p,q}\otimes V$ equipped with the diagonal action of $\C$, where $V\in\CL$.
\end{remark}

\begin{example}
By definition, $\LL (\R^{k\delta},\R^n)_+$ is the space of linear isometries from $\R^{k\delta}$ to $\R^n$, which is known to be isomorphic to $O(n)/O(n-k)_+$ (via the map $O(n)\rightarrow \LL (\R^{k\delta},\R^n)$ that sends a matrix $A$ to the matrix given by the first $k$ columns of $A$).  

Let $f\in \LL (\R^{k\delta},\R^n)_+$, then $\sigma f \sigma$ is as follows, where the last equality is known from linearity, rather than an assumption in equivariance. 
\begin{equation*}
x\overset{\sigma}\mapsto -x \overset{f}\mapsto f(-x) \overset{\sigma}\mapsto f(-x)=-f(x)
\end{equation*}
That is, $\C$ acts on $\LL (\R^{k\delta},\R^n)_+$ as multiplication by -1. 
\end{example}

\begin{remark}
Note that $\CTop_*$ is also enriched over itself, by equipping the space of continuous maps $\Top_*(X,Y)$ with the conjugation action (see Section \ref{sec:Gspace}). 
\end{remark}

We can now define the input category for $\C$-orthogonal calculus, which we denote by $\Ezero$. 
\begin{definition}\label{jzero and ezero def}\index{$\Jzero$}
Define $\Jzero$ to be the $\CTop_*$-enriched category with the same objects as $\CL$ and morphisms defined by $\Jzero(U,V)=\LL(U,V)_+$. 
Define the \emph{input category} $\Ezero$ to be the category of $\CTop_*$-enriched functors from $\Jzero$ to $\CTop_*$ and $\C$-equivariant natural transformations (see Remark \ref{remNat} for details). 
\end{definition}

\begin{examples}\label{Examples}
$\quad$
\begin{itemize}
    \item $BO(-):V\mapsto BO(V)_+ $ 
    \item $S^{(-)}:V\mapsto S^V$
    \item $\Jmor{p,q}{U}{-}:V\mapsto \Jmor{p,q}{U}{V}$ 
\end{itemize}
where $BO(V)$ is the classifying space of the space of linear isometries on $V$ and $S^V$ is the one point compactification of $V$. The $\C$-action on $O(V)$ is conjugation, that is $\sigma * h = \sigma h \sigma^{-1}$ for $h\in O(V)$. The $\C$-action on $BO(V)$ is induced by the action on $O(V)$, see Section \ref{sec: BO}. The $\C$-action on $S^V$ is induced by the $\C$-action on $V$, and the $\C$-action on $\Jpq(U,V)$ is induced by the $\C$-action on $\C\gamma_{p,q}(U,V)$ (see Definition \ref{def: gamma pq bundle}). 
\end{examples}

We now define a model structure on the input category $\Ezero$. This model structure is the projective model structure, similar to that of \cite[Theorem 6.5]{MMSS01}, and will be used to build the polynomial and homogeneous model structures. The projective model structure on $\Ezero$ is a special case of the projective model structure on $\OEpq$ for $p=q=0$, as such we defer the proof to Lemma \ref{objectwise model structure proof}.
\begin{proposition}\label{proj model structure}\index{$\Ezero$}
There is a proper, cellular, $\C$-topological model structure on $\Ezero$ where the fibrations and weak equivalences are defined objectwise from the fine model structure on $\CTop_*$. We call them objectwise fibrations and objectwise weak equivalences. We call this the projective model structure on $\Ezero$ and denote it by $\Ezero$. It is cofibrantly generated by the following sets of generating cofibrations and generating acyclic cofibrations respectively 
\begin{align*}
    &I_{\proj}=\{\Jzero (V,-)\wedge i : i\in I_{\C}\}\\
    &J_{\proj}=\{\Jzero (V,-)\wedge j : j\in J_{\C}\},
\end{align*}
where $V\in \Jzero$ and $I_{\C},J_{\C}$ are the generating cofibrations and acyclic cofibrations of the fine model structure on $\CTop_*$ (see Proposition \ref{finemodelstructure}). 
\end{proposition}

\section{The intermediate categories}

In orthogonal calculus, one constructs intermediate categories $O(n)\mathcal{E}_n$ which are intermediate between the input category and the category of orthogonal spectra with an action of $O(n)$. We replicate this process in the $\C$-equivariant setting to construct new intermediate categories $\OEpq$. These categories will give greater insight into the structure of the input category $\Ezero$ defined in Section \ref{sect: input functors}. The construction follows the orthogonal calculus version of Weiss \cite[Sections 1 and 2]{Wei95} and Barnes and Oman \cite[Sections 3 and 8]{BO13}.

We begin by defining the following $\C$-equivariant vector bundle.
\begin{definition}\label{def: gamma pq bundle}\index{$\Gmor{p,q}{U}{V}$}
Let $U,V\in \CL$. Define the \emph{$(p,q)$-th complement bundle} $\Gmor{p,q}{U}{V}$ to be the $\C$-equivariant vector bundle on $\LL(U,V)$, whose total space is given by
\begin{equation*}
    \Gmor{p,q}{U}{V}=\{(f,x):f\in \LL(U,V), x\in \R^{p,q} \otimes f(U)^\perp\},
\end{equation*}
where $f(U)^\perp$ denotes the orthogonal complement of the image of $f$.

\noindent Let $(f,x)\in \Gmor{p,q}{U}{V} $. Where $\sigma * f = \sigma f \sigma $, define a $\C$-action on $\Gmor{p,q}{U}{V} $ by \begin{equation*}
\sigma(f,x)=(\sigma *f, \sigma x).
\end{equation*}
\end{definition}

Note that $\Gmor{p,q}{U}{V}$ is a subbundle of the product bundle over $\LL(U,V)$ whose total space is $\LL(U,V)\times (\R^{p,q} \otimes V)$.

One can verify that the $\C$-action on $\Gmor{p,q}{U}{V}$ is well defined by considering an element $$x=\sum\limits_i w_i\otimes v_i\in \R^{p,q} \otimes f(U)^\perp.$$ Under the $\C$-action this is mapped to $\sigma x=\sum\limits_i \sigma w_i\otimes \sigma v_i$. Clearly $\sigma w_i\in \R^{p,q} $ for all $i$, since $\R^{p,q} $ is closed under the $\C$-action. Now take an arbitrary element of $(\sigma *f) (U)$, which is of the form $(\sigma *f)(y)=(\sigma f \sigma )(Y)$, for $y\in U$. Then, as there is an inner product on $V$,
\begin{equation*}
    \langle (\sigma f \sigma) (y), \sigma v_i\rangle = \langle f(\sigma y), v_i \rangle= 0,
\end{equation*}
since $v_i \in f(U)^\perp$, $\sigma y\in U$ and the inner product on $V$ is $G$-invariant. Therefore, $\sigma v_i$ is an element of $\sigma f\sigma (U)^\perp$, and the action is well defined. 

\begin{example}
The total space
\begin{equation*}
    \Gmor{0,m}{\R^{k\delta}}{\R^n}=\{(f,x):f\in \LL(\R^{k\delta},\R^n), x\in \R^{m\delta}\otimes f(\R^{k\delta})^\perp\}
\end{equation*}
has $\C$-action 
\begin{equation*}
    \sigma (f,x)=(-f, -x).
\end{equation*}
\end{example}

The following result outlines the effect of the fixed point functor $(-)^{\C}:\CTop_*\rightarrow\Top_*$ on the $C_2$-spaces $\LL(\mathbb{R}^{a,b},\mathbb{R}^{c,d})$ and $C_2 \gamma_{p,q}(\mathbb{R}^{a,b},\mathbb{R}^{c,d})$. 

\begin{theorem}[The Equivariant Splitting Theorems]\label{splittingtheorems}There are homeomorphisms
\begin{equation*}
\LL(\mathbb{R}^{a,b},\mathbb{R}^{c,d})^{C_2}\cong \LL(\mathbb{R}^{a},\mathbb{R}^{c})\times \LL(\mathbb{R}^{b\delta},\mathbb{R}^{d\delta})
\end{equation*}
\begin{equation*}
C_2\gamma_{p,q}(\mathbb{R}^{a,b},\mathbb{R}^{c,d})^{C_2}\cong \C\gamma_{p,0}(\mathbb{R}^{a},\mathbb{R}^{c})\times \C\gamma_{0,q}(\mathbb{R}^{b\delta},\mathbb{R}^{d\delta}).
\end{equation*}
\end{theorem}
\begin{proof}
A map $f:\mathbb{R}^{a,b}\rightarrow\mathbb{R}^{c,d}$ is $\C$-fixed if and only if it is $\C$-equivariant. As there are no non-zero $\C$-equivariant maps $\mathbb{R}^a\rightarrow\mathbb{R}^{d\delta}$ or $\mathbb{R}^{b\delta}\rightarrow\mathbb{R}^c$ the first splitting theorem follows. 

Let $(f,x)\in\Gmor{p,q}{\mathbb{R}^{a,b}}{\mathbb{R}^{c,d}}$. Then $(f,x)$ is $\C$-fixed if and only if
\begin{equation*}
    (\sigma *f,\sigma x)=(f,x).
\end{equation*}
By the first splitting theorem $\sigma *f=f$ if and only if $f\in\LL(\mathbb{R}^{a},\mathbb{R}^{c})\times \LL(\mathbb{R}^{b\delta},\mathbb{R}^{d\delta})$. That is, $f$ is of the form $f_1\times f_2$, where $f_1\in\LL(\mathbb{R}^{a},\mathbb{R}^{c})$ and $f_2\in \LL(\mathbb{R}^{b\delta},\mathbb{R}^{d\delta})$. 

Now, since $x\in \R^{p,q}\otimes \text{Im}(f)^\perp\subset \R^{p,q}\otimes \R^{c,d}$, $x$ is in 
\begin{equation*}
(\R^p\otimes\R^c)\oplus(\R^p\otimes\R^{d\delta})\oplus(\R^{q\delta}\otimes\R^{c})\oplus(\R^{q\delta}\otimes\R^{d\delta})
\end{equation*}
where $\C$ acts as $\id\oplus -1\oplus -1\oplus\id$. Therefore $\sigma(x)=x$ if and only if $$x\in (\R^p\otimes\R^c)\oplus(\R^{q\delta}\otimes\R^{d\delta}).$$ That is, $x$ is of the form $x_1\oplus x_2$, where $x_1\in \R^p\otimes \text{Im}(f_1)^\perp$ and $x_2\in\R^{q\delta}\otimes\text{Im}(f_2)^\perp$. The second splitting theorem follows from the well defined homeomorphism 
\begin{equation*}
    (f_1\times f_2,x_1\oplus x_2)\mapsto((f_1,x_1),(f_2,x_2)).\qedhere
\end{equation*}\end{proof}

Now we define what will become the morphism spaces for the categories $\Jpq$. These categories are analogous to the $n$-th jet categories of orthogonal calculus, and will be used to build the intermediate categories. 

\begin{definition}\label{thom}\index{$\Jmor{p,q}{U}{V}$}
Let $U,V \in \CL$. Define $\Jmor{p,q}{U}{V}$ to be the Thom space of $\Gmor{p,q}{U}{V}$.
\end{definition}

This Thom space is the one point compactification of $\Gmor{p,q}{U}{V}$, since $\LL(U,V)$ is compact. Hence, each $\Jmor{p,q}{U}{V}$ is a pointed $\C$-space, with $\C$-action inherited from $\Gmor{p,q}{U}{V}$. 

We can now define the categories $\Jpq$. 
\begin{definition}\label{def: p,q-jet cat}\index{$\Jpq$}
Let the \emph{$(p,q)$-th jet category} $\Jpq$ be the $\CTop$-enriched category whose objects are the same as $\CL$, and whose morphism are given by $\Jmor{p,q}{U}{V}$. 

Composition in $\Jpq$ is defined as follows. There are maps of $\C$-spaces defined by
\begin{align*}
\Gmor{p,q}{V}{X}\times \Gmor{p,q}{U}{V}&\rightarrow \Gmor{p,q}{U}{X}\\
((f,x),(g,y))&\mapsto (f\circ g,x+(\id\otimes f)(y)).
\end{align*}Passing to Thom spaces then yields the desired composition maps
\begin{equation*}
    \Jmor{p,q}{V}{X}\wedge \Jmor{p,q}{U}{V}\rightarrow \Jmor{p,q}{U}{X}.
 \end{equation*}
One can check that this composition is a continuous map. Moreover, the composition maps are unital and associative. The following argument verifies that it is $\C$-equivariant. 
\begin{align*}
    \sigma(f\circ g,x+(\id\otimes f)(y))&=(\sigma * (f\circ g) , \sigma (x+(\id\otimes f)(y))\\
    &=((\sigma * f)\circ(\sigma *g), \sigma x + (\id\otimes \sigma f\sigma) (\sigma y)),
\end{align*}
which is equal to the image of $((\sigma *f, \sigma x), (\sigma *g, \sigma y))$ under the composition map. 
\end{definition}

One can see that the category $\Jzero$ has morphisms $\Jmor{0,0}{U}{V}=\LL(U,V)_+$, and therefore is exactly the category defined in Definition \ref{jzero and ezero def}.

The following is the a $\C$-equivariant generalisation of \cite[Theorem 1.2]{Wei95}. It demonstrates that it is possible to build the morphism spaces $\Jmor{p,q}{U}{V}$ inductively. That is, we can construct $C_2\mathcal{J}_{p+1,q}$ and $C_2\mathcal{J}_{p,q+1}$ from $\Jpq$.

\begin{proposition}\label{cofibseq}
For all $U,V,W$ in $\Jzero$, the homotopy cofibre (specific construction given in proof) of the restricted composition map
\begin{equation*}
\Jmor{W}{U\oplus X}{V}\wedge S^{W\otimes X} \rightarrow \Jmor{W}{U}{V}
\end{equation*}
is $\C$-homeomorphic to $\Jmor{W\oplus X}{U}{V}$, where $X=\R$ or $X=\Rdelta$.
\end{proposition}

\begin{proof}
We give the proof for the $X=\Rdelta$ case, and leave the similar $X=\R$ case to the reader.

We can $\C$-equivariantly identify the one point compactification of $W\otimes \Rdelta$, denoted $S^{W\otimes \Rdelta}$,  with a subspace of $\Jmor{W}{U}{U\oplus \Rdelta}$. Consider the map \begin{align*}
    W\otimes \Rdelta &\rightarrow \Gmor{W}{U}{U\oplus \Rdelta}\\
    w\otimes y&\mapsto (i_U, w\otimes (0,y)),
\end{align*} where $i_U:U\rightarrow U\oplus \Rdelta$ is the map $u \mapsto  (u,0)$. Taking Thom spaces then gives the desired identification. 

Composing this identification map with the composition map
\begin{equation*}
    \Jmor{W}{U \oplus \Rdelta}{V}\wedge \Jmor{W}{U}{U\oplus \Rdelta}\rightarrow \Jmor{W}{U}{V}
\end{equation*}
yields a continuous map 
\begin{equation*}
\Jmor{W}{\Rdelta\oplus U}{V}\wedge S^{W\otimes \Rdelta} \rightarrow \Jmor{W}{U}{V}
\end{equation*}
which is $\C$-equivariant, since both the identification and composition are $\C$-equivariant maps. This map is defined as 
\begin{equation*}
(f,x)\wedge (w\otimes y)\mapsto (f|_U, x+(\id_W\otimes f|_\Rdelta)(w\otimes y)).    
\end{equation*}

Since the map is surjective, the homotopy cofibre is a quotient of 
\begin{equation*}
    [0,\infty]\times \Jmor{W}{\Rdelta\oplus U}{V}\times (W\otimes \Rdelta).
\end{equation*}
Using the element form of this map, we can construct a homeomorphism $\phi$ from this cofibre to $\Jmor{W\oplus \Rdelta}{U}{V}$ defined by 
\begin{equation*}
    (t,f,x,w\otimes y)\mapsto (f|_U, x+ (\id_W\otimes f|_\Rdelta)(w\otimes y)+t\alpha(f|_\Rdelta (1))),
\end{equation*}
where 
\begin{align*}
&t\in[0,\infty]\\
&f:\Rdelta\oplus U\rightarrow V\\
&x\in W\otimes f(\Rdelta\oplus U)^\perp \\
&w\otimes y\in W\otimes \Rdelta
\end{align*}and $\alpha:V\rightarrow (W\oplus \Rdelta)\otimes V$ is the map that sends $V$ to the orthogonal complement of $W\otimes V$ in $(W\oplus \Rdelta)\otimes V$. That is, $\alpha(v)=(0,1\otimes v)$ in $(W\otimes V)\oplus (\Rdelta\otimes V)\cong(W\oplus \Rdelta)\otimes V$. Note that $\alpha$ need not be $\C$-equivariant, and in this case it is not. 

All that remains is to check that the map $\phi$ is $\C$-equivariant. 
\begin{align*}
    &\phi(t,\sigma * f , \sigma x, \sigma w\otimes -y)\\
    =&((\sigma *f) |_U, \sigma x+ (W\otimes (\sigma *f)|_\Rdelta)(\sigma w\otimes -y)+t\alpha((\sigma *f)|_\Rdelta (1)))\\
    =&((\sigma *f) |_U, \sigma x+ \sigma(w\otimes f(y))+t ((\sigma *f)|_{\Rdelta}(1))),
\end{align*} where we have used that $\alpha(\sigma (v))=(0,1\otimes \sigma(v))$ and $\sigma(\alpha(v))=(0,-1\otimes \sigma(v))$. This is equal to the image of $(f|_U, x+ (W\otimes f|_\Rdelta)(w\otimes y)+t\alpha(f|_\Rdelta (1))$ under the $\C$-action.
\end{proof}

\begin{remark}\label{rem: cofib seq}   
These cofibre sequences are analogous to the following cofibre sequences constructed in the underlying calculus (see Proposition \ref{prop: wei 1.2}).  
\begin{equation*}
    \mathcal{J}_{n}(U\oplus \mathbb{R},V)\wedge S^{\mathbb{R}^n} \rightarrow \mathcal{J}_{n}(U,V)\rightarrow \mathcal{J}_{n+1} (U,V)
\end{equation*}
If one wanted to replace $\mathbb{R}$ with something higher dimensional this would involve `gluing' these cofibre sequences together in an iterative manner. We want the $C_2$-equivariant calculus constructed in this thesis to reduce down to the underlying calculus after forgetting the $C_2$-actions. This forces that the cofibre sequences in Proposition \ref{cofibseq} only hold for $X=\mathbb{R}$ and $X=\mathbb{R}^\delta$, since these cases both correspond $\mathbb{R}$ in the non-equivariant statement. To replace $X$ with something of higher dimension, for example $\mathbb{R}^{1+1\delta}$, would again mean taking some kind of iteration of cofibre sequences. This indicates that a potentially more involved approach may be needed if one wants to construct this kind of result in a $G$-equivariant orthogonal calculus, for an arbitrary group $G$. As a result, it is also not obvious how derivatives should behave for the arbitrary $G$ setting, since the fibre sequences that describe derivatives (see Proposition \ref{loops fibre sequence}) are a direct consequence of these cofibre sequences. 
\end{remark}

We now wish to define the functor categories $\Epq$. At the same time, we will also define functor categories $\OEpq$, which will later be used to classify the layers of the orthogonal tower. Before we can do this, we introduce the group $O(p,q)$ and discuss its actions. 
\begin{definition}\label{def:O(p,q) and matrix A}\index{$O(p,q)$}\index{$O(p,q)\rtimes\C$}
Define $O(p,q)$ to be the group of linear isometries from $\mathbb{R}^{p,q}$ to $\mathbb{R}^{p,q}$ with the $\C$-action defined as follows. Let $\C$ act on $O(p,q)$ by conjugation by the matrix
\begin{equation*}
    A=\begin{pmatrix}
        \Id_p & 0 \\ 0 &-\Id_q
    \end{pmatrix},
\end{equation*}
where $\Id_m$ denotes the $m$-dimensional identity matrix. That is, $\sigma g=AgA^{-1}$ for $g$ in  $O(p,q)$. In particular, $O(p,q)^{\C}=O(p)\times O(q)$. 
\end{definition}

The group action of $\C$ on $O(p,q)$ can be described by the group homomorphism 
\begin{equation*}
\upvarphi:\C\rightarrow \text{Aut}(O(p,q)),\quad  \alpha\mapsto \upvarphi_\alpha   
\end{equation*}
where $\upvarphi_\alpha  (g)=\alpha g\alpha^{-1}$ for $g\in O(p,q)$. If $\alpha =e$, then $\upvarphi_\alpha $ is the identity on $O(p,q)$. We can construct a new group called the semidirect product of $O(p,q)$ and $\C$ with respect to the map $\upvarphi$. The underlying set is $O(p,q)\times \C$ and the group operation is given by 
\begin{equation*}
(g_1,\alpha_1)\bullet (g_2,\alpha_2)=(g_1\upvarphi_{\alpha_1}(g_2), \alpha_1 \alpha_2).     
\end{equation*}
The actions of $\C $ and $O(p,q)$ do not commute, but they do commute up to the operation $\upvarphi$, that is $g\alpha=\alpha \upvarphi_\alpha (g)$. We denote this semidirect product by $O(p,q)\rtimes \C$.

The group homomorphisms 
\begin{align*}
\inc:O(p,q)&\rightarrow O(p,q)\rtimes \C\\
g&\mapsto (g,e)
\end{align*}
and 
\begin{align*}
\proj:O(p,q)&\rtimes \C\rightarrow \C\\
(g,\alpha)&\mapsto \alpha
\end{align*}
form a short exact sequence
\begin{equation*}
    1\rightarrow O(p,q)\overset{\inc}{\rightarrow} O(p,q)\rtimes \C \overset{\proj}{\rightarrow} \C \rightarrow 1,
\end{equation*}
where $1$ is the trivial group. The group homomorphism $\beta:\C\rightarrow O(p,q)\rtimes\C$, defined by $\alpha\mapsto (\Id_{p+q},\alpha)$ is such that $\proj \beta= \Id_{\C}$.

\begin{remark}\label{rem: what is the ms on semi direct prod}
Throughout this thesis, we equip the category of $\C$-spaces with the fine model structure and the category of $O(p,q)$-spaces with the coarse model structure (see Propositions \ref{prop: coarse ms} and \ref{finemodelstructure}). We then equip the category of $(O(p,q)\rtimes \C)$-spaces with a model structure which is fine with respect to $\C$ and coarse with respect to $O(p,q)$. This is discussed more in Remark \ref{remark: model structure on semi direct}. 
\end{remark}

There is an action of $O(p,q)\rtimes \C$ on $\mathbb{R}^{p,q}$ given by $(T,\sigma)(x):= T(\sigma (x))$. This can be extended to an action on $\mathbb{R}^{p,q} \otimes f(U)^\perp$ by $(T,\sigma)x:=((T,\sigma) \otimes \sigma) (x)$, where $f\in \LL(U,V)$. This induces an $O(p,q)\rtimes \C$-action on $\Gmor{p,q}{U}{V}$ by $$(T,\sigma)(f,x):=(\sigma *f,((T,\sigma) \otimes\sigma)(x)).$$ Hence, there is also an $O(p,q)\rtimes \C$-action on its Thom space $\Jmor{p,q}{U}{V}$, making $\Jpq$ an $\COTop_* $-enriched category.

\begin{proposition}\label{sphere as quotient of orthogonal groups prop}
For all $p>0$ and $q\geq 0$, there exists a $\C$-equivariant homeomorphism 
\begin{equation*}
    O(p,q)/O(p-1,q)\cong S(\mathbb{R}^{p+q\delta}).
\end{equation*}
\end{proposition}
\begin{proof}

Since $O(p,q)$ acts on $\mathbb{R}^{p+q\delta}$ transitively by linear isometries, there is a restricted transitive action of $O(p,q)$ on $S(\mathbb{R}^{p+q\delta})$. Fix the vector $e_1=(1,0,...,0)$ in $S(\mathbb{R}^{p+q\delta})$. There is a continuous $\C$-equivariant map $\phi: O(p,q)\rightarrow S(\mathbb{R}^{p+q\delta})$ given by $g\mapsto ge_1$. The stabiliser of $e_1$ is the subgroup of $O(p,q)$ given by 
\begin{equation*}
    \biggl\{ \begin{pmatrix} 1 & 0\\0&A\end{pmatrix}:A\in O(p-1,q) \biggr\}, 
\end{equation*}
which is $\C$-homeomorphic to $O(p-1,q)$. 
The quotient $O(p,q)/O(p-1,q)$ inherits a $\C$-action defined by $\sigma ([g]):= [\sigma (g)]$. It follows by the orbit-stabiliser theorem that there is a continuous homeomorphism $O(p,q)/O(p-1,q)\cong S(\mathbb{R}^{p+q\delta})$. This is shown in the following commutative diagram, where $i$ is the inclusion as the subgroup above and $\proj$ is the projection onto the quotient. 

% https://q.uiver.app/?q=WzAsNCxbMSwwLCJPKHArcSkiXSxbMiwwLCJTKFxcbWF0aGJie1J9XntwK3FcXGRlbHRhfSkiXSxbMCwwLCJPKHArcS0xKSJdLFsxLDEsIk8ocCtxKS9PKHArcS0xKSJdLFsyLDAsImkiXSxbMCwxLCJcXHBoaSJdLFswLDMsIlxcdGV4dHtwcm9qfSIsMl0sWzMsMSwiXFxleGlzdHMiLDIseyJzdHlsZSI6eyJib2R5Ijp7Im5hbWUiOiJkYXNoZWQifX19XV0=
\[\begin{tikzcd}
	{O(p-1,q)} & {O(p,q)} & {S(\mathbb{R}^{p+q\delta})} \\
	& {O(p,q)/O(p-1,q)}
	\arrow["i", from=1-1, to=1-2]
	\arrow["\phi", from=1-2, to=1-3]
	\arrow["{\proj}"', from=1-2, to=2-2]
	\arrow["\exists"', dashed, from=2-2, to=1-3]
\end{tikzcd}\]

\noindent By the commutativity of the diagram, the homeomorphism is given by $[g]\mapsto ge_1$, which is a $\C$-equivariant map, since the $\C$-action on $S(\R^{p+q\delta})$ fixes $e_1$.   
\end{proof}

We can now define the intermediate categories. 
\begin{definition}\index{$\Epq$}\index{$\OEpq$}
Define $\Epq$ to be the category of $\CTop_*$ enriched functors from the $(p,q)$-th jet category $\Jpq$ to $\CTop_*$ and $\C$-equivariant natural transformations, denoted by $\C \Nat_{p,q}(-,-)$.

Define the \emph{$(p,q)$-th intermediate category} $\OEpq$ to be the category of $\COTop_*$-enriched functors from the $(p,q)$-th jet category $\Jpq$ to $\COTop_*$, and $(O(p,q)\rtimes\C)$-equivariant natural transformations.
\end{definition}

For $p,q=0$ this definition is exactly the category $\Ezero$ in Definition \ref{jzero and ezero def}.

\begin{remark}\label{remNat}\index{$\Nat_{p,q}(E,F)$}\index{$\C\Nat_{p,q}(E,F)$}

The set of natural transformations between $E,F\in \Epq$ is denoted by $\Nat_{p,q}(E,F)$. There is a natural topology on $\Nat_{p,q}(E,F)$, which is the subspace topology of a product space as follows. 
\begin{align*}
    \Nat_{p,q} (E,F)&:= \int\limits_{V\in\Jpq} \Top_* (E(V),F(V))\\
    &\subseteq \prod\limits_{V\in\Jpq} \Top_* (E(V),F(V))
\end{align*}
There is a $\C$-action on the space of natural transformations $\Nat_{p,q}(E,F)$ induced by the conjugation action on $\Top_* (E(V),F(V))$. This defines an enrichment of $\Epq$ in $\CTop_*$. 

With respect to this conjugation action, we topologise the set of $\C$-equivariant natural transformations between $E,F\in\Epq$, denoted by $\C\Nat_{p,q}(E,F):=\Nat_{p,q}(E,F)^{\C}$ as follows. 
\begin{align*}
    \C\Nat_{p,q} (E,F)&:= \int\limits_{V\in\Jpq} \C\Top_* (E(V),F(V))\\
    &\subseteq \prod\limits_{V\in\Jpq} \C\Top_* (E(V),F(V))
\end{align*}

Similar descriptions exist for the morphisms in $O(p,q)\Epq$. 

We can describe a functor $E\in \Epq$ in terms of an enriched coend (and similarly for $O(p,q)\Epq$), by the Yoneda lemma (see for example \cite[Section 3.10]{Kel05}). 
\begin{equation*}
    \int^{W\in\Jpq} E(W) \wedge \Jpq(W,-)\cong E.
\end{equation*}
Alternatively, we can describe a functor $E\in \Epq$ in terms of natural transformations, by the enriched Yoneda lemma.
\begin{equation*}
    E(W)\cong\Nat_{p,q}(\Jpq(W,-),E)=\int\limits_{V\in\Jpq} \Top_* (\Jpq(W,V),E(V)).
\end{equation*}

Another useful result, that we use throughout the thesis, is that $\Nat_{p,q}(-,F)$ sends homotopy cofibre sequences to homotopy fibre sequences. This follows from the fact that the functor $\Top_*(-,A):\C\Top_*\rightarrow \C\Top_*$ sends homotopy cofibre sequences to homotopy fibre sequences (it is contravariant, sends colimits to limits, and $\CTop_*$ is closed symmetric monoidal) and using the definition of $\Nat_{p,q}(-,F)$ as the end above. 

\end{remark}

\section{Derivatives}

Derivatives play a key role in calculus of real functions. They describe the difference between successive polynomial approximations in the Taylor series. As the name calculus suggests, one can define a notion of derivatives of functors in orthogonal calculus, as done by Weiss in \cite[Section 2]{Wei95} and Barnes and Oman in \cite[Section 4]{BO13}. In this section, we will extend this theory to the $\C$-equivariant setting. The derivatives of these functors play a key role in the classification of $(p,q)$-homogeneous functors, as the derivative adjunctions form one half of the zig-zag of equivalences between the $(p,q)$-homogeneous model structure and the category of orthogonal $\C$-spectra with an action of $O(p,q)$, see Theorem \ref{zigzagclassification}. 

Let $i_{p,q}^{l,m}:\R^{p,q}\rightarrow\R^{l,m}$\index{$i_{p,q}^{l,m}$} be the $\C$-equivariant inclusion map $(x,y)\mapsto (x,0,y,0)$, where $p\leq l$ and $q\leq m$. Such a map induces a group homomorphism $O(p,q)\rightarrow O(l,m)$, which is $O(p,q)$-equivariant by letting $O(p,q)$ act on the first $p$ and $q$ coordinates of $O(l,m)$. That is, both $\R^{p,q}$ and $\R^{l,m}$ are $(O(p,q)\rtimes \C)$-spaces. 

This map induces a map of $(O(p,q)\rtimes\C)$-equivariant spaces 
\begin{align*}
(i_{p,q}^{l,m} )_{U,V}:\Gmor{p,q}{U}{V}&\rightarrow\Gmor{l,m}{U}{V}\\
(f,x)&\mapsto(f,(i_{p,q}^{l,m}\otimes id)(x))   
\end{align*}
which in turn induces a map on the associated Thom spaces, and hence also on the $(O(p,q)\rtimes\C)\Top_*$-enriched categories $\CJ{p}{q}\rightarrow\CJ{l}{m}$. These maps form commutative diagrams of categories as follows.
% https://q.uiver.app/#q=WzAsNCxbMCwwLCJcXEpwcSJdLFsyLDAsIlxcQ1xcbWF0aGNhbHtKfV97cCsxLHF9Il0sWzAsMiwiXFxDXFxtYXRoY2Fse0p9X3twLHErMX0iXSxbMiwyLCJcXENcXG1hdGhjYWx7Sn1fe3ArMSxxKzF9Il0sWzAsMSwiaV97cCxxfV57cCsxLHF9Il0sWzEsMywiaV97cCsxLHF9XntwKzEscSsxfSJdLFswLDIsImlfe3AscX1ee3AscSsxfSIsMl0sWzIsMywiaV97cCxxKzF9XntwKzEscSsxfSIsMl0sWzAsMywiaV97cCxxfV57cCsxLHErMX0iLDJdXQ==
\[\begin{tikzcd}
	\Jpq && {\C\mathcal{J}_{p+1,q}} \\
	\\
	{\C\mathcal{J}_{p,q+1}} && {\C\mathcal{J}_{p+1,q+1}}
	\arrow["{i_{p,q}^{p+1,q}}", from=1-1, to=1-3]
	\arrow["{i_{p+1,q}^{p+1,q+1}}", from=1-3, to=3-3]
	\arrow["{i_{p,q}^{p,q+1}}"', from=1-1, to=3-1]
	\arrow["{i_{p,q+1}^{p+1,q+1}}"', from=3-1, to=3-3]
	\arrow["{i_{p,q}^{p+1,q+1}}"', from=1-1, to=3-3]
\end{tikzcd}\]
We can use these maps to define functors between the categories $\Epq$, and with the addition of an orbit functor we can do the same for the categories $\OEpq$.

\begin{definition}\index{$\res_{p,q}^{l,m}$}\index{$\res_{p,q}^{l,m}/O(l-p,m-q)$}
Let $p\leq l$ and $q\leq m$.

\noindent Define the \emph{restriction functor} $\res_{p,q}^{l,m}:\CE{l}{m} \rightarrow \CE{p}{q}$ as precomposition with $i_{p,q}^{l,m}$.

\noindent Define the \emph{restriction-orbit functor} by
\begin{align*}
\res_{p,q}^{l,m}/O(l-p,m-q):\OCE{l}{m}&\rightarrow\OCE{p}{q}\\
F&\mapsto (F\circ i_{p,q}^{l,m})/O(l-p,m-q).
\end{align*}

\end{definition}

\begin{remark}
For an $O(l,m)$-space $X$, the $O(p,q)$ action on $X/O(l-p,m-q)$ is given by $g[x]:=[gx]$, where $g\in O(p,q)$ and $x\in X$. This is a well defined action, since for $h\in  O(l-p,m-q)$ 
$$g[x]=[gx]\sim [hgx]=[ghx]=g[hx].$$Note that the restriction functor is often omitted from notation. 
\end{remark}

The restriction functors also form commutative diagrams, induced by the diagram above. 
% https://q.uiver.app/#q=WzAsNCxbMCwwLCJcXENcXG1hdGhjYWx7RX1fe3ArMSxxKzF9Il0sWzIsMCwiXFxDXFxtYXRoY2Fse0V9X3twKzEscX0iXSxbMCwyLCJcXENcXG1hdGhjYWx7RX1fe3AscSsxfSJdLFsyLDIsIlxcQ1xcbWF0aGNhbHtFfV97cCxxfSJdLFswLDEsIlxccmVzX3twKzEscX1ee3ArMSxxKzF9Il0sWzEsMywiXFxyZXNfe3AscX1ee3ArMSxxfSJdLFswLDIsIlxccmVzX3twLHErMX1ee3ArMSxxKzF9IiwyXSxbMiwzLCJcXHJlc197cCxxfV57cCxxKzF9IiwyXSxbMCwzLCJcXHJlc197cCxxfV57cCsxLHErMX0iLDJdXQ==
\[\begin{tikzcd}
	{\C\mathcal{E}_{p+1,q+1}} && {\C\mathcal{E}_{p+1,q}} \\
	\\
	{\C\mathcal{E}_{p,q+1}} && {\C\mathcal{E}_{p,q}}
	\arrow["{\res_{p+1,q}^{p+1,q+1}}", from=1-1, to=1-3]
	\arrow["{\res_{p,q}^{p+1,q}}", from=1-3, to=3-3]
	\arrow["{\res_{p,q+1}^{p+1,q+1}}"', from=1-1, to=3-1]
	\arrow["{\res_{p,q}^{p,q+1}}"', from=3-1, to=3-3]
	\arrow["{\res_{p,q}^{p+1,q+1}}"', from=1-1, to=3-3]
\end{tikzcd}\]

The restriction and restriction-orbit functors have right adjoints. Before we can define them, we must define an adjoint to the orbit functor, see \cite[Lemma 4.2]{BO13}. 

\begin{lemma}
Let $p\leq l$ and $q\leq m$. There is an adjoint pair 
\begin{equation*}
    (-)/O(l-p,m-q):(O(l,m)\rtimes\C)\Top_* \rightleftarrows \COTop_* : \CI^{l,m}_{p,q}.
\end{equation*}
The right adjoint $\CI^{l,m}_{p,q}$ is defined as follows. An $(O(p,q)\rtimes\C)$-space $A$ can be considered as an $((O(p,q)\times O(l-p,m-q))\rtimes\C)$-space, by letting $O(l-p,m-q)$ act trivially. Call this space $\varepsilon^*A$. Define $\CI^{l,m}_{p,q}A$ to be the space of $(O(p,q)\times O(l-p,m-q))$-equivariant maps from $O(l,m)\rtimes \C$ to $\varepsilon^*A$, which has the $(O(l,m)\rtimes\C)$-action induced by the conjugation $\C$-action and the action of $O(l,m)$ on itself.

\end{lemma}

\newpage
\begin{definition}\label{induction definition}\index{$\ind_{p,q}^{l,m}$}\index{$\ind_{p,q}^{l,m}\CI$}\index{$\ind_{0,0}^{l,m}\varepsilon^*$}
Let $p\leq l$ and $q\leq m$.

\noindent Define the \emph{induction functor} $\ind_{p,q}^{l,m}:\CE{p}{q} \rightarrow \CE{l}{m}$ by
\begin{equation*}
    \ind_{p,q}^{l,m}F:U\mapsto\Nat_{p,q}(\CJ{l}{m}(U,-),F),
\end{equation*}
where the space of natural transformations of objects of $\Epq$ is equipped with the conjugation $\C$-action (see Remark \ref{remNat}).
\noindent 

Define the \emph{inflation-induction functor} $\ind_{p,q}^{l,m}\CI:\OCE{p}{q} \rightarrow \OCE{l}{m}$ by
\begin{equation*}
    \ind_{p,q}^{l,m}\CI F:U\mapsto\Nat_{\OEpq}(\CJ{l}{m}(U,-),\CI_{p,q}^{l,m}\circ F).
\end{equation*}
When $p,q=0$, $\CI_{p,q}^{l,m}$ simply gives $F$ the trivial $O(l,m)$-action, hence we write $\ind_{0,0}^{l,m}\CI F$ as $\ind_{0,0}^{l,m}\varepsilon^*F$. This is what we call the $(l,m)$-th derivative of $F$, denoted by $$F^{(l,m)}:=\ind_{0,0}^{l,m}\varepsilon^* F.$$\index{$(-)^{(p,q)}$} 
\end{definition}

\begin{lemma}
The induction functor $\ind_{p,q}^{l,m}$ is right adjoint to the restriction functor $\res_{p,q}^{l,m}$.
The inflation-induction functor $\ind_{p,q}^{l,m}\CI$ is right adjoint to the restriction-orbit functor $\res_{p,q}^{l,m}/O(l-p,m-q)$.
\end{lemma}

\begin{proof}
We will prove the adjunction between the restriction and induction functors, leaving the restriction-orbit inflation-induction adjunction to the reader.  
{\allowdisplaybreaks
\begin{align*}
&\quad \C\Nat_{p,q}\left(\res_{p,q}^{l,m} E,F\right)\\
&=\int\limits_{V\in\Jpq} \CTop_*\left(\res_{p,q}^{l,m} E(V),F(V)\right)\\
&=\int\limits_{V\in\Jpq} \CTop_* \left([(i_{p,q}^{l,m})^*E](V),F(V)\right)\\
&\cong \int\limits_{V\in\Jpq} \CTop_* \left((i_{p,q}^{l,m})^*\left(\int\limits^{W\in\CJ{l}{m}}E(W)\wedge\Jmor{l,m}{W}{-}\right)(V),F(V)\right)\\
&\cong \int\limits_{V\in\Jpq} \CTop_* \left(\int\limits^{W\in\CJ{l}{m}}E(W)\wedge\bigl[(i_{p,q}^{l,m})^*\Jmor{l,m}{W}{-}\bigr](V),F(V)\right)\\
&\cong \int\limits_{V\in\Jpq} \int\limits_{W\in\CJ{l}{m}}\CTop_* \left(E(W)\wedge\bigl[(i_{p,q}^{l,m})^*\Jmor{l,m}{W}{-}\bigr](V),F(V)\right)\\
&\cong \int\limits_{W\in\CJ{l}{m}}\int\limits_{V\in\Jpq} \CTop_* \left(E(W), \Top_*\left(\bigl[(i_{p,q}^{l,m})^*\Jmor{l,m}{W}{-}\bigr](V),F(V)\right)\right)\\
&\cong \int\limits_{W\in\CJ{l}{m}}\CTop_* \left(E(W), \int\limits_{V\in\Jpq} \Top_*\left(\bigl[(i_{p,q}^{l,m})^*\Jmor{l,m}{W}{-}\bigr](V),F(V)\right)\right)\\
&\cong \int\limits_{W\in\CJ{l}{m}}\CTop_* \left(E(W),\Nat_{p,q}\left((i_{p,q}^{l,m})^*\Jmor{l,m}{W}{-},F\right)\right)\\
&\cong \int\limits_{W\in\CJ{l}{m}}\CTop_* \left(E(W),\ind_{p,q}^{l,m}F(W)\right)\\
&=\C\Nat_{l,m}\left(E,\ind_{p,q}^{l,m}F\right)
\end{align*}}

where $(i_{p,q}^{l,m})^*$ represents precomposition with $i_{p,q}^{l,m}$. Here we have made use of standard results of enriched ends and coends, including the Yoneda Lemma \cite[Lemma 6.3.5]{BR20}.
\end{proof}

As a result of the adjuction above and the commutative diagrams involving the restriction functors, there are commutative diagrams of categories
% https://q.uiver.app/#q=WzAsNCxbMCwwLCJcXENcXG1hdGhjYWx7RX1fe3AscX0iXSxbMiwwLCJcXENcXG1hdGhjYWx7RX1fe3ArMSxxfSJdLFswLDIsIlxcQ1xcbWF0aGNhbHtFfV97cCxxKzF9Il0sWzIsMiwiXFxDXFxtYXRoY2Fse0V9X3twKzEscSsxfSJdLFswLDEsIlxcaW5kX3twLHF9XntwKzEscX0iXSxbMSwzLCJcXGluZF97cCsxLHF9XntwKzEscSsxfSJdLFswLDIsIlxcaW5kX3twLHF9XntwLHErMX0iLDJdLFsyLDMsIlxcaW5kX3twLHErMX1ee3ArMSxxKzF9IiwyXSxbMCwzLCJcXGluZF97cCxxfV57cCsxLHErMX0iLDJdXQ==
\[\begin{tikzcd}
	{\C\mathcal{E}_{p,q}} && {\C\mathcal{E}_{p+1,q}} \\
	\\
	{\C\mathcal{E}_{p,q+1}} && {\C\mathcal{E}_{p+1,q+1}}
	\arrow["{\ind_{p,q}^{p+1,q}}", from=1-1, to=1-3]
	\arrow["{\ind_{p+1,q}^{p+1,q+1}}", from=1-3, to=3-3]
	\arrow["{\ind_{p,q}^{p,q+1}}"', from=1-1, to=3-1]
	\arrow["{\ind_{p,q+1}^{p+1,q+1}}"', from=3-1, to=3-3]
	\arrow["{\ind_{p,q}^{p+1,q+1}}"', from=1-1, to=3-3]
\end{tikzcd}\]

\begin{remark}
As in \cite{Wei95}, the induction functors give us a notion of differentiation of functors in our input category $\Ezero$ (see Definition \ref{jzero and ezero def}). In particular, for the $\C$-equivariant case, there are two directions in one can take a derivative; in the $p$ direction and in the $q$ direction. These two different directions of differentiating can be thought of as partial derivatives, and then the commuting diagram above tells us that taking both possible orders of mixed partial derivatives is the same as taking the total derivative. 
\end{remark}

We have already seen one relation between induction and the $(p,q)$-jet categories $\Jpq$ in Proposition \ref{cofibseq}. We now recreate another key result \cite[Theorem 2.2]{Wei95} in the $\C$-equivariant setting. The following proposition defines induction iteratively as a homotopy fibre, and acts as a tool for calculating the derivatives. 

\begin{proposition}\label{loops fibre sequence}
For all $U\in\Jzero$ and for all $F\in\Epq$, there are homotopy fibre sequences of $\C$-spaces
\begin{equation*}
\res_{p,q}^{p+1,q}\ind_{p,q}^{p+1,q} F(U)\rightarrow F(U)\rightarrow \Omega^{(p,q)\mathbb{R}}F(U\oplus \R)
\end{equation*}
and 
\begin{equation*}
\res_{p,q}^{p,q+1}\ind_{p,q}^{p,q+1} F(U)\rightarrow F(U)\rightarrow \Omega^{(p,q)\Rdelta}F(U\oplus \Rdelta),
\end{equation*}
\end{proposition}
where $\Omega^{(p,q)V}Y$ represents the space of maps $S^{(p,q)V}\rightarrow Y$, for a $\C$-space $Y$, and is given the conjugation $\C$-action. 

\begin{proof}
We prove the existence of the first fibre sequence and leave the similar second case to the reader.
From Proposition \ref{cofibseq} there exists a homotopy cofibre sequence  
\begin{equation*}
    \Jmor{p,q}{U\oplus \R}{-}\wedge S^{(p,q)\R} \rightarrow \Jmor{p,q}{U}{-}\rightarrow \Jmor{p+1,q}{U}{-}.
\end{equation*}
Pick $F\in\Epq$ and apply the contravariant functor $\Nat_{p,q}(-,F)$ to the cofibre sequence above. This yields a homotopy fibre sequence (see Remark \ref{remNat})

\begin{align*}
\Nat_{p,q}\left(\Jmor{p,q}{U\oplus \R}{-}\wedge S^{(p,q)\R},F\right)\leftarrow \Nat_{p,q}&\left(\Jmor{p,q}{U}{-},F\right)\\
&\leftarrow \Nat_{p,q}\left(\Jmor{p+1,q}{U}{-},F\right).
\end{align*}

Application of the enriched Yoneda Lemma and the definition of  $\ind_{p,q}^{p+1,q}$ gives the desired fibre sequence, since the functor $\Nat_{p,q}(-,F)$ preserves $\C$-equivariant maps. 
\end{proof}

\section{The $(p,q)$-stable model structure}

We want to compare the $(p,q)$-th intermediate category $\OEpq$ with the category of orthogonal $\C$-spectra with an action of $O(p,q)$. The $(p,q)$-stable model structure constructed will be a modification of the stable model structure on orthogonal $\C$-spectra, see \cite[Section 3.4]{MM02}. This modification will account for the fact that the structure maps of objects in $\OEpq$ are of the form 
\begin{equation*}
    \sigma_X:S^{(p,q)V}\wedge X(W)\rightarrow X(W\oplus V).
\end{equation*}
The structure maps $\sigma_X$\index{$\sigma_X$} of an object $X\in\OEpq$ are induced by the identification of $S^{(p,q)V}$ as a subspace of $\Jpq(W,W\oplus V)$ (see the proof of Proposition \ref{cofibseq}) and the structure maps of $X$ being an enriched functor. 

\begin{remark}
In the underlying calculus, there exists a description of the intermediate category $O(n)\mathcal{E}_n$ as a category of diagram spectra (see \cite[Part 1]{MMSS01}). An analogous description is also true in the $\C$-equivariant setting for the $(p,q)$-th intermediate category $\OEpq$. Since this description will not be used in the remainder of the thesis, we omit the details. The statement is analogous to the underlying calculus version of Barnes and Oman \cite[Lemma 7.3 and Proposition 7.4]{BO13} and checking that the maps used are equivariant uses the same method as Taggart in \cite[Lemma 5.12 and Proposition 5.13]{Tag22real}.
\end{remark}

We begin by defining four functors. These functors will form adjoint pairs that when composed give an adjunction between the categories $\CTop_*$ and $\OEpq$. Recall from Definition \ref{def:O(p,q) and matrix A} that $\beta:\C\rightarrow O(p,q)\rtimes \C$ is defined by $\alpha\mapsto (\Id_{p+q},\alpha)$.

\begin{definition}
Let $\beta^*$ be the restriction functor $\COTop_*\rightarrow \CTop_*$, which sends $X$ to the underlying space $X$ with $\C$-action given by $\sigma x= (\beta (\sigma)) x$. 

Let $\beta_!$ be the functor $\CTop_*\rightarrow \COTop_*$ defined by 
\begin{equation*}
    X\mapsto  (O(p,q)\rtimes \C)_+\wedge_{\C} X.
\end{equation*}

Let $\Jmor{p,q}{U}{-}\wedge (-)$ be the free functor $\COTop_*\rightarrow \OEpq$ defined by 
\begin{equation*}
    X\mapsto \Jmor{p,q}{U}{-}\wedge X.
\end{equation*}
    
Let $\Ev_U$ be the evaluation at $U\in \Jzero$ functor $\OEpq\rightarrow \COTop_*$ defined by 
\begin{equation*}
   F\mapsto F(U).
\end{equation*}
\end{definition}

\begin{proposition}\label{adjunctions}
The restriction functor $\beta^*$ is right adjoint to the functor $\beta_!$. 
The evaluation at $U$ functor $\Ev_U$ is right adjoint to the free functor $\Jmor{p,q}{U}{-}\wedge (-)$. 
The right adjoints commute with colimits and hence pushouts. 
\end{proposition}

\begin{proof}
The adjunction between the restriction functor $\beta^*$ and the functor $\beta_!$ is well known (see Proposition \ref{prop: mm3.1.2}). The following argument proves the second adjoint pair,
{\allowdisplaybreaks
\begin{align*}
&\quad \OEpq (\Jmor{p,q}{U}{-}\wedge X, F) \\
&= \int\limits_{W\in \Jpq} \COTop_*(\Jmor{p,q}{U}{W}\wedge X, F(W))\\
&\cong \int\limits_{W\in \Jpq} \COTop_*( X, \Top_*(\Jmor{p,q}{U}{W},F(W)))\\
&= \COTop_*\left( X,  \int\limits_{W\in \Jpq} \Top_*(\Jmor{p,q}{U}{W},F(W))\right)\\
&\cong \COTop_*( X, F(U))
\end{align*}}
where $X\in \COTop_*$ and $F\in\OEpq$. Here we have used a standard smash product adjunction along with the Yoneda Lemma, and $(O(p,q)\rtimes \C)$ acts on $\Top_*(\Jmor{p,q}{U}{W},F(W))$ by conjugation. 

The right adjoints $\beta^*$ and $\Ev_U$ commute with colimits, since colimits in $\C$-spaces are constructed in spaces and then given a $\C$-action, and colimits in $\OEpq$ are constructed objectwise. 
\end{proof}

There is a projective model structure on the intermediate categories $\OEpq$, similar to the levelwise model structure constructed by Barnes and Oman \cite[Lemma 7.6]{BO13}, in which fibrations and weak equivalences are defined objectwise. A left Bousfield localisation of this model structure will give the $(p,q)$-stable model structure. This projective model structure is exactly the level model structure of \cite[Section 6]{MMSS01}. 

\begin{definition}\label{def: obj WE}
Let $f:X\rightarrow Y$ be a map in $\OEpq$. Call $f$ an \emph{objectwise fibration} or an \emph{objectwise weak equivalence} if $\beta^*(f(U)):\beta^*(X(U))\rightarrow \beta^*(Y(U))$ is a fibration or weak equivalence of pointed $\C$-spaces, for each $U\in \Jzero$. Call $f$ a \emph{cofibration} if it has the left lifting property with respect to the objectwise acyclic fibrations. Denote the collection of objectwise weak equivalences by $W_{level}$. 
\end{definition}

Now we define two sets of maps in the $(p,q)$-th intermediate category $\OEpq$.

\begin{definition}\index{$\Ilevel$}\index{$\Jlevel$}
Define sets $\Ilevel$ and $\Jlevel$ in $\OEpq$ by 
\begin{align*}
    &\Ilevel=\{\Jmor{p,q}{U}{-}\wedge  \beta_!(i) : U\in\Jzero, i\in I_{\C}\}\\
    &\Jlevel=\{\Jmor{p,q}{U}{-}\wedge \beta_!(j) : U\in\Jzero, j\in J_{\C}\}
\end{align*}
where $I_{\C},J_{\C}$ are the generating cofibrations and acyclic cofibrations of the fine model structure on $\CTop_*$ (see Proposition \ref{finemodelstructure}).
\end{definition}

For a class of maps $K$ in a category $\mathcal{C}$, let $K$-inj denote the class of maps that have the right lifting property with respect to every map in $K$ (see \cite[Definition 2.1.7]{Hov99}).

\begin{proposition}\label{inj}
$\Ilevel$-inj is the class of objectwise acyclic fibrations and $\Jlevel$-inj is the class of objectwise fibrations. 
\end{proposition}

\begin{proof}
Let $U\in\Jzero$, $i\in I_{\C}$ and consider a diagram 
% https://q.uiver.app/#q=WzAsNCxbMCwwLCJDXzJcXG1hdGhjYWx7Sn1fe3AscX0oVSwtKVxcd2VkZ2UgXFxiZXRhXyooQ18yL0hfK1xcd2VkZ2UgU157bi0xfV8rKSJdLFsyLDAsIlgiXSxbMCwyLCJDXzJcXG1hdGhjYWx7Sn1fe3AscX0oVSwtKVxcd2VkZ2UgXFxiZXRhXyooQ18yL0hfK1xcd2VkZ2UgRF57bn1fKykiXSxbMiwyLCJZIl0sWzAsMV0sWzAsMiwiSSIsMl0sWzEsM10sWzIsM11d
\[\begin{tikzcd}
	{C_2\mathcal{J}_{p,q}(U,-)\wedge \beta_!(C_2/H_+\wedge S^{n-1}_+)} && X \\
	\\
	{C_2\mathcal{J}_{p,q}(U,-)\wedge \beta_!(C_2/H_+\wedge D^{n}_+)} && Y
	\arrow[from=1-1, to=1-3]
	\arrow["\Jmor{p,q}{U}{-}\wedge  \beta_!(i)"', from=1-1, to=3-1]
	\arrow[from=1-3, to=3-3]
	\arrow[from=3-1, to=3-3]
\end{tikzcd}\]
Using the adjunctions of Proposition \ref{adjunctions}, the square above has a lift if and only if the following square lifts in $\COTop_*$.
% https://q.uiver.app/#q=WzAsNCxbMCwwLCIgXFxiZXRhXyooQ18yL0hfK1xcd2VkZ2UgU157bi0xfV8rKSJdLFsyLDAsIlgoVSkiXSxbMCwyLCIgXFxiZXRhXyooQ18yL0hfK1xcd2VkZ2UgRF57bn1fKykiXSxbMiwyLCJZKFUpIl0sWzAsMV0sWzAsMl0sWzEsM10sWzIsM11d
\[\begin{tikzcd}
	{ \beta_!(C_2/H_+\wedge S^{n-1}_+)} && {X(U)} \\
	\\
	{ \beta_!(C_2/H_+\wedge D^{n}_+)} && {Y(U)}
	\arrow[from=1-1, to=1-3]
	\arrow["\beta_!(i)"',from=1-1, to=3-1]
	\arrow[from=1-3, to=3-3]
	\arrow[from=3-1, to=3-3]
\end{tikzcd}\]
Again, by adjunctions, the square above has a lift if and only if the following square lifts in $\CTop_*$.
% https://q.uiver.app/#q=WzAsNCxbMCwwLCJDXzIvSF8rXFx3ZWRnZSBTXntuLTF9XysiXSxbMiwwLCJcXGJldGFeKihYKFUpKSJdLFswLDIsIkNfMi9IXytcXHdlZGdlIERee259XysiXSxbMiwyLCJcXGJldGFeKihZKFUpKSJdLFswLDFdLFswLDJdLFsxLDNdLFsyLDNdXQ==
\[\begin{tikzcd}
	{C_2/H_+\wedge S^{n-1}_+} && {\beta^*(X(U))} \\
	\\
	{C_2/H_+\wedge D^{n}_+} && {\beta^*(Y(U))}
	\arrow[from=1-1, to=1-3]
	\arrow["i"',from=1-1, to=3-1]
	\arrow[from=1-3, to=3-3]
	\arrow[from=3-1, to=3-3]
\end{tikzcd}\]
By the fine model structure on $\CTop_*$ (Proposition \ref{finemodelstructure}), the above square lifts if and only if the map $\beta^*(X(U))\rightarrow \beta^*(Y(U))$ is an acyclic fibration of $\C$-spaces. Thus, the first diagram lifts if and only if the map $X\rightarrow Y$ is an objectwise acyclic fibration. 

A similar proof gives the $\Jlevel$ case. 
\end{proof}

\begin{lemma}\label{objectwise model structure proof}\index{$\OEpql$}
There is a cellular, proper, $\C$-topological model structure on the $(p,q)$-th intermediate category $\OEpq$ formed by the objectwise weak equivalences and objectwise fibrations. Denote this model category by $\OEpql$ and call it the projective model structure on $\OEpq$. The generating cofibrations and generating acyclic cofibrations are given by $\Ilevel$ and $\Jlevel$ respectively. 
\end{lemma}

\begin{proof}
We will use the recognition theorem of Hovey \cite[Theorem 2.1.19]{Hov99} to prove the existence of the model structure. 

1. (2 out of 3) Objectwise weak equivalences clearly have the 2 out of 3 property, since they are defined objectwise. 

2. (Smallness) Let $B$ be small in $\CTop_*$ and $X_i$ in $\Ilevel$-cell (respectively $\Jlevel$-cell), where $\Ilevel$-cell (respectively $\Jlevel$-cell) denotes the collection of transfinite compositions of pushouts of elements of $\Ilevel$ (respectively $\Jlevel$). Let $\phi$ be a map 
\begin{align*}
    \phi: \colim_i \OEpq(&\Jmor{p,q}{U}{-}\wedge O(p,q)_+\wedge B, X_i)\\
    &\rightarrow \OEpq(\Jmor{p,q}{U}{-}\wedge O(p,q)_+\wedge B, \colim_iX_i).
\end{align*}
We let $S$ and $R$ denote the domain and codomain of $\phi$ respectively to save space. Using the adjunctions of Proposition \ref{adjunctions} gives a commutative diagram 
% https://q.uiver.app/#q=WzAsNSxbMCwxLCJcXGNvbGltX2lcXENcXFRvcF8qKEIsXFxiZXRhXiooWF9pKFUpKSkiXSxbMiwxLCJcXENcXFRvcF8qKEIsXFxiZXRhXiooXFxjb2xpbV9pIFhfaShVKSkpIl0sWzIsMiwiXFxDXFxUb3BfKihCLFxcY29saW1faShcXGJldGFeKihYX2koVSkpKSkiXSxbMCwwLCJTIl0sWzIsMCwiUiJdLFsxLDIsIlxcY29uZyJdLFswLDFdLFswLDIsIlxcY29uZyIsMl0sWzMsMCwiXFxjb25nIiwyXSxbNCwxLCJcXGNvbmciXSxbMyw0LCJcXHBoaSJdXQ==
\[\begin{tikzcd}
	S && R \\
	{\colim_i\C\Top_*(B,\beta^*(X_i(U)))} && {\C\Top_*(B,\beta^*(\colim_i X_i(U)))} \\
	&& {\C\Top_*(B,\colim_i(\beta^*(X_i(U))))}
	\arrow["\phi", from=1-1, to=1-3]
	\arrow["\cong"', from=1-1, to=2-1]
	\arrow["\cong", from=1-3, to=2-3]
	\arrow[from=2-1, to=2-3]
	\arrow["\cong"', from=2-1, to=3-3]
	\arrow["\cong", from=2-3, to=3-3]
\end{tikzcd}\]
where the diagonal map is an isomorphism because $B$ is small in $\CTop_*$, and the bottom right vertical map is an isomorphism because the right adjoints commute with colimits by Proposition \ref{adjunctions}. Hence, the map $\phi$ is an isomorphism as required.

3. ($\Ilevel$-inj $= \Jlevel$-inj $\cap$ $W_{level}$) This is clear by Proposition \ref{inj}. 

4. ($\Jlevel$-cell $\subseteq \Ilevel$-cof $\cap$ $W_{level}$) Since $J_{\C}\subseteq I_{\C}$-cof, $\Jlevel \subseteq \Ilevel$-cof. Therefore $\Jlevel\text{-cell} \subseteq \Ilevel\text{-cof}$.

It remains to show that $\Jlevel$-cell $\subseteq$ $W_{level}$. Since $\Jlevel\subseteq \Jlevel$-cof, each map $f\in\Jlevel$ is such that $\beta^*f(U)$ is an acyclic cofibration, for each $U\in\Jzero$. Now consider a pushout square
% https://q.uiver.app/#q=WzAsNCxbMCwwLCJBIl0sWzIsMCwiQiJdLFswLDIsIkMiXSxbMiwyLCJEIl0sWzAsMV0sWzIsM10sWzAsMiwiZiIsMl0sWzEsMywiayJdXQ==
\[\begin{tikzcd}
	A && B \\
	\\
	C && D
	\arrow[from=1-1, to=1-3]
	\arrow[from=3-1, to=3-3]
	\arrow["f"', from=1-1, to=3-1]
	\arrow["k", from=1-3, to=3-3]
\end{tikzcd}\]
where $f\in\Jlevel$.
Since the right adjoints $\beta^*$ and $\Ev_U$ commute with pushouts (see Proposition \ref{adjunctions}), $\beta^*(k(U))$ is the pushout of $\beta^*(f(U))$, for each $U\in\Jzero$. Since $\beta^*(f(U))$ is an acyclic cofibration and $\CTop_*$ is a model category, the pushout $\beta^*(k(U))$ is also an acyclic cofibration (see \cite[Corollary 1.1.11]{Hov99}). In particular, this means that the map $k$ is an objectwise weak equivalence. 

Now consider a diagram
% https://q.uiver.app/#q=WzAsNCxbMCwwLCJYXzAiXSxbMiwwLCJYXzEiXSxbNCwwLCJYXzIiXSxbNiwwLCJcXGRvdHMiXSxbMCwxLCJrXzAiXSxbMSwyLCJrXzEiXSxbMiwzLCJrXzIiXV0=
\[\begin{tikzcd}
	{X_0} && {X_1} && {X_2} && \dots
	\arrow["{k_0}", from=1-1, to=1-3]
	\arrow["{k_1}", from=1-3, to=1-5]
	\arrow["{k_2}", from=1-5, to=1-7]
\end{tikzcd}\]
where each $k_i$ is a pushout of a map in $\Jlevel$. Since the right adjoints $\beta^*$ and $\Ev_U$ commute with colimits (see Proposition \ref{adjunctions}) and the maps $\beta^*(k_i(U))$ are acyclic cofibrations by above, the map $\beta^*(\alpha (U))$ is a weak equivalence in $\CTop_*$, where $\alpha : X_0\rightarrow\colim_i X_i$. Therefore, $\alpha$ is an objectwise weak equivalence. 

(Properness) The functors $\beta^*$ and $\Ev_U$ preserve weak equivalences, fibrations and cofibrations. Then, since $\beta^*$ and $\Ev_U$ commute with pushouts and pullbacks, properness follows from taking adjoints of the appropriate diagrams. 
\end{proof}

\begin{remark}\label{remark: model structure on semi direct}
A similar method using only the adjunction 
% https://q.uiver.app/#q=WzAsMixbMCwwLCJDXzJcXFRvcF8qIl0sWzIsMCwiKE8ocCtxKVxccnRpbWVzIENfMilcXFRvcF8qIl0sWzAsMSwiXFxiZXRhXyoiLDAseyJvZmZzZXQiOi0yfV0sWzEsMCwiXFxiZXRhXioiLDAseyJvZmZzZXQiOi0yfV1d
\[\begin{tikzcd}
	{C_2\Top_*} && {(O(p,q)\rtimes C_2)\Top_*}
	\arrow["{\beta_!}", shift left=2, from=1-1, to=1-3]
	\arrow["{\beta^*}", shift left=2, from=1-3, to=1-1]
\end{tikzcd}\]
shows that there exists a cellular, proper, $\C$-topological model structure on the category $(O(p,q)\rtimes \C)\Top_*$, where weak equivalences and fibrations are defined by restricting to $\CTop_*$ along $\beta^*$. The generating (acyclic) cofibrations are of the form $\beta_! (i)$ where $i$ is a generation (acyclic) cofibration of $\CTop_*$. This is exactly the model structure which is coarse with respect to $O(p,q)$ and fine with respect to $\C$ (see Remark \ref{rem: what is the ms on semi direct prod}). In our chosen notation for these model structures, this could be denoted by $\C\Top_*[O(p,q)]$, however we chose to denote it by $\COTop_*$ to remind ourselves of the underlying group. As such, the projective model structure could alternatively be constructed by evaluating at $U$ and using this model structure on $(O(p,q)\rtimes \C)$-spaces. 
\end{remark}

We will use this projective model structure to construct the $(p,q)$-stable model structure using a left Bousfield localisation. In the same was as for $\C$-spectra \cite[Chapter 3]{MM02}, we begin by first defining homotopy groups on objects of $\OEpq$. These homotopy groups detect the weak equivalences of the $(p,q)$-stable model structure.  

\begin{definition}\label{def: pq pi equiv}\index{$(p,q)\pi_k^H (-)$}
Define the \emph{$(p,q)$-homotopy groups} of $X\in\OEpq$ by 
\begin{equation*}
(p,q)\pi_k^H X =
\left\{
	\begin{array}{ll}
		\colim_V \pi_k\left( \Omega^{(p,q)V} X(V)\right)^H,  & \mbox{if } k \geq 0 \\
		\colim_{V\supset \mathbb{R}^{|k|}} \pi_0\left( \Omega^{p,q({V-\mathbb{R}^{|k|}})} X(V)\right)^H,& \mbox{if } k < 0
	\end{array}
\right.
\end{equation*} 
where $V$ runs over the indexing $\C$-representations in $\CL$, $H\leq\C$ is a closed subgroup, and $V-\mathbb{R}^{|k|}$ denotes the orthogonal complement of $\mathbb{R}^{|k|}$ in $V$. Define a map $f:X\rightarrow Y$ in $\OEpq$ to be a $(p,q)\pi_*$-equivalence if the map $(p,q)\pi_k^H f: (p,q)\pi_k^H X\rightarrow (p,q)\pi_k^H Y$ is an isomorphism for all $k$ and all closed subgroups $H\leq \C$. 
\end{definition}

One can easily verify that if $\C$ was replaced by the trivial group, the $(p,q)$-homotopy groups for $p+q=n$ are exactly the $n$-homotopy groups defined in \cite[Definition 7.7]{BO13}. The $(p,q)\pi_*$-equivalences will be the weak equivalences in our stable model structure. 

\begin{lemma}\label{Level equiv is pi equiv}
An objectwise weak equivalence in $\OEpq$ is a $(p,q)\pi_*$-equivalence. 
\end{lemma}
\begin{proof}
One can show that if $\beta^*(f(V)):\beta^*(X(V))\rightarrow \beta^*(Y(V))$ is a weak equivalence of $\C$-spaces, then the induced map $\Top_* (S^V,\beta^*(X(V)))\rightarrow \Top_* (S^V,\beta^*(Y(V)))$ is a weak equivalence of $\C$-spaces (see \cite[Lemma 3.3]{MM02}). The lemma then follows. 
\end{proof}

Now we want to identify the fibrant objects of the $(p,q)$-stable model structure. These are a generalisation of $\Omega$-spectra, which are the fibrant objects of the stable model structure on orthogonal spectra (see Barnes and Roitzheim \cite[Corollary 5.2.17]{BR20}). They are defined analogously to the $n\Omega$-spectra of orthogonal calculus \cite[Definition 7.9]{BO13}.

\begin{definition}\label{def: pq omega spectrum}
An object $X$ of $\OEpq$ has structure maps $$\sigma_X:S^{(p,q)V}\wedge X(W)\rightarrow X(W\oplus V)$$induced by the identification of $S^{(p,q)V}$ as a subspace of $\Jpq(W,W\oplus V)$ (see the proof of Proposition \ref{cofibseq}) and the structure maps of $X$ being an enriched functor. The object $X$ is called a \emph{$(p,q)\Omega$-spectrum} if its adjoint structure maps $$\tilde{\sigma}_X:X(W)\rightarrow \Omega^{(p,q)V} X(W\oplus V)$$\index{$\tilde{\sigma}_X$}are weak equivalences of $\C$-spaces, for all $V,W\in \Jzero$. \end{definition}

\begin{lemma}\label{omega spectra lemma}
$X$ is a $(p,q)\Omega$-spectrum if and only if for all $W\in\Jzero$ the maps
\begin{align*}
    X(W)&\rightarrow\Omega^{(p,q)\R}X(W\oplus\R)\\
    X(W)&\rightarrow\Omega^{(p,q)\Rdelta}X(W\oplus\Rdelta)
\end{align*}
are weak equivalences of $\C$-spaces.
\end{lemma}
\begin{proof}
If $X$ is a $(p,q)\Omega$-spectrum, then clearly both maps are weak equivalences by setting $V=\R$ and $V=\Rdelta$ in Definition \ref{def: pq omega spectrum}. 

If $X$ is such that the two maps are weak equivalences, then $X$ being a $(p,q)\Omega$-spectrum follows by repeated application of the weak equivalences, as demonstrated in the diagram below. 
% https://q.uiver.app/?q=WzAsMTAsWzAsMCwiWChWKSJdLFsxLDAsIlxcT21lZ2Fee3AscVxcbWF0aGJie1J9fVgoVlxcb3BsdXNcXG1hdGhiYntSfSkiXSxbMywwLCJcXE9tZWdhXntwLHFcXG1hdGhiYntSfV5tfVgoVlxcb3BsdXNcXG1hdGhiYntSfV5tKSJdLFswLDEsIlxcT21lZ2Fee3AscVxcbWF0aGJie1J9XlxcZGVsdGF9WChWXFxvcGx1c1xcbWF0aGJie1J9XlxcZGVsdGEpIl0sWzAsMywiXFxPbWVnYV57cCxxXFxtYXRoYmJ7Un1ee25cXGRlbHRhfX1YKFZcXG9wbHVzXFxtYXRoYmJ7Un1ee25cXGRlbHRhfSkiXSxbMSwxLCJcXE9tZWdhXntwLHFcXG1hdGhiYntSfV57MSwxfX1YKFZcXG9wbHVzXFxtYXRoYmJ7Un1eezEsMX0pIl0sWzMsMywiXFxPbWVnYV57cCxxXFxtYXRoYmJ7Un1ee20sbn19WChWXFxvcGx1c1xcbWF0aGJie1J9XnttLG59KSJdLFswLDIsIlxcdmRvdHMiXSxbMiwyLCJcXGRkb3RzIl0sWzIsMCwiXFxkb3RzIl0sWzAsMSwiXFxzaW1lcSJdLFswLDMsIlxcc2ltZXEiLDJdLFszLDUsIlxcc2ltZXEiLDJdLFs3LDQsIlxcc2ltZXEiLDJdLFsxLDUsIlxcc2ltZXEiXSxbMiw2LCJcXHNpbWVxIiwwLHsic3R5bGUiOnsiYm9keSI6eyJuYW1lIjoiZG90dGVkIn19fV0sWzQsNiwiXFxzaW1lcSIsMix7InN0eWxlIjp7ImJvZHkiOnsibmFtZSI6ImRvdHRlZCJ9fX1dLFs5LDIsIlxcc2ltZXEiXSxbMSw5LCJcXHNpbWVxIl0sWzMsNywiXFxzaW1lcSIsMl1d
\[\begin{tikzcd}
	{X(W)} & {\Omega^{(p,q)\mathbb{R}}X(W\oplus\mathbb{R})} & \dots & {\Omega^{(p,q)\mathbb{R}^m}X(W\oplus\mathbb{R}^m)} \\
	{\Omega^{(p,q)\mathbb{R}^\delta}X(W\oplus\mathbb{R}^\delta)} & {\Omega^{(p,q)\mathbb{R}^{1,1}}X(W\oplus\mathbb{R}^{1,1})} \\
	\vdots && \ddots \\
	{\Omega^{(p,q)\mathbb{R}^{n\delta}}X(W\oplus\mathbb{R}^{n\delta})} &&& {\Omega^{(p,q)\mathbb{R}^{m,n}}X(W\oplus\mathbb{R}^{m,n})}
	\arrow["\simeq", from=1-1, to=1-2]
	\arrow["\simeq"', from=1-1, to=2-1]
	\arrow["\simeq"', from=2-1, to=2-2]
	\arrow["\simeq"', from=3-1, to=4-1]
	\arrow["\simeq", from=1-2, to=2-2]
	\arrow["\simeq", dotted, from=1-4, to=4-4]
	\arrow["\simeq"', dotted, from=4-1, to=4-4]
	\arrow["\simeq", from=1-3, to=1-4]
	\arrow["\simeq", from=1-2, to=1-3]
	\arrow["\simeq"', from=2-1, to=3-1]
\end{tikzcd}\]
\end{proof}

The following Lemma is a partial converse to Lemma \ref{Level equiv is pi equiv}. 

\begin{lemma}\label{BOLem7.10}
A $(p,q)\pi_*$-equivalence between $(p,q)\Omega$-spectra is an objectwise weak equivalence. 
\end{lemma}

\begin{proof}
The proof is identical to that of Mandell and May \cite[Section 3.9]{MM02}, where in the inductive steps we only need to consider the cases where $H\in\{e,\C\}$, since the only closed subgroups of $\C$ are $e$ and $\C$ itself. 
\end{proof}

%%%stable equivalences

We now want to identify the class of maps that will be used in the left Bousfield localisation. Let 
\begin{equation*}
    \lambda_{V,W}^{p,q} :\Jmor{p,q}{W\oplus V}{-}\wedge S^{(p,q)W} \rightarrow \Jmor{p,q}{V}{-}
\end{equation*}
be the restricted composition map, where $S^{(p,q)W}$ has been $\C$-equivariantly identified with the closure of the subspace of pairs $(i,x)\in \Jmor{p,q}{V}{W\oplus V}$ with $i$ the standard inclusion (see the proof of Proposition \ref{cofibseq}).

The maps $\lambda_{V,W}^{p,q}$ are $(p,q)\pi_*$-equivalences, by a similar argument as in the non-equivariant case \cite[Lemma 7.12]{BO13}. 

\begin{lemma}
    The maps $\lambda_{V,W}^{p,q}$ are $(p,q)\pi_*$-equivalences.
\end{lemma}
\begin{proof}
    As in the non-equivariant case, by fixing a linear isometry $W\oplus V\rightarrow U$, the map $\lambda_{V,W}^{p,q}$ can be written
    \begin{equation*}
        \lambda_{V,W}^{p,q}(U): O(U)_+\wedge_{O(U-V-W)} S^{(p,q)(U-V)}\rightarrow O(U)_+\wedge_{O(U-V)} S^{(p,q)(U-V)}.
    \end{equation*}
    This map is a $(p,q)\pi_*$-equivalence if and only if its suspension by $(p,q)V$ is (see for example \cite[Theorem III.3.7]{MM02}). Therefore, we only need to consider the map 
    \begin{equation*}
        \Sigma^{(p,q)V}\lambda_{V,W}^{p,q}(U): O(U)/O(U-V-W)_+\wedge S^{(p,q)U}\rightarrow O(U)/O(U-V)_+\wedge S^{(p,q)U}.
    \end{equation*}
    This map is $v$-connected, for the dimension function $v$ defined by 
    \begin{align*}
        v(e)&=(p+q+1)\dim U -\dim V -\dim W\\
        v(\C)&=(p+1)\dim U^{\C} + (q+1)\dim (U^{\C})^\perp -\dim (W\oplus V)^{\C} -\dim ((W\oplus V)^{\C})^\perp,
    \end{align*}
    where $(U^{\C})^\perp$ denotes the orthogonal complement of $U^{\C}$. That is, $(\Sigma^{(p,q)V}\lambda_{V,W}^{p,q}(U))^{e}$ is $v(e)$-connected, and $(\Sigma^{(p,q)V}\lambda_{V,W}^{p,q}(U))^{\C}$ is $v(\C)$-connected. 
    
    When we take $(p,q)\pi_k^H$-homotopy groups, the dimension of $U$ (and in turn the dimensions of $U^{\C}$ and $(U^{\C})^\perp$) increases in the colimit, and we get an isomorphism of homotopy groups. 
\end{proof}

Now we will follow the same procedure as Mandell and May \cite[Section 3.4]{MM02} to turn these maps into cofibrations, in order to make generating sets for the $(p,q)$-stable model structure.

%and stable equivalences 

Let $M\lambda_{V,W}^{p,q}$ be the mapping cylinder of $\lambda_{V,W}^{p,q}$. Then the map $\lambda_{V,W}^{p,q}$ can be factored as a cofibration $k_{V,W}^{p,q}$ and a deformation retract $r_{V,W}^{p,q}$ as follows. 
\begin{equation*}
    \Jmor{p,q}{W\oplus V}{-}\wedge S^{p,q W}\xrightarrow{k_{V,W}^{p,q}} M\lambda_{V,W}^{p,q} \xrightarrow{r_{V,W}^{p,q}} \Jmor{p,q}{V}{-}
\end{equation*}

\begin{definition}\index{$\Jstable$}
Define $\Jstable=\Jlevel\cup \{i\square k_{V,W}^{p,q} : i\in I_{\C}\text{ and } V,W\in \Jzero\}$, where $f\square g $ denotes the pushout product of two maps $f:A\rightarrow B$ and $g:X\rightarrow Y$, which is defined by 
\begin{equation*}
    f\square g: A\wedge Y \coprod_{A\wedge X} B\wedge X \rightarrow B\wedge Y.
\end{equation*}
\end{definition}

The next lemma classifies the fibrations of the $(p,q)$-stable model structure. The proof is identical to that of Mandell and May \cite[Proposition 4.8]{MM02}, so we can omit it here.
\begin{lemma}
A map $f:E\rightarrow B$ in $\OEpq$ has the right lifting property with respect to $\Jstable$ if and only if $f$ is an objectwise fibration and the diagram
% https://q.uiver.app/#q=WzAsNCxbMCwwLCJFKFYpIl0sWzIsMCwiXFxPbWVnYV57cCxxV31FKFZcXG9wbHVzIFcpIl0sWzAsMiwiQihWKSJdLFsyLDIsIlxcT21lZ2Fee3AscVd9QihWXFxvcGx1cyBXKSJdLFswLDEsIlxcdGlsZGV7XFxzaWdtYX1fRSJdLFswLDIsImYoVikiLDJdLFsyLDMsIlxcdGlsZGV7XFxzaWdtYX1fQiIsMl0sWzEsMywiXFxPbWVnYV57cCxxV31mKFZcXG9wbHVzIFcpIl1d
\[\begin{tikzcd}
	{E(V)} && {\Omega^{(p,q)W}E(V\oplus W)} \\
	\\
	{B(V)} && {\Omega^{(p,q)W}B(V\oplus W)}
	\arrow["{\tilde{\sigma}_E}", from=1-1, to=1-3]
	\arrow["{f(V)}"', from=1-1, to=3-1]
	\arrow["{\tilde{\sigma}_B}"', from=3-1, to=3-3]
	\arrow["{\Omega^{(p,q)W}f(V\oplus W)}", from=1-3, to=3-3]
\end{tikzcd}\]

\noindent is a homotopy pullback for all $V,W$. 
\end{lemma}

\begin{proposition}\label{prop: stable ms}\index{$\OEpq^s$}
There is a cofibrantly generated, proper, cellular $\C$-topological model structure on the $(p,q)$-th intermediate category $\OEpq$ called the $(p,q)$-stable model structure. The cofibrations are the same as for the projective model structure, the weak equivalences are the $(p,q)\pi_*$-equivalences, and the fibrant objects are the $(p,q)\Omega$-spectra. The generating cofibrations and generating acyclic cofibrations are the sets $\Ilevel$ and $\Jstable$ respectively. Denote this model category by $\OEpq^s$.
\end{proposition}

\begin{proof}
Letting $\lambda$ be the class of maps $\lambda_{V,W}^{p,q}$. By a theorem of Hirschhorn \cite[Theorem 4.1.1]{Hir03}, the fibrant objects of the left Bousfield localisation $L_{\lambda}\OEpql$ are the $\lambda$-local objects. It can easily be seen that a $\lambda$-local object is the same as a $(p,q)\Omega$-spectrum. It follows that the $(p,q)$-stable model structure on $\OEpq$ is exactly the left Bousfield localisation of the projective model structure with respect to the class of maps $\lambda$, since the fibrant objects and cofibrations are the same. 
\end{proof}

We get the following corollary as an application of \cite[Proposition 3.3.4]{Hir03}.
\begin{corollary}
There is a Quillen adjunction 
\begin{equation*}
    \Id: \OEpq^l \rightleftarrows \OEpq^s:\Id.
\end{equation*}
\end{corollary}

\chapter{Equivariant polynomial functors}
\label{ch:EquivPolyFunct}

\fancyhf{}
\fancyhead[C]{\rightmark}
\fancyhead[R]{\thepage}

In differential calculus, polynomial functions and derivatives are used to approximate real functions via the Taylor series. In orthogonal calculus, Weiss defines a class of input functors with properties analogous to those of polynomial functions, called polynomial functors, see \cite[Section 5]{Wei95}. These polynomial functors can be constructed into a tower that approximates a given functor, much like the Taylor series does for functions. It is the fibres of the maps between these polynomial approximation functors that are classified as spectra by the classification theorem, see \cite[Theorem 9.1]{Wei95} and \cite[Theorem 10.3]{BO13}. 

One can define a class of functors with polynomial properties in the $\C$-equivariant input category $\Ezero$. The addition of the $\C$-action makes it necessary to introduce an indexing shift from the underlying calculus. In particular, $\tau_n$ in the underlying calculus is defined using the poset $\{0\neq U\subseteq \R^{n+1}\}$ and $\taupq$ in the $\C$-calculus is defined using the poset $\{0\neq U\subseteq \R^{p,q}\}$.  The author introduces a new class of `strongly' polynomial functors to account for this indexing shift, see Definition \ref{def: polynomial}.

\section{Polynomial functors}
In this section, we will adapt the definition of polynomial functors from the underlying calculus to fit the new $\C$-equivariant categories defined in Chapter \ref{ch:EquivFunct}. These functors should be input functors (see Definition \ref{jzero and ezero def}), which have properties that mimic those of polynomial functions.    

We begin by defining a functor $\taupq$ on the input category. The functor $\taupq$ is analogous to the functor $\tau_n$ from the underlying calculus, but requires a slightly different indexing due to the group action. The $(p,q)$-th complement bundle $\Gmor{p,q}{U}{V}$ has an associated sphere bundle $S\Gmor{p,q}{U}{V}$\index{$S\Gmor{p,q}{U}{V}$}. Considering $S\Gmor{p,q}{-}{-}:\Jzero^{op}\times \Jzero\rightarrow \CTop_*$ as a $\C\Top_*$-enriched functor, we can define the functor $\taupq:\Ezero\rightarrow \Ezero$ as follows.
\begin{definition}\index{$\taupq$}
Let $E\in \Ezero$. Define the functor $\taupq E\in\Ezero$ by 
\begin{equation*}
    \taupq E(V)=\Nat_{0,0}(S\Gmor{p,q}{V}{-}_+, E).
\end{equation*}
\end{definition}

The composition of the sphere bundle inclusion map $S\Gmor{p,q}{V}{W}_+\rightarrow\Gmor{p,q}{V}{W}_+$ with the projection map $\Gmor{p,q}{V}{W}_+\rightarrow \Jzero(V,W)_+ $, defined by $(f,x)\mapsto f$, results in a map $S\Gmor{p,q}{V}{W}_+\rightarrow \Jzero(V,W)_+$. This map induces a natural transformation $ S\Gmor{p,q}{V}{-}_+ \rightarrow \Jzero(V,-)_+$. Applying the contravariant functor $\Nat_{0,0}(-,E)$, for some $E\in\Ezero$, gives a map $$\Nat_{0,0}(\Jzero(V,-)_+,E)\rightarrow \Nat_{0,0}(S\Gmor{p,q}{V}{-}_+,E).$$ Identifying the left side with $E(V)$ using the Yoneda Lemma, and the right side with $\taupq E(V)$ by definition, yields a map $\rho_{p,q}E(V):E(V)\rightarrow \taupq E(V)$. That is, there is a $\C$-equivariant natural transformation $$\rho_{p,q}:\Id\rightarrow\taupq.$$\index{$\rho_{p,q}$} 

There is an alternative description of $\taupq$ as a homotopy limit, that can be derived as a consequence of the following Proposition. This a generalisation of another key result of orthogonal calculus by Weiss \cite[Proposition 4.2]{Wei95}. The proposition states that the sphere bundle of the vector bundle $\Gmor{p,q}{U}{V}$ can be written as a homotopy colimit, and in particular we can do this in a $\C$-equivariant way. 

Recall that $\Jzero$ is the category of finite dimensional subrepresentations of $\oplus_0^\infty \R[\C]$ with inner product and with morphism spaces $\Jzero(U,V)=\LL (U,V)_+$ (Definition \ref{jzero and ezero def}). 
\begin{definition}\label{def: jzerobar}
Let $\Jzerobar$\index{$\Jzerobar$} denote the $\Top_*$-enriched category of finite dimensional subspaces of $\oplus_0^\infty \R[\C]$ with inner product and with morphism spaces 
\begin{equation*}
\Jzerobar(U,V)=\LL (U,V)_+.    
\end{equation*}
\end{definition} 
There is a $\Top_*$-enriched inclusion functor $$\Jzero \hookrightarrow \Jzerobar,$$since $\Jzero$ forms a subcategory of $\Jzerobar$. 

Let $V,W,X\in\Jzero$. Let $\mathcal{C}$ be the category of non-zero subspaces of $X$ ordered by reverse inclusion. Consider the functor $Z$ from $\mathcal{C}$ to $\Top_*$ defined by
    \begin{equation*}
        Z:U\mapsto \Jzerobar (U\oplus V,W).
    \end{equation*}
By \cite[Section 4]{Wei95}, the homotopy colimit of $Z$ is the geometric realisation of the simplicial space
    \begin{equation*}
        [k]\mapsto \coprod\limits_{G:[k]\rightarrow\mathcal{C}} \LL(G(0)\oplus V,W),
    \end{equation*}
where $G$ runs over the order-preserving injections from the poset $[k]=\{0,1,...,k\}$ to $\mathcal{C}$. 

The geometric realisation is, by definition, a quotient of 
    \begin{equation*}
    \coprod\limits_{k\geq 0}\coprod\limits_{G:[k]\rightarrow\mathcal{C}} \LL(G(0)\oplus V,W)\times \triangle^k.
\end{equation*}
Therefore, there is a $\C$-action on the homotopy colimit of $Z$ given by 
\begin{equation*}
    \sigma (G,l,s)= (\sigma G, \sigma l, s),
\end{equation*}
where $G:[k]\rightarrow \mathcal{C}$, $l\in \LL(G(0)\oplus V,W)$ and $s\in\triangle^k$. Details of this action are given in the following proof. Details of this homotopy colimit construction are discussed for the dual homotopy limit in Lemma \ref{e.3}

\begin{proposition}\label{hocolimprop}
For all $V,W,X\in \Jzero$, there is a $\C$-homeomorphism 
\begin{equation*}
    S\Gmor{X}{V}{W}_+\cong \underset{0\neq U \subseteq X}
    {\hocolim} \Jzerobar (U\oplus V,W),
\end{equation*}
where $U$ is a non-zero subspace of $X$. 
\end{proposition}

\begin{remark}
Since $U$ is not necessarily closed under the induced $\C$-action from $X$, the notation $\Jzero(U\oplus V,W)$ would not make sense, however this morphism spaces is exactly the space $\Jzerobar(U\oplus V,W)$ by definition. 
\end{remark}

\begin{proof}
To prove this result, we will construct a homeomorphism 
\begin{equation*}
    \phi:\Jmor{X}{V}{W}\backslash\Jzero(V,W)\rightarrow(0,\infty)\times \underset{U}{\hocolim}\Jzerobar(U\oplus V,W)
\end{equation*}
and then the $\C$-equivariant identification (see \cite[Proposition 4.2]{Wei95})
\begin{align*}
 \Jmor{X}{V}{W}\backslash\Jzero(V,W)&\rightarrow(0,\infty)\times S\Gmor{X}{V}{W}_+   \\
 (f,x)&\mapsto (\| x\|, (f,x/\| x\|))
\end{align*}
yields the desired homeomorphism.

Let $(f,x)\in \Jmor{X}{V}{W}\backslash\Jzero(V,W) $. That is, $f\in \LL(V,W)$ and $x\in X\otimes f(V)^\perp$ such that $x\neq 0$. Since $X$ and $f(V)^\perp$ are finite dimensional inner product spaces, the duality isomorphism tells us that 
\begin{align*}
    X\otimes f(V)^\perp &\cong X^*\otimes f(V)^\perp\\
    &=\text{Hom}(X,\R)\otimes f(V)^\perp\\
    &\cong \text{Hom}(X,f(V)^\perp).
\end{align*}
Hence, we can think of $x$ as a linear map from $X$ to $f(V)^\perp$. In this way, $x$ has an adjoint $x^*:f(V)^\perp\rightarrow X$, and their composition $x^*x:X\rightarrow X$ is self adjoint, since 
\begin{equation*}
    (x^* x)^*=x^*(x^*)^*=x^*x.
\end{equation*}

As a result of \cite[Theorem 6.25]{FIS89}, we can write $X$ as a direct sum of the eigenspaces of $x^*x$. Thus
\begin{equation*}
    X=\text{ker}(x^*x)\oplus E(\lambda_0)\oplus \dots \oplus E(\lambda_k),
\end{equation*}
where $0<\lambda_0<\dots<\lambda_k$ are the non-zero eigenvalues of $x^*x$ and $E(\lambda_i)$ is the eigenspace corresponding to the eigenvalue $\lambda_i$. Note that all of the $\lambda_i$ are real. 

Given the data of $(f,x)\in\Jmor{X}{V}{W}\backslash\Jzero(V,W) $ we can then define the following.

1. A functor $G:[k]\rightarrow \mathcal{C}$ defined by 
\begin{equation*}
    r\mapsto E(\lambda_0)\oplus \dots \oplus E(\lambda_{k-r}),
\end{equation*}
where $\mathcal{C}$ is the category of non-zero subspaces of $X$ with reverse inclusion ordering, and $[k]$ is the category with objects $\{0,1,\dots,k\}$ with the standard ordering. 

2. A linear isometry $l\in\LL(G(0)\oplus V,W)$ defined by  
\begin{equation*}
    l=\begin{cases}
    f \text{ on } V \\
    \lambda_{i}^{-\frac{1}{2}} x \text{ on } E(\lambda_i)
    \end{cases}
\end{equation*}

3. $s\in \triangle^k$ defined in barycentric coordinates by 
\begin{equation*}
    \lambda_k^{-1}(\lambda_0, \lambda_1-\lambda_0, \dots,\lambda_k-\lambda_{k-1}).
\end{equation*}

4. $t=\lambda_k>0$.

Now, by Bousfield and Kan \cite[Section VIII.2.6]{BK72}, we know that the homotopy colimit $\underset{U}{\hocolim}\Jzerobar(U\oplus V,W)$ is a quotient of 
\begin{equation*}
    \coprod\limits_{k\geq 0}\coprod\limits_{G:[k]\rightarrow\mathcal{C}} \LL(G(0)\oplus V,W)\times \triangle^k.
\end{equation*}
We can define the desired homeomorphism as follows.
\begin{align*}
    \phi:\Jmor{X}{V}{W}\backslash\Jzero(V,W)&\rightarrow(0,\infty)\times \underset{U}{\hocolim}\Jzerobar(U\oplus V,W)\\
    (f,x)&\mapsto (t,G,l,s)
\end{align*}

Note that this is a well defined homeomorphism, since it was in the non-equivariant case, as demonstrated by Weiss in \cite[Theorem 4.2]{Wei95}. What remains is to show that this homeomorphism is $\C$-equivariant. 

Let $(f,x)\in \Jmor{X}{V}{W}\backslash\Jzero(V,W) $. We know that $\sigma (f,x)=(\sigma *f, \sigma x)$. Via the duality isomorphism, we see that the vector $\sigma x$ is the map $\sigma\cdot x:=\sigma x \sigma$. In the same way as for the map $x$, $\sigma \cdot x$ has an adjoint, which is given by $(\sigma \cdot x)^*=\sigma x^* \sigma$. Hence, in the same way as for $x^*x$, the composition $(\sigma \cdot x)^*(\sigma \cdot x)$ is self-adjoint, and it is defined by
$\sigma x^*x \sigma$. Now, consider the following calculations.
\begin{align*}
x^*x(v)&=\lambda v\\
\Leftrightarrow \sigma(x^*x(v))&=\sigma(\lambda v)\\
\Leftrightarrow \sigma x^*x\sigma (\sigma v)&=\sigma(\lambda v)\\
\Leftrightarrow (\sigma \cdot x)^*(\sigma \cdot x)(\sigma v)&=\lambda \sigma v\\
& \\
(\sigma \cdot x)^*(\sigma \cdot x)(v)&=\lambda v\\
\Leftrightarrow \sigma x^*x \sigma (v)&=\lambda v\\
\Leftrightarrow \sigma \sigma x^*x\sigma  (v)&=\sigma(\lambda v)\\
\Leftrightarrow x^*x(\sigma v)&=\lambda \sigma v
\end{align*}
These arguments show that the eigenvalues of $x^*x$ and $(\sigma \cdot x)^*(\sigma \cdot x)$ are the same, and the eigenvectors of $(\sigma \cdot x)^*(\sigma \cdot x)$ associated to the eigenvalue $\lambda$ are of the form $\sigma v$, where $v$ is an eigenvector of $x^*x$ corresponding to the eigenvalue $\lambda$. Denote by $\sigma E(\lambda_i)$ the eigenspace of eigenvectors of $(\sigma \cdot x)^*(\sigma \cdot x)$ associated to the eigenvalue $\lambda_i$.

The image of $(\sigma f \sigma, \sigma x)$ under the homeomorphism $\phi$ is $(t,\sigma G, \sigma l, s)$, where 

1. $\sigma G :[k]\rightarrow\mathcal{C}$ is the functor 
\begin{equation*}
    r\mapsto \sigma E(\lambda_0)\oplus \dots \oplus \sigma E(\lambda_{k-r}).
\end{equation*}

2. $\sigma l\in\LL(G(0)\oplus V,W)$ is the linear isometry below. 
\begin{equation*}
    \sigma l=\begin{cases}
    \sigma *f  \text{ on } V \\
    \lambda_{i}^{-\frac{1}{2}} (\sigma \cdot x) \text{ on } \sigma E(\lambda_i)
    \end{cases}
\end{equation*}

3. $s\in \triangle^k$ defined in barycentric coordinates by 
\begin{equation*}
    \lambda_k^{-1}(\lambda_0, \lambda_1-\lambda_0, \dots,\lambda_k-\lambda_{k-1}).
\end{equation*}

4. $t=\lambda_k>0$.

The $\C$-action on $(0,\infty)\times \underset{U}{\hocolim}\Jzerobar(U\oplus V,W)$ is given by
\begin{equation*}
    (t,G,l,s)\mapsto(t,\sigma G, \sigma l, s),
\end{equation*}
hence we conclude that $\sigma \phi (f,x)=\phi (\sigma *f, \sigma x)$. That is, the homeomorphism $\phi$ is $\C$-equivariant. 
\end{proof}

\begin{remark}
Notice that the splitting of $X$ into eigenspaces is not a $C_2$-equivariant splitting. That is, the eigenspaces $E(\lambda_i)$ are not necessarily closed under the $C_2$-action inherited from $X$. This forces the homotopy colimit to be taken over the poset of non-zero subspaces of $\mathbb{R}^{p,q}$, rather than non-zero subrepresentations. Since the proof doesn't rely on any specific properties of $\C$, it should also hold if $\C$ were replaced by an arbitrary finite group $G$ (with suitable replacements for the categories involved). In particular, this implies that for a general $G$-equivariant calculus there should be a natural description of polynomial functors analogous to the $\C$-equivariant polynomial functors discussed in the remainder of this section. 
\end{remark}

Using Proposition \ref{hocolimprop}, we get the following alternative description of $\taupq$. Recall from Definition \ref{def: jzerobar} that $\Jzerobar$ is the $\Top_*$-enriched category whose objects are finite dimensional subspaces of $\oplus_0^\infty \RC$ with inner product. Note that the categories $\Jzero$ and $\CTop_*$ are also $\Top_*$-enriched, and that $\C\Top_*$ is powered over $\Top_*$. We use the notation $\overline{E}$\index{$\overline{E}$} to represent the right Kan extension of an input functor $E\in\Ezero$ along the inclusion $\Jzero\hookrightarrow\Jzerobar$. In particular, $\overline{E}$ is the $\Top_*$-enriched functor $\Jzerobar \rightarrow \C\Top_*$ defined by 
\begin{equation*}
    \overline{E}(X)=\int\limits_{W\in\Jzero} \Top_*\left(\Jzerobar (X,W),E(W)\right).
\end{equation*}
with the $\C$-action induced by the following $\C$-action on $\Top_*\left(\Jzerobar (X,W),E(W)\right)$
\begin{equation*}
    \sigma * f (x) := \sigma (f(x))
\end{equation*}
for all $f\in \Top_*\left(\Jzerobar (X,W),E(W)\right)$ and $x\in \Jzerobar (X,W)$. For details of this construction see \cite[Chapter 4]{Kel05}.

The homotopy limit $\underset{0\neq U \subseteq \R^{p,q}}{\holim} \overline{E}(U\oplus V)$ is the homotopy limit of the functor 
\begin{align*}
    Z:\{0\neq U\subseteq \R^{p,q}\}&\rightarrow \Top_*\\
    U&\mapsto \overline{E}(U\oplus V).
\end{align*}
This homotopy limit has a $\C$-action, since it can be expressed as the totalization of a cosimplicial space, which has a $\C$-action. This is discussed in more detail in Lemma \ref{e.3}.

\begin{proposition}
Let $E\in\Ezero$. There is a $\C$-equivariant homeomorphism $$\taupq E(V) \cong \underset{0\neq U \subseteq \R^{p,q}}{\holim} \overline{E}(U\oplus V).$$
\begin{proof}
Using Proposition \ref{hocolimprop}, we get that
\begin{align*}
\taupq E(V)&=\Nat_{0,0}(S \Gmor{p,q}{V}{-}_+,E) \\
&\cong \Nat_{0,0} \left( \underset{0\neq U \subseteq \R^{p,q} } {\hocolim} \Jzerobar (U\oplus V,-),E\right)\\
&= \int\limits_{W\in\Jzero} \Top_*\left(\underset{0\neq U \subseteq \R^{p,q} } {\hocolim} \Jzerobar (U\oplus V,W),E(W)\right)\\
&\cong \underset{0\neq U \subseteq \R^{p,q}}{\holim} \int\limits_{W\in\Jzero} \Top_*\left(\Jzerobar (U\oplus V,W),E(W)\right)\\
&= \underset{0\neq U \subseteq \R^{p,q}}{\holim} \overline{E}(U\oplus V)
\end{align*}    
The homeomorphism in line 4 above is the comparison, made by Weiss, between the homotopy colimit and homotopy limit constructions (see \cite[Proposition 5.2]{Wei95}). The comparison used in the underlying calculus is naturally $\C$-equivariant.
\end{proof}

\end{proposition}

Now we can define what it means for a functor to be (strongly) polynomial. This definition is a $\C$-equivariant version of \cite[Definition 5.1]{Wei95}. 

\begin{definition}\label{def: polynomial}
A functor $E\in\Ezero$ is called \emph{strongly $(p,q)$-polynomial} if and only if the map 
\begin{equation*}
    \rho_{p,q}E:E\rightarrow\taupq E
\end{equation*}
is an objectwise weak equivalence. 

A functor $E\in\Ezero$ is called \emph{$(p,q)$-polynomial} if and only if $E$ is both strongly $(p+1,q)$-polynomial and strongly $(p,q+1)$-polynomial. 
\end{definition}

\begin{remark}
The term `strongly' is needed in order to keep notation as consistent as possible. In particular, we do this so that $(p,q)$-homogeneous functors are indeed $(p,q)$-polynomial (Section \ref{sec: homog functors}). 
\end{remark}

Functors that are strongly $(p,q)$-polynomial satisfy properties that one might expect based on the properties of polynomial functions. For example, a strongly $(p,q)$-polynomial functor is also strongly $(p+1,q)$ and strongly $(p,q+1)$-polynomial, see Proposition \ref{pq poly implies more poly}. In particular, this means that a strongly $(p,q)$-polynomial functor is $(p,q)$-polynomial. A number of other properties of polynomial functors are discussed in this section. 

The fibre of the map $\rho_{p,q}E$ determines how far a functor $E$ is from being strongly $(p,q)$-polynomial. Lemma \ref{hofiblemma} is a $\C$-equivariant generalisation of \cite[Proposition 5.3]{Wei95}, and it describes how this fibre can be calculated. In particular, this fibre is a derivative of the functor in question. To prove this, we first need the following $\C$-equivariant cofibre sequence. 

\begin{proposition}\label{sphere cofibre seq}
For all $V,W\in \Jzero$ there is a homotopy cofibre sequence in $\C\Top_*$
\begin{equation*}
    S\Gmor{p,q}{V}{W}_+\rightarrow\Jzero(V,W)\rightarrow\Jmor{p,q}{V}{W}.
\end{equation*}
\end{proposition}
\begin{proof}
The proof follows that of the non-equivariant setting, see \cite[Proposition 5.4]{BO13}. In particular, we construct the relevant cofibre using a pushout diagram and define a specific homeomorphism from this pushout to $\Jmor{p,q}{V}{W}$. 

The relevant cofibre is the pushout $P$ of the diagram 
% https://q.uiver.app/#q=WzAsNCxbMCwwLCJTXFxDXFxnYW1tYV97cCxxfShWLFcpXysiXSxbMiwwLCJcXENcXG1hdGhjYWx7Sn1fezAsMH0oVixXKSJdLFsyLDIsIlAiXSxbMCwyLCJTXFxDXFxnYW1tYV97cCxxfShWLFcpXytcXHdlZGdlWzAsXFxpbmZ0eV0iXSxbMCwxXSxbMywyXSxbMCwzXSxbMSwyXV0=
\[\begin{tikzcd}
	{S\C\gamma_{p,q}(V,W)_+} && {\C\mathcal{J}_{0,0}(V,W)} \\
	\\
	{S\C\gamma_{p,q}(V,W)_+\wedge[0,\infty]} && P
	\arrow[from=1-1, to=1-3]
	\arrow[from=3-1, to=3-3]
	\arrow[from=1-1, to=3-1]
	\arrow[from=1-3, to=3-3]
\end{tikzcd}\]
\noindent where $[0,\infty]$ is the one point compactification of $[0,\infty)$ with base point $\infty$, the top horizontal map is the projection $(f,x)\mapsto f$, and the left vertical map is $y\mapsto (y,0)$. 

Elements of the pushout are points $(f,x,t)$, where $(f,x)\in S\Gmor{p,q}{V}{W}$ and $t\in [0,\infty]$, with the following identifications. 
\begin{align*}
    (f,x,\infty)&=(f',x',\infty)\\
    (f,x,0)&= (f',x',0)
\end{align*}

The homeomorphism $\psi$ from $P$ to $\Jpq(V,W)$ is then defined by \begin{align*}
    (f,x,\infty)&\mapsto \text{basepoint}\\
    (f,x,t)&\mapsto (f,xt).
\end{align*}
From the non-equivariant case, it is clear that this map is a well defined homeomorphism (see \cite[Proposition 5.4]{BO13}). It remains to show that it is a $\C$-equivariant map. 
\begin{align*}
    &\psi \circ \sigma (f,x,t)=\psi(\sigma *f , \sigma x, t)=(\sigma *f, (\sigma x) t)\\
    &\sigma \circ \psi (f,x,t)=\sigma (f,xt)=(\sigma *f , (\sigma x)t)\qedhere
\end{align*}
\end{proof}

Using this proposition, we can prove that the functors $S \Gmor{p,q}{V}{-}_+$ and $\Jmor{p,q}{V}{-}$ are cofibrant in the projective model structure. 
\begin{lemma}\label{sphereandmorpharecofibrant}
The functors $S \Gmor{p,q}{V}{-}_+$ and $\Jmor{p,q}{V}{-}$ are cofibrant objects in $\Ezero$. 
\end{lemma}

\begin{proof}
The representable functor $\Jzero (V,-)$ is cofibrant by construction.

The homotopy limit used to construct $\taupq$ (see Lemma \ref{e.3}) preserves objectwise acyclic fibrations, since (indexed) products, totalization and the functor $\Top_*(A,-)$, for a $\C$-CW complex $A$, all preserve acyclic fibrations in $\C\Top_*$. It follows that $S\C\gamma_{p,q}(V,-)_+$ is cofibrant, by applying $\taupq$ to the diagram
% https://q.uiver.app/#q=WzAsNCxbMCwxLCJTXFxDXFxnYW1tYV97cCxxfShWLC0pXysiXSxbMSwxLCJGIl0sWzEsMCwiRSJdLFswLDAsIioiXSxbMywyXSxbMywwXSxbMiwxXSxbMCwxXV0=
\[\begin{tikzcd}
	{*} & E \\
	{S\C\gamma_{p,q}(V,-)_+} & F
	\arrow[from=1-1, to=1-2]
	\arrow[from=1-1, to=2-1]
	\arrow[from=1-2, to=2-2]
	\arrow[from=2-1, to=2-2]
\end{tikzcd}\]
where $E\rightarrow F$ is an objectwise acyclic fibration. 

Since, by Proposition \ref{sphere cofibre seq}, $\Jmor{p,q}{V}{-}$ is the cofibre of a map of cofibrant objects, it is also cofibrant. 
\end{proof}

The following result describes the relation between derivatives and polynomial functors. From this fibration sequence, one can see that if a functor $E$ is strongly $(p,q)$-polynomial, then $\ind_{0,0}^{p,q}E(V)$ is contractible. This is analogous to an $n$-polynomial function having zero $(n+1)^{\text{st}}$ derivative. 

\begin{lemma}\label{hofiblemma}
For all $E\in \Ezero$, and for all $V\in \Jzero$ there is a homotopy fibre sequence in $\C\Top_*$
\begin{equation*}
    \ind_{0,0}^{p,q} E(V)\rightarrow E(V)\rightarrow\taupq E(V).
\end{equation*}
\end{lemma}
\begin{proof}
By the previous proposition, there is a $\C$-equivariant homotopy cofibre sequence in $\Ezero$
\begin{equation*}
     S\Gmor{p,q}{V}{-}_+\rightarrow\Jzero(V,-)\rightarrow\Jmor{p,q}{V}{-}.
\end{equation*}

Pick $E\in\Ezero$ and apply the contravariant functor $\Nat_{p,q}(-,E)$ to the cofibre sequence above. This yields a homotopy fibre sequence (see Remark \ref{remNat})
\begin{equation*}
\Nat_{0,0}\left(S\Gmor{p,q}{V}{-}_+,E\right)\leftarrow \Nat_{0,0}\left(\Jzero(V,-),E\right)\leftarrow \Nat_{0,0}\left(\Jmor{p,q}{V}{-},E\right).
\end{equation*}
Application of the Yoneda Lemma and the definitions of  $\ind_{0,0}^{p,q}$ and $\taupq E$ gives the desired fibre sequence.
\end{proof}

\begin{corollary}\label{poly implies ind contractible}
If $E\in\Ezero$ is strongly $(p,q)$-polynomial, then $\ind_{0,0}^{p,q}E$ and $\ind_{0,0}^{p,q}\varepsilon^*E$ are objectwise contractible.  
\end{corollary}

\section{Polynomial approximation}\label{sec: poly approx}

The partial sums of the Taylor series for a real function are known as the Taylor polynomials. These polynomial functions approximate the given function, and in general become better approximations as the degree of polynomial increases. In orthogonal calculus, Weiss defines a polynomial approximation functor $T_n$, see \cite[Theorem 6.3]{Wei95}. For an input functor $E\in\Ezero$, each $T_nE$ is indeed an $n$-polynomial functor. In the $\C$-equivariant setting, we define an analogous functor $T_{p,q}$, and the $(p,q)$-polynomial approximation functor is given by the composition $T_{p+1,q}T_{p,q+1}$.

\begin{definition}\index{$\Tpq$}
Let $E\in \Ezero$, define the functor $\Tpq$ by
\begin{equation*}
    \Tpq E=\hocolim\left(E\rightarrow\taupq E\rightarrow\taupq^2 E\rightarrow\dots\right),
\end{equation*}
where the homotopy colimit is taken over the maps $$\rho_{p,q}(\taupq^kE):\taupq^k E\rightarrow \taupq^{k+1} E.$$
\end{definition}
There is a natural transformation $\eta :E\rightarrow \Tpq E$, which is inclusion as the first term in the colimit. Alternatively, this map can be thought of as a map of homotopy colimits 
% https://q.uiver.app/#q=WzAsMixbMCwwLCJcXHRleHR7aG9jb2xpbX0oRVxccmlnaHRhcnJvdyBFXFxyaWdodGFycm93IEVcXHJpZ2h0YXJyb3cgXFxkb3RzKSJdLFswLDEsIlxcdGV4dHtob2NvbGltfVxcbGVmdChFXFxyaWdodGFycm93XFx0YXVfe3AscX0gRVxccmlnaHRhcnJvd1xcdGF1X3twLHF9XjIgRVxccmlnaHRhcnJvd1xcZG90c1xccmlnaHQpIl0sWzAsMSwiXFxldGEiXV0=
\[\begin{tikzcd}
	{\hocolim(E\rightarrow E\rightarrow E\rightarrow \dots)} \\
	{\hocolim\left(E\rightarrow\tau_{p,q} E\rightarrow\tau_{p,q}^2 E\rightarrow\dots\right)}
	\arrow["\eta", from=1-1, to=2-1]
\end{tikzcd}\]
where each $E\rightarrow \taupq^k E$ is made from the maps $\rho_{p,q}(\taupq^kE)$.

\begin{example}\label{exampleT10T01}
Let $E\in\Ezero$. Then $T_{1,0}E$ is $E(-\oplus \mathbb{R}^\infty)$ and $T_{0,1}E$ is $E(-\oplus \mathbb{R}^{\infty\delta})$. This can be seen by the following calculations.
\begin{equation*}
    T_{1,0}E(V)=\hocolim_j \tau_{1,0}^j E(V)=\hocolim_j E(V\oplus \mathbb{R}^j):= E(V\oplus \mathbb{R}^\infty)
\end{equation*}
\begin{equation*}
    T_{0,1}E(V)=\hocolim_j \tau_{0,1}^j E(V)=\hocolim_j E(V\oplus \mathbb{R}^{j\delta}):= E(V\oplus \mathbb{R}^{\infty\delta}),
\end{equation*}
where $\mathbb{R}^\infty:=\hocolim_j \mathbb{R}^j$ and $\mathbb{R}^{\infty\delta}:=\hocolim_j \mathbb{R}^{j\delta}$. Recall that, as a consequence of Proposition \ref{hocolimprop}, $\taupq E(V)\cong \underset{0\neq U \subseteq \R^{p,q}}{\holim}\overline{E}(U\oplus V)$, where $\overline{E}$ is the right Kan extension of $E$ along $\Jzero\hookrightarrow \Jzerobar$. However, since $V\oplus \R^j$ and $V\oplus \R^{j\delta}$ are elements of $\Jzero$, the right Kan extension $\overline{E}$ is exactly the functor $E$ on these values.

The strongly $(0,0)$-polynomial approximation is the constant functor $T_{0,0}E(V)=*$, since $\tau_{0,0}E$ is the homotopy limit over the empty set. 
\end{example}

\begin{remark}
In orthogonal calculus, the $0$-polynomial approximation $T_0 F$ of an input functor $F\in\mathcal{E}_0$ is the constant functor taking value $F(\mathbb{R}^\infty)$. In particular, $T_0 F$ is a constant functor, but $T_{1,0} E$ and $T_{0,1}E$ are not in general. See Example \ref{ex: 00-poly approx} for the correct $\C$-analogue of $T_0 F$. 
\end{remark}

The functor $\Tpq E$ is strongly $(p,q)$-polynomial for all $E\in \Ezero$. To prove this, we need to generalise the erratum to orthogonal calculus \cite{Wei98} to the $\C$-equivariant setting. This is done over the following collection of lemmas. A formula similar to that used for the connectivity of $(\taupq s(W))^{\C}$ in Part 2 of the following lemma is used by Dotto in \cite[Corollary A.2]{Dot16b}. 
Recall that for a functor $G\in\Ezero$, the functor $\taupq G\in\Ezero$ is given by 
$$\taupq G(W)\cong \underset{0\neq U \subseteq \R^{p,q}}{\holim} \overline{G}(U\oplus W).$$ By abuse of notation, for a functor $G\in \text{Fun}_{\Top_*}(\Jzerobar,\C\Top_*)$, we define the functor $\taupq G\in\Ezero$ by 
$$\taupq G(W):= \underset{0\neq U \subseteq \R^{p,q}}{\holim} G(U\oplus W).$$

\begin{lemma}\label{e.3}
Let $s:G\rightarrow F$ be a morphism in $\Fun_{\Top_*}(\Jzerobar,\C\Top_*)$ and $p,q\geq 0$. If there exists integers $b,c$ such that $s(W)$ is $v$-connected for all $W\in \Jzero$, where 
\begin{align*}
    v(e)&=(p+q)\Dim W -b\\
    v(\C)&=\Min\{(p+q)\Dim W -b, p\Dim W^{\C} + q\Dim (W^{\C})^\perp -c\},
\end{align*}
then $\taupq s(W)$ is $(v+1)$-connected.

\end{lemma}
\begin{proof}
Let $\mathcal{D}$ be the topological poset of non-zero linear subspaces of $\R^{p+q\delta}$. Similar to \cite[Lemma e.3]{Wei98}, the homotopy limit in $\taupq s(W)$ is the totalization of a cosimplicial object as follows. For $Z:\mathcal{D}\rightarrow \Top_*$, the homotopy limit of $Z$ is the totalization of the cosimplicial space $$ [k]\mapsto \prod\limits_{L:[k]\rightarrow \mathcal{D}} Z(L(K))$$ taken over all monotone injections $[k]\rightarrow \mathcal{D}$. In particular, we are interested in the cases $Z(U):= G(U\oplus W)$ and $Z(U):=F(U\oplus W)$

Note that $\mathcal{D}$ is a category internal to $\C$-spaces. The space of objects is given by a disjoint union of $\C$-Grassmann manifolds 
\begin{equation*}
    \coprod\limits_{0\leq i\leq p+q} \mathcal{L}(\mathbb{R}^i,\mathbb{R}^{p+q\delta}) / O(i)
\end{equation*}
and the space of morphisms is the space of flags of subspaces of $\mathbb{R}^{p+q\delta}$ of length two
\begin{equation*}
    \coprod\limits_{0\leq i\leq j\leq p+q} \mathcal{L}(\mathbb{R}^j,\mathbb{R}^{p+q\delta}) / O(j-i)\times O(i),
\end{equation*}
where $\C$ acts by conjugation. 

To capture this structure, we can replace the cosimplicial space above by another. Let $\xi$ be the fibre bundle over $\mathcal{D}$, 
\begin{equation*}
    \coprod\limits_{0\leq i\leq p+q} Z(\R^i)\times_{O(i)}\LL(\R^i,\R^{p,q})\overset{\text{proj}}{\rightarrow} \coprod\limits_{0\leq i\leq p+q} \mathcal{L}(\mathbb{R}^i,\mathbb{R}^{p+q\delta}) / O(i)=\text{ob}(D),
\end{equation*}
such that the fibre of $U\in\mathcal{D}$ is $Z(U)$. Let $e_k: \mathcal{C}\rightarrow \mathcal{D}$ be defined by $L\mapsto L(k)$, where $\mathcal{C}$ is the $\C$-space of monotone injections $[k]\rightarrow \mathcal{D}$. Then we can replace the previous cosimplicial space by $$[k]\mapsto \Gamma (e_k^*\xi),$$ where $e_k^*\xi$ is the pullback bundle over $\mathcal{C}$ and $\Gamma$ denotes taking the section space. 

The space $\mathcal{C}$ is a disjoint union of $\C$-manifolds $C(\lambda)$, taken over monotone injections $\lambda:[k]\rightarrow [p+q]$ that avoid $0\in [p+q]$. These manifolds are defined by $$C(\lambda)=\{L:[k]\rightarrow\mathcal{D} : \Dim(L(i))=\lambda(i),\space \forall i\}.$$
That is, $C (\lambda)$ is the space of all flags of length $k$ and weight $\lambda$. Writing this as a quotient of orthogonal groups (where $\lambda_i=\lambda(i)$)
\begin{equation*}
    C(\lambda)= \coprod\limits_{0\leq \lambda_0\leq ...\leq \lambda_k\leq p+q} \mathcal{L}(\mathbb{R}^{\lambda_k},\mathbb{R}^{p+q\delta}) / O(\lambda_k-\lambda_{k-1})\times O(\lambda_{k-1}-\lambda_{k-2})\times ...\times O(\lambda_{0}),
\end{equation*}
one can calculate the dimension of $C(\lambda)$. 

\begin{align*}
    \Dim(C(\lambda))&=((p+q)-\lambda(k))\lambda(k) +\sum\limits_{i=0}^{k-1} (\lambda(i+1)-\lambda(i))\lambda(i)\\
    &= (p+q)\lambda(k) +\sum\limits_{i=0}^{k-1}\lambda(i)\lambda(i+1) -\sum\limits_{i=0}^{k}\lambda(i)^2\\
    &< (p+q)\lambda(k) -k
\end{align*}
The final inequality results from expanding the two sums and noting that $-\lambda(0)^2\leq -1$ and $\lambda(i-1)\lambda(i)-\lambda(i)^2\leq -1$. 

Using the discussion of the homotopy limit as the totalization above, we see that the connectivity of $(\taupq s(W))^e$ is greater than or equal to the minimum of $$\conn(s(L(k)\oplus W))-\Dim((C(\lambda)) -k$$ taken over triples $(L,\lambda,k)$ with $L\in C(\lambda)$ and $\lambda:[k]\rightarrow [p+q]$. That is, the connectivity of the map $s$ at the level $k$ minus the dimension of the space we are mapping from when considering the totalization as an enriched end. Substituting in the hypothesis on the connectivity of $s(L(k)\oplus W)$ and the bound on the dimension of $C(\lambda)$ yields that the connectivity of $(\taupq s(W))^e$ is at least $v(e)+1$. 

Note that for $G$-spaces $A,B$   
$$\conn (\Top_*(A,B)^G)>\underset{H\leq G}{\text{min}}\{\conn B^H-\Dim A^H\}$$
taken over closed subgroups $H$ of $G$. Using this, along with the fact that fixed points commute with totalization, we see that the connectivity of $(\taupq s(W))^{\C}$ is greater than or equal to the minimum of 
\begin{equation*}
    \underset{H\leq \C}{\text{min}}\{\conn(s(L(k)\oplus W)^{H})-\Dim((C(\lambda)^{H}) -k\}
\end{equation*}
taken over triples $(L,\lambda,k)$ with $L\in C(\lambda)$ and $\lambda:[k]\rightarrow [p+q]$. A similar formula is used by Dotto in \cite[Corollary A.2]{Dot16b}.

One can determine $C(\lambda)^{\C}$ by applying the splitting theorem, Theorem \ref{splittingtheorems}, to the definition of $C(\lambda)$. Then a calculation similar to that above for $\Dim(C(\lambda))$ shows that $$\Dim(C(\lambda)^{\C}) < p\lambda(k)^{\C} +q(\lambda(k)^{\C})^\perp -k,$$ where $\lambda(k)^{\C}= \Dim(L(k)^{\C})$ and $(\lambda(k)^{\C})^{\perp}= \Dim((L(k)^{\C})^{\perp})$. The result then follows by substituting in the hypothesis on the connectivity of $s(L(k)\oplus W)^{\C}$ as we did for the first result above. 
\end{proof}

The following corollary can be proved using the same method as Lemma \ref{e.3}. 

\begin{corollary}\label{erratum corollary}Let $s:G\rightarrow F$ be a morphism in $\Fun_{\Top_*}(\Jzerobar,\C\Top_*)$ and $p,q\geq 1$.
\begin{enumerate}
    \item If there exists integers $b,c$ such that $s(W)$ is $v$-connected for all $W\in \Jzero$, where 
        \begin{align*}
            v(e)&=2(p+q)\Dim W -b\\
            v(\C)&=\Min\{2(p+q)\Dim W -b, 2p\Dim W^{\C} + 2q\Dim (W^{\C})^\perp -c\},
        \end{align*}
        then $\tau_{p+1,q}\tau_{p,q+1} s(W)$ is at least $(v+1)$-connected.
    \item If there exists integers $b,c$ such that $s(W)$ is $v$-connected for all $W\in \Jzero$, where 
        \begin{align*}
            v(e)&=2(p+q)\Dim W -b\\
            v(\C)&=\Min\{2(p+q)\Dim W -b, (p+q)\Dim W -c\},
        \end{align*}
        then $\tau_{p+1,q}\tau_{p,q+1} s(W)$ is at least $(v+1)$-connected.
\end{enumerate}
\end{corollary}

\begin{remark}
    Let $F,G\in\Ezero$. Recall that $\overline{G},\overline{F}$ are the right Kan extensions of $F,G$ respectively along the inclusion $\Jzero\rightarrow\Jzerobar$. If the map $s(W):F(W)\rightarrow G(W)$ is $v$-connected for all $W\in\Jzero$, then the map $\overline{s}(V):\overline{F}(V)\rightarrow \overline{G}(V)$ is $v$-connected for all $V\in\Jzerobar$. In particular, this means that given the connectivity of the map $s(W)$, Lemma \ref{e.3} can be applied to the map $\overline{s}(V)$ to make conclusions about the connectivity of the map $\taupq s(W)=\taupq \overline{s}(W)$.
\end{remark}

We aim to show that $\Tpq E$ is strongly $(p,q)$-polynomial for any $E\in \Ezero$. That is, we wish to show that $\rho: \Tpq E \rightarrow \taupq \Tpq E$ is an objectwise weak equivalence. To do this we will need the following two Lemmas, the first of which is a $\C$-version of \cite[Lemma e.7]{Wei98}. 

\begin{lemma}\label{e.7}
Let $G:= S\Gmor{p,q}{V}{-}$, $F:=\Jzero(V,-)$, and let $s:G\rightarrow F$ be the projection sphere bundle map. $\Tpq s$ is an objectwise weak equivalence. 
\end{lemma}

\begin{proof}
We know from the underlying calculus that $s(W)^e$ is $[(p+q)(\Dim W -\Dim V) -1]$-connected.
Using the splitting theorem (see Theorem \ref{splittingtheorems}), we see that 
\begin{align*}
    \Jpq(V,W)^{\C} &= [T(\C\gamma_{p,q}(V,W))]^{\C}\\
    &\cong T[\C\gamma_{p,q}(V,W)^{\C}]\\
    &\cong T[\C\gamma_{p,0}(V^{\C},W^{\C})\times \C\gamma_{0,q}((V^{\C})^\perp, (W^{\C})^\perp)]\\
    &=\C\mathcal{J}_{p,0}(V^{\C},W^{\C}) \wedge \C\mathcal{J}_{0,q}((V^{\C})^\perp, (W^{\C})^\perp)
\end{align*}
which is $[p(\dim W^{\C}-\dim V^{\C} ) + q(\dim (W^{\C})^\perp -\dim (V^{\C})^\perp ) -1]$-connected. Therefore, so is the map $s(W)^{\C}$, since $\Jpq(V,W)^{\C}$ is the homotopy cofibre of $s(W)^{\C}$. 

Thus, $s(W)$ satisfies the hypothesis of Lemma \ref{e.3} with 
\begin{align*}
    b&=(p+q)\Dim V+1\\
    c&=p\Dim V^{\C}+q\Dim (V^{\C})^{\perp}+1.
\end{align*}
Repeated application of Lemma $\ref{e.3}$ shows that the connectivity of both $(\taupq ^l s(W))^e$ and $(\taupq ^l s(W))^{\C}$ tend to infinity as $l$ tends to infinity. Thus, $(\Tpq s(W))^e$ and $(\Tpq s(W))^{\C}$ are weak homotopy equivalences, which is exactly that $\Tpq s$ is an objectwise weak equivalence. 
\end{proof}

The following Lemma is a $\C$-equivariant version of the discussion above \cite[Theorem 6.3.1]{Wei98}. Note that when $H=e$ these diagrams are similar to \cite[e.8 and e.9]{Wei98}.

\begin{lemma}\label{lem: enlarging diagram}
Let $p,q\geq 0$ and $E\in \Ezero$. The commutative diagram 
% https://q.uiver.app/#q=WzAsNCxbMCwwLCJFKFYpXkgiXSxbMiwwLCJUX3twLHF9RShWKV5IIl0sWzIsMiwiXFx0YXVwcSBUX3twLHF9RShWKV5IIl0sWzAsMiwiXFx0YXVwcSBFKFYpXkgiXSxbMCwxLCJcXHN1YnNldGVxIl0sWzEsMiwiXFxyaG9eSCJdLFszLDIsIlxcc3Vic2V0ZXEiLDJdLFswLDMsIlxccmhvXkgiLDJdXQ==
\[\begin{tikzcd}
	{E(V)^H} && {(T_{p,q}E(V))^H} \\
	\\
	{(\taupq E(V))^H} && {(\taupq T_{p,q}E(V))^H}
	\arrow["\eta^H", from=1-1, to=1-3]
	\arrow["{\rho^H}", from=1-3, to=3-3]
	\arrow["\eta^H"', from=3-1, to=3-3]
	\arrow["{\rho^H}"', from=1-1, to=3-1]
\end{tikzcd}\]
can be enlarged to a commutative diagram 
% https://q.uiver.app/#q=WzAsNixbMCwwLCJFKFYpXkgiXSxbNCwwLCJUX3twLHF9RShWKV5IIl0sWzQsMiwiXFx0YXVwcSBUX3twLHF9RShWKV5IIl0sWzAsMiwiXFx0YXVwcSBFKFYpXkgiXSxbMiwwLCJYIl0sWzIsMiwiWSJdLFsxLDIsIlxccmhvXkgiXSxbMCwzLCJcXHJob15IIiwyXSxbNCw1LCJnIiwyXSxbMCw0XSxbNCwxXSxbMyw1XSxbNSwyXV0=
\[\begin{tikzcd}
	{E(V)^H} && X && {(T_{p,q}E(V))^H} \\
	\\
	{(\taupq E(V))^H} && Y && {(\taupq T_{p,q}E(V))^H}
	\arrow["{\rho^H}", from=1-5, to=3-5]
	\arrow["{\rho^H}"', from=1-1, to=3-1]
	\arrow["g"', from=1-3, to=3-3]
	\arrow[from=1-1, to=1-3]
	\arrow[from=1-3, to=1-5]
	\arrow[from=3-1, to=3-3]
	\arrow[from=3-3, to=3-5]
\end{tikzcd}\]
where $g$ is a weak homotopy equivalence. 
\end{lemma}

\begin{proof}

By the Yoneda Lemma and definition of $\taupq$, we have a commutative diagram similar to \cite[e.6]{Wei98}, where $G:= S\Gmor{p,q}{V}{-}$, $F:=\Jzero(V,-)$ and $s:G\rightarrow F$ is the projection sphere bundle map.. 
% https://q.uiver.app/#q=WzAsNCxbMCwwLCJFKFYpXkgiXSxbMiwwLCJcXHRhdXBxIEUoVileSCJdLFsyLDIsIlxcTmF0X3swLDB9KEcsRSgtKV5IKSJdLFswLDIsIlxcTmF0X3swLDB9KEYsRSgtKV5IKSJdLFswLDEsIlxccmhvXkgiXSxbMSwyLCJcXGNvbmciXSxbMywyLCJzXioiLDJdLFswLDMsIlxcY29uZyIsMl1d
\[\begin{tikzcd}
	{E(V)^H} && {(\taupq E(V))^H} \\
	\\
	{\Nat_{0,0}(F,E(-))^H} && {\Nat_{0,0}(G,E(-))^H}
	\arrow["{\rho^H}", from=1-1, to=1-3]
	\arrow["=", from=1-3, to=3-3]
	\arrow["{s^*}"', from=3-1, to=3-3]
	\arrow["\cong"', from=1-1, to=3-1]
\end{tikzcd}\]
Using this, the first diagram above can be written as
% https://q.uiver.app/#q=WzAsNCxbMCwwLCJcXE5hdF97MCwwfShGLEUoLSleSCkiXSxbMiwwLCJcXE5hdF97MCwwfShGLFRfe3AscX1FKC0pXkgpIl0sWzIsMiwiXFxOYXRfezAsMH0oRyxUX3twLHF9RSgtKV5IKSJdLFswLDIsIlxcTmF0X3swLDB9KEcsRSgtKV5IKSJdLFswLDFdLFszLDJdLFsxLDIsInNeKiJdLFswLDMsInNeKiIsMl1d
\[\begin{tikzcd}
	{\Nat_{0,0}(F,E(-))^H} && {\Nat_{0,0}(F,T_{p,q}E(-))^H} \\
	\\
	{\Nat_{0,0}(G,E(-))^H} && {\Nat_{0,0}(G,T_{p,q}E(-))^H}
	\arrow[from=1-1, to=1-3]
	\arrow[from=3-1, to=3-3]
	\arrow["{s^*}", from=1-3, to=3-3]
	\arrow["{s^*}"', from=1-1, to=3-1]
\end{tikzcd}\]
which in the same way as \cite[e.10]{Wei98} can be enlarged as below.
% https://q.uiver.app/#q=WzAsNixbMCwwLCJcXE5hdF97MCwwfShGLEUoLSleSCkiXSxbMywwLCJcXE5hdF97MCwwfShGLFRfe3AscX1FKC0pXkgpIl0sWzMsMiwiXFxOYXRfezAsMH0oRyxUX3twLHF9RSgtKV5IKSJdLFswLDIsIlxcTmF0X3swLDB9KEcsRSgtKV5IKSJdLFsxLDAsIlxcTmF0X3swLDB9KFRfe3AscX1GLFRfe3AscX1FKC0pXkgpIl0sWzEsMiwiXFxOYXRfezAsMH0oVF97cCxxfUcsVF97cCxxfUUoLSleSCkiXSxbMSwyLCJzXioiXSxbMCwzLCJzXioiLDJdLFswLDQsIlRfe3AscX0iXSxbMyw1LCJUX3twLHF9IiwyXSxbNCwxLCJcXHJlcyJdLFs1LDIsIlxccmVzIiwyXSxbNCw1LCIoVF97cCxxfXMpXioiXV0=
\[\begin{tikzcd}
	{\Nat_{0,0}(F,E(-))^H} & {\Nat_{0,0}(T_{p,q}F,T_{p,q}E(-))^H} && {\Nat_{0,0}(F,T_{p,q}E(-))^H} \\
	\\
	{\Nat_{0,0}(G,E(-))^H} & {\Nat_{0,0}(T_{p,q}G,T_{p,q}E(-))^H} && {\Nat_{0,0}(G,T_{p,q}E(-))^H}
	\arrow["{s^*}", from=1-4, to=3-4]
	\arrow["{s^*}"', from=1-1, to=3-1]
	\arrow["{T_{p,q}}", from=1-1, to=1-2]
	\arrow["{T_{p,q}}"', from=3-1, to=3-2]
	\arrow["\res", from=1-2, to=1-4]
	\arrow["\res"', from=3-2, to=3-4]
	\arrow["{(T_{p,q}s)^*}", from=1-2, to=3-2]
\end{tikzcd}\]
In the projective model structure defined in Proposition \ref{proj model structure}, we can factorize the maps $\eta_F: F \rightarrow \Tpq F$ and $\eta_G: G \rightarrow \Tpq G$ as a cofibration followed by an acyclic fibration. This looks as follows. 
\begin{equation*}
    F\hookrightarrow M \tilde{\twoheadrightarrow} \Tpq F
\end{equation*}
\begin{equation*}
   G\hookrightarrow N \tilde{\twoheadrightarrow} \Tpq G
\end{equation*}
We know that $G$ and $F$ are cofibrant in the projective model structure, by Lemma \ref{sphereandmorpharecofibrant}, therefore $M$ and $N$ are also cofibrant. Hence, they act as cofibrant replacements for $\Tpq F$ and $\Tpq G$ respectively. 

Applying these cofibrant replacements in the above diagram we get another commutative diagram 
% https://q.uiver.app/#q=WzAsNixbMCwwLCJcXE5hdF97MCwwfShGLEUoLSleSCkiXSxbMywwLCJcXE5hdF97MCwwfShGLFRfe3AscX1FKC0pXkgpIl0sWzMsMiwiXFxOYXRfezAsMH0oRyxUX3twLHF9RSgtKV5IKSJdLFswLDIsIlxcTmF0X3swLDB9KEcsRSgtKV5IKSJdLFsxLDAsIlxcTmF0X3swLDB9KE0sVF97cCxxfUUoLSleSCkiXSxbMSwyLCJcXE5hdF97MCwwfShOLFRfe3AscX1FKC0pXkgpIl0sWzEsMiwic14qIl0sWzAsMywic14qIiwyXSxbMCw0XSxbMyw1XSxbNCwxLCJcXHJlcyJdLFs1LDIsIlxccmVzIiwyXSxbNCw1LCJqIl1d
\[\begin{tikzcd}
	{\Nat_{0,0}(F,E(-))^H} & {\Nat_{0,0}(M,T_{p,q}E(-))^H} && {\Nat_{0,0}(F,T_{p,q}E(-))^H} \\
	\\
	{\Nat_{0,0}(G,E(-))^H} & {\Nat_{0,0}(N,T_{p,q}E(-))^H} && {\Nat_{0,0}(G,T_{p,q}E(-))^H}
	\arrow["{s^*}", from=1-4, to=3-4]
	\arrow["{s^*}"', from=1-1, to=3-1]
	\arrow[from=1-1, to=1-2]
	\arrow[from=3-1, to=3-2]
	\arrow["\res", from=1-2, to=1-4]
	\arrow["\res"', from=3-2, to=3-4]
	\arrow["j", from=1-2, to=3-2]
\end{tikzcd}\]
Since $\Tpq s$ is an objectwise weak equivalence, by Lemma \ref{e.7}, we get that $M$ and $N$ are objectwise weakly equivalent. Finally, since the projective model structure is $\C$-topological, $\Nat_{0,0}(-,A)$ preserves weak equivalences of cofibrant objects, and we can conclude that $j$ is a weak homotopy equivalence as required. 
    
\end{proof}

\begin{theorem}\label{weiss6.3.1}
$\Tpq E$ is strongly $(p,q)$-polynomial for all $E\in\Ezero$ and all $p,q\geq 0$.
\end{theorem}

\begin{proof}
We must show that the map $\Tpq E \rightarrow \taupq \Tpq E$ is an objectwise weak equivalence. The proof follows in the same way as \cite[Theorem 6.3.1]{Wei98}. 
% https://q.uiver.app/#q=WzAsOCxbMCwwLCJFKFYpXkgiXSxbMiwwLCJcXHRhdXBxIEUoVileSCJdLFs0LDAsIlxcdGF1cHFeMiBFKFYpXkgiXSxbMCwyLCJcXHRhdXBxIEUoVileSCJdLFsyLDIsIlxcdGF1cHFeMiBFKFYpXkgiXSxbNCwyLCJcXHRhdXBxXjMgRShWKV5IIl0sWzYsMiwiXFxkb3RzIl0sWzYsMCwiXFxkb3RzIl0sWzAsMywiXFxyaG9eSCJdLFswLDEsIlxccmhvXkgiXSxbMSwyLCJcXHJob15IIl0sWzIsNywiXFxyaG9eSCJdLFsxLDQsIlxccmhvXkgiXSxbMiw1LCJcXHJob15IIl0sWzMsNCwiXFx0YXVwcVxccmhvXkgiLDJdLFs0LDUsIlxcdGF1cHFcXHJob15IIiwyXSxbNSw2LCJcXHRhdXBxXFxyaG9eSCIsMl1d
\[\begin{tikzcd}
	{E(V)^H} && {(\taupq E(V))^H} && {(\taupq^2 E(V))^H} && \dots \\
	\\
	{(\taupq E(V))^H} && {(\taupq^2 E(V))^H} && {(\taupq^3 E(V))^H} && \dots
	\arrow["{\rho^H}", from=1-1, to=3-1]
	\arrow["{\rho^H}", from=1-1, to=1-3]
	\arrow["{\rho^H}", from=1-3, to=1-5]
	\arrow["{\rho^H}", from=1-5, to=1-7]
	\arrow["{\rho^H}", from=1-3, to=3-3]
	\arrow["{\rho^H}", from=1-5, to=3-5]
	\arrow["{\taupq\rho^H}"', from=3-1, to=3-3]
	\arrow["{\taupq\rho^H}"', from=3-3, to=3-5]
	\arrow["{\taupq\rho^H}"', from=3-5, to=3-7]
\end{tikzcd}\]
It suffices to show that the the vertical maps in the diagram above induce a weak homotopy equivalence, $r:(\Tpq E(V))^H \rightarrow (\taupq \Tpq E(V))^H$ for all closed subgroups $H\leq \C$,  between the homotopy colimits of the rows. By Lemma \ref{lem: enlarging diagram}, each diagram 
% https://q.uiver.app/#q=WzAsNCxbMCwwLCJcXHRhdV97cCxxfV5rRShWKV5IIl0sWzAsMSwiXFx0YXVfe3AscX1ee2srMX1FKFYpXkgiXSxbMSwwLCJUX3twLHF9RShWKV5IIl0sWzEsMSwiXFx0YXVfe3AscX1UX3twLHF9RShWKV5IIl0sWzAsMiwiXFxzdWJzZXRlcSJdLFsxLDMsIlxcc3Vic2V0ZXEiLDJdLFswLDEsIlxccmhvXkgiLDJdLFsyLDMsInIiXV0=
\[\begin{tikzcd}
	{(\tau_{p,q}^kE(V))^H} & {(T_{p,q}E(V))^H} \\
	{(\tau_{p,q}^{k+1}E(V))^H} & {(\tau_{p,q}T_{p,q}E(V))^H}
	\arrow["\subseteq", from=1-1, to=1-2]
	\arrow["\subseteq"', from=2-1, to=2-2]
	\arrow["{\rho^H}"', from=1-1, to=2-1]
	\arrow["r", from=1-2, to=2-2]
\end{tikzcd}\]
can be enlarged to a commutative diagram 
% https://q.uiver.app/#q=WzAsNixbMCwwLCJcXHRhdV97cCxxfV5rRShWKV5IIl0sWzAsMSwiXFx0YXVfe3AscX1ee2srMX1FKFYpXkgiXSxbMiwwLCJUX3twLHF9RShWKV5IIl0sWzIsMSwiXFx0YXVfe3AscX1UX3twLHF9RShWKV5IIl0sWzEsMCwiWCJdLFsxLDEsIlkiXSxbMCwxLCJcXHJob15IIiwyXSxbMCw0XSxbNCwyXSxbMSw1XSxbNSwzXSxbMiwzLCJyIl0sWzQsNSwiZyJdXQ==
\[\begin{tikzcd}
	{(\tau_{p,q}^kE(V))^H} & X & {(T_{p,q}E(V))^H} \\
	{(\tau_{p,q}^{k+1}E(V))^H} & Y & {(\tau_{p,q}T_{p,q}E(V))^H}
	\arrow["{\rho^H}"', from=1-1, to=2-1]
	\arrow[from=1-1, to=1-2]
	\arrow[from=1-2, to=1-3]
	\arrow[from=2-1, to=2-2]
	\arrow[from=2-2, to=2-3]
	\arrow["r", from=1-3, to=2-3]
	\arrow["g", from=1-2, to=2-2]
\end{tikzcd}\]
where $g$ is a weak homotopy equivalence. 

Any element in the homotopy group $\pi_*(T_{p,q}E(V))^H$ may be realised as an element of the corresponding homotopy group $\pi_*(\tau_{p,q}^kE(V))^H$, for some $k$. Proving injectivity and surjectivity of $\pi_*r$ follows from the existence of $g$. Then $\Tpq E$ is strongly $(p,q)$-polynomial. 
\end{proof}

In particular, this allows us to define the $(p,q)$-polynomial approximation functor. 

\begin{definition}
Define the \emph{$(p,q)$-polynomial approximation} to $E\in \Ezero$ to be the functor $T_{p+1,q}T_{p,q+1}E$. By the above lemma, this is indeed a $(p,q)$-polynomial functor. 
\end{definition}

\begin{example}\label{ex: 00-poly approx}
The $(0,0)$-polynomial approximation of a functor $E$ is the constant functor 
\begin{equation*}
    T_{1,0}T_{0,1}E(V)\cong \underset{k}{\hocolim}E(\R^{k,k})\cong \underset{a,b}{\hocolim}E(\R^{a,b})=: E(\R^{\infty,\infty}).
\end{equation*}
This is analogous to the $0$-polynomial approximation of a functor $E$ being the constant functor $T_0E(V)=E(\R^\infty)$ in the underlying calculus. 
\end{example}

The following is the $\C$-equivariant generalisation of \cite[Theorem 6.3.2]{Wei95}. The lemma demonstrates another property that one might expect strongly $(p,q)$-polynomial functors to satisfy based on the properties of polynomial functions. That is, the strongly $(p,q)$-polynomial approximation of a strongly $(p,q)$-polynomial functor is the functor itself. 

\begin{lemma}\label{weiss6.3.2}
If $E$ is strongly $(p,q)$-polynomial, then $\eta: E\rightarrow\Tpq E$ is an objectwise weak equivalence. 
\end{lemma}
\begin{proof}
If $E$ is strongly $(p,q)$-polynomial, then by definition $\rho^H:E(V)^H\rightarrow (\taupq E(V))^H$ is a weak homotopy equivalence, for all $V\in \Jzero$ and all $H$ closed subgroups of $\C$. Therefore, $E(V)^H\rightarrow (\hocolim_k \taupq^k E(V))^H$ is a weak homotopy equivalence, since fixed points commute with sequential homotopy colimits, which is exactly the map $\eta$. 
\end{proof}

\begin{remark}\label{rem: (1,0)/(0,1)-polynomial is constant}
Combining this result with Example \ref{exampleT10T01} shows that a strongly $(1,0)$-polynomial functor is constant in the $\mathbb{R}$ direction, and a strongly $(0,1)$-polynomial functor is constant in the $\mathbb{R}^\delta$ direction. We can think of these types of functors as being `horizontally constant' and `vertically constant' respectively. This can be illustrated by the following diagram for a functor $E$. 

% https://q.uiver.app/?q=WzAsMTAsWzAsMywiRShcXG1hdGhiYntSfV57MCwwfSkiXSxbMSwzLCJFKFxcbWF0aGJie1J9XnsxLDB9KSJdLFszLDMsIkUoXFxtYXRoYmJ7Un1ee1xcaW5mdHksMH0pIl0sWzAsMiwiRShcXG1hdGhiYntSfV57MCwxfSkiXSxbMCwwLCJFKFxcbWF0aGJie1J9XnswLFxcaW5mdHl9KSJdLFszLDAsIkUoXFxtYXRoYmJ7Un1ee1xcaW5mdHksXFxpbmZ0eX0pIl0sWzEsMiwiRShcXG1hdGhiYntSfV57MSwxfSkiXSxbMiwzLCJcXGRvdHMiXSxbMCwxLCJcXHZkb3RzIl0sWzIsMSwiRklYIE1FIl0sWzAsMV0sWzAsM10sWzQsNSwiIiwyLHsic3R5bGUiOnsiYm9keSI6eyJuYW1lIjoiZG90dGVkIn19fV0sWzIsNSwiIiwwLHsic3R5bGUiOnsiYm9keSI6eyJuYW1lIjoiZG90dGVkIn19fV0sWzEsNl0sWzMsNl0sWzEsN10sWzcsMl0sWzMsOF0sWzgsNF1d
\[\begin{tikzcd}
	{E(\mathbb{R}^{0,\infty})} &&& {E(\mathbb{R}^{\infty,\infty})} \\
	\vdots && {\iddots} \\
	{E(\mathbb{R}^{0,1})} & {E(\mathbb{R}^{1,1})} \\
	{E(\mathbb{R}^{0,0})} & {E(\mathbb{R}^{1,0})} & \dots & {E(\mathbb{R}^{\infty,0})}
	\arrow[from=4-1, to=4-2]
	\arrow[from=4-1, to=3-1]
	\arrow[dotted, from=1-1, to=1-4]
	\arrow[dotted, from=4-4, to=1-4]
	\arrow[from=4-2, to=3-2]
	\arrow[from=3-1, to=3-2]
	\arrow[from=4-2, to=4-3]
	\arrow[from=4-3, to=4-4]
	\arrow[from=3-1, to=2-1]
	\arrow[from=2-1, to=1-1]
\end{tikzcd}\]
If $E$ is strongly $(1,0)$-polynomial, then each horizontal arrow is a weak equivalence and $E$ is `horizontally constant'. If $E$ is strongly $(0,1)$-polynomial, then each vertical arrow is a weak equivalence and $E$ is `vertically constant'. If $E$ is $(0,0)$-polynomial, then all arrows are weak equivalences and we call $E$ `constant'. 
\end{remark}

Combining Lemma \ref{weiss6.3.2} with Theorem \ref{weiss6.3.1} gives the following Corollary. 

\begin{corollary}\label{cor: TpqTpq we to Tpq}
Let $E\in\Ezero$, then $\Tpq E$ is objectwise weakly equivalent to $\Tpq \Tpq E$. 
\end{corollary}

The following Lemma is a $\C$-equivariant version \cite[Lemma 5.11]{Wei95}. It says that $\tau_{l,m}$ preserves strongly $(p,q)$-polynomial functors. 
\begin{lemma}\label{lem: 5.11}
If $E$ is strongly $(p,q)$-polynomial, then so is $\tau_{l,m}E$ for all $l,m\geq 0$. 
\end{lemma}
\begin{proof}
Since homotopy limits commute and $\tau_{l,m}$ preserves objectwise weak equivalences we get the following. 
\begin{align*}
    \tau_{p,q}\tau_{l,m}E(V) &= \tau_{l,m}\taupq E(V)\\
    &\simeq \tau_{l,m}E(V)\qedhere
\end{align*}
\end{proof}

\begin{corollary}\label{Tlm is pq poly if E is}
If $E$ is strongly $(p,q)$-polynomial, then so is $T_{l,m}E$ for all $l,m\geq 0$.    
\end{corollary}
\begin{proof}
This is clear from Lemma \ref{lem: 5.11} using that $\taupq$ commutes with sequential homotopy colimits. 
\end{proof}

Combining Lemma \ref{weiss6.3.2} and Corollary \ref{Tlm is pq poly if E is}, extends the result of Lemma \ref{weiss6.3.2} from strongly polynomial functors to polynomial functors.
\begin{lemma}
If $E$ is $(p,q)$-polynomial, then $E\simeq T_{p+1,q}T_{p,q+1}E$. \qed
\end{lemma}

\section{The $(p,q)$-polynomial model structure}
Similar to Barnes and Oman \cite{BO13}, we would like to construct a model structure on the input category $\Ezero$ (see Definition \ref{jzero and ezero def}), that captures the homotopy theory of polynomial functors. We will construct the $(p,q)$-polynomial model structure on $\Ezero$ whose fibrant objects are functors that are $(p,q)$-polynomial. We construct this model structure, using the same method as Barnes and Oman in \cite[Section 6]{BO13}, by Bousfield-Friedlander localisation and left Bousfield localisation. To do this, we will also need the projective model structure on $\Ezero$ defined in Proposition \ref{proj model structure}.

To begin, we will construct a model structure on $\Ezero$ whose fibrant objects are the strongly $(p,q)$-polynomial functors. This model structure will allow us to easily deduce results about strongly $(p,q)$-polynomial functors, without having to keep track of the more complex indexing of the $(p,q)$-polynomial model structure.  

\begin{definition}
A morphism $f\in \Ezero$ is a \emph{$\Tpq$-equivalence} if $\Tpq f$ is an objectwise weak equivalence (see Definition \ref{def: obj WE}). 
\end{definition}

\begin{proposition}\label{Tpq model structure}\index{$(p,q)\poly\Ezero^S$}
There is a proper model structure on $\Ezero$ such that a morphism $f$ is a weak equivalence if and only if it is a $\Tpq$-equivalence. The cofibrations are the same as for the projective model structure. The fibrant objects are the strongly $(p,q)$-polynomial functors. A morphism $f$ is a fibration if and only if it is an objectwise fibration and the diagram 
% https://q.uiver.app/#q=WzAsNCxbMCwwLCJYIl0sWzIsMCwiWSJdLFsyLDIsIlRfe3AscX1ZIl0sWzAsMiwiVF97cCxxfVgiXSxbMCwxLCJmIl0sWzAsMywiXFxldGEiLDJdLFsxLDIsIlxcZXRhIl0sWzMsMiwiVF97cCxxfWYiLDJdXQ==
\[\begin{tikzcd}
	X && Y \\
	\\
	{T_{p,q}X} && {T_{p,q}Y}
	\arrow["f", from=1-1, to=1-3]
	\arrow["\eta"', from=1-1, to=3-1]
	\arrow["\eta", from=1-3, to=3-3]
	\arrow["{T_{p,q}f}"', from=3-1, to=3-3]
\end{tikzcd}\]
\noindent is a homotopy pullback square in $\Ezero$. Denote this model structure by $(p,q)\poly\Ezero^S$. 
\end{proposition}

\begin{proof}
We claim that this model structure is the Bousfield-Friedlander localisation of $\Ezero$ with respect to the functor $T_{p,q}:\Ezero\rightarrow \Ezero$. Since $\Ezero$ is a proper model category, \cite[Theorem 9.3]{Bou01} applies and to prove the existence of the Bousfield-Friedlander localisation we must verify the following three axioms: 

(A1) $\Tpq$ preserves objectwise weak equivalences. 

(A2) For every $E\in \Ezero$, the morphisms $\eta_{\Tpq E}$, $\Tpq \eta_E:\Tpq E\rightarrow\Tpq \Tpq E $ are objectwise weak equivalences. 

(A3) If given a pullback square in $\Ezero$
% https://q.uiver.app/#q=WzAsNCxbMCwwLCJWIl0sWzIsMCwiWCJdLFsyLDIsIlkiXSxbMCwyLCJXIl0sWzAsMSwiayJdLFsxLDIsImYiXSxbMCwzLCJnIiwyXSxbMywyLCJoIiwyXV0=
\[\begin{tikzcd}
	V && X \\
	\\
	W && Y
	\arrow["k", from=1-1, to=1-3]
	\arrow["f", from=1-3, to=3-3]
	\arrow["g"', from=1-1, to=3-1]
	\arrow["h"', from=3-1, to=3-3]
\end{tikzcd}\]
where $f$ is a fibration of fibrant objects and $\eta_x$, $\eta_Y$ and $\Tpq h$ are all objectwise weak equivalences, then $k$ is a $\Tpq$-equivalence. 

(A1) amounts to using that taking fixed points commutes with sequential homotopy colimits, homotopy colimits preserve weak equivalences in $\Top_*$ and $\taupq$ preserves objectwise weak equivalences (see the proof of Lemma \ref{sphereandmorpharecofibrant}). 

(A2) follows by combining Theorem \ref{weiss6.3.1}, Lemma \ref{weiss6.3.2} and Corollary \ref{cor: TpqTpq we to Tpq}. 

(A3) follows in the same way as \cite[Proposition 6.5]{BO13}, and using that $\Tpq$ preserves fibrations of fibrant objects in $\Ezero$ (see the proof of Lemma \ref{sphereandmorpharecofibrant}). 

Now that we know the model structure exists, we can use the classification of fibrations provided by the localisation to classify the fibrant objects. By the localisation, $X\in \Ezero$ is fibrant if $X\rightarrow *$ is an objectwise fibration (which it is for all $X$) and the diagram below is a homotopy pullback in $\Ezero$. 
% https://q.uiver.app/#q=WzAsNCxbMCwwLCJYIl0sWzIsMCwiKiJdLFsyLDIsIlRfe3AscX0qPSoiXSxbMCwyLCJUX3twLHF9WCJdLFswLDFdLFsxLDJdLFswLDNdLFszLDJdXQ==
\[\begin{tikzcd}
	X && {*} \\
	\\
	{T_{p,q}X} && {T_{p,q}*=*}
	\arrow[from=1-1, to=1-3]
	\arrow[from=1-3, to=3-3]
	\arrow[from=1-1, to=3-1]
	\arrow[from=3-1, to=3-3]
\end{tikzcd}\]
This diagram is a homotopy pullback square if and only if $\eta:X\rightarrow \Tpq X$ is an objectwise weak equivalence. Hence we have a commutative diagram
% https://q.uiver.app/#q=WzAsNCxbMCwwLCJYIl0sWzIsMCwiXFx0YXVwcSBYIl0sWzAsMiwiVF97cCxxfVgiXSxbMiwyLCJcXHRhdXBxIFRfe3AscX1YIl0sWzAsMSwiXFxyaG8iXSxbMCwyLCJcXGV0YSIsMl0sWzIsMywiXFxyaG8iLDJdLFsxLDMsIlxcdGF1cHFcXGV0YSJdXQ==
\[\begin{tikzcd}
	X && {\taupq X} \\
	\\
	{T_{p,q}X} && {\taupq T_{p,q}X}
	\arrow["\rho", from=1-1, to=1-3]
	\arrow["\eta"', from=1-1, to=3-1]
	\arrow["\rho"', from=3-1, to=3-3]
	\arrow["\taupq\eta", from=1-3, to=3-3]
\end{tikzcd}\]
where the bottom map along with the two vertical maps are all objectwise weak equivalences, making the top map an objectwise weak equivalence, and thus making $X$ strongly $(p,q)$-polynomial as required. 
\end{proof}

\begin{remark}\label{Tpq homotopy category isoms}
Since the cofibrations of $\Ezero$ and $(p,q)\poly\Ezero^S$ are the same and $\Tpq$ preserves objectwise weak equivalences, there is a Quillen adjunction 
\begin{equation*}
    \Id:\Ezero\rightleftarrows (p,q)\poly\Ezero^S:\Id.
\end{equation*}
\end{remark}

In the same way as in \cite[Section 6]{BO13}, if we let $[-,-]^{T_{p,q}}$ denote maps in the homotopy category of $(p,q)\poly\Ezero^S$, then we find that 
\begin{equation*}
    [X,\Tpq Y]\cong [X,Y]^{T_{p,q}}.
\end{equation*}
Hence, $\Tpq$ is indeed a fibrant replacement in $(p,q)\poly\Ezero^S$. 

We could alternatively construct this model structure by a left Bousfield localisation. The benefit of this type of localisation is that we will then be able to conclude that $(p,q)\poly\Ezero^S$ is cellular. 

\begin{proposition}\label{TpqEzero as left bousfield localisation}\index{$S_{p,q}$}
The model category $(p,q)\poly\Ezero^S$ is the left Bousfield localisation of $\Ezero$ with respect to the class of maps 
\begin{equation*}
    S_{p,q}=\{S\Gmor{p,q}{V}{-}_+\rightarrow\Jzero(V,-):V\in\Jzero\}.
\end{equation*}
\end{proposition}

\begin{proof}
The proof is analogous to that of Barnes and Oman \cite[Proposition 6.6]{BO13}, where is sufficient to show that the fibrant objects of the left Bousfield localisation are the strongly $(p,q)$-polynomial functors, since both classes of cofibrations are the same.

Note that given a weak equivalence of $\C$-spaces, applying the fixed point functor $(-)^H$ and then the singular functor $S:\Top_*\rightarrow \text{sSet}$ (see \cite[Chapter 16]{May99}) gives a weak equivalence of simplicial sets. By the model theoretic properties of $\C\Top_*$-enrichment, it suffices to use $\Nat_{0,0}(\hat{c}X,Y)$, where $\hat{c}$ denotes cofibrant replacement in $\Ezero$, as a homotopy mapping object $X\rightarrow Y$ in $\Ezero$, since $\Ezero$ is enriched over $\CTop_*$ and all objects of $\Ezero$ are fibrant. Note that the domains of $S_{p,q}$ are already cofibrant by Lemma \ref{sphereandmorpharecofibrant}. 

By a theorem of Hirschhorn \cite[Theorem 4.1.1]{Hir03}, the fibrant objects of $L_{S_{p,q}}\Ezero$ are the $S_{p,q}$-local objects of $\Ezero$. That is, $X\in\Ezero$ is fibrant in $L_{S_{p,q}}\Ezero$ if 
\begin{equation*}
    \Nat_{0,0}(\Jzero(V,-),X)\rightarrow \Nat_{0,0}(S\Gmor{p,q}{V}{-}_+, X)
\end{equation*}
is a weak equivalence in $\CTop_*$, for all $V\in\Jzero$. Application of the Yoneda lemma and the definition of $\taupq$ yields that $X(V)\rightarrow \taupq X(V)$ is a weak equivalence in $\CTop_*$ for all $V\in\Jzero$, which is exactly that $X$ is strongly $(p,q)$-polynomial. 
\end{proof} 

\begin{corollary}\label{tpqisspq}
The class of $\Tpq$-equivalences is the collection of $S_{p,q}$-local equivalences. 
\end{corollary}

Using this model structure, one can finally prove that a strongly $(p,q)$-polynomial functor is indeed $(p,q)$-polynomial. The underlying version is proven in by Weiss in \cite[Proposition 5.4]{Wei95} and by Barnes and Oman in \cite[Proposition 6.7]{BO13}. 

\begin{proposition}\label{pq poly implies more poly}
If $X\in \Ezero$ is strongly $(p,q)$-polynomial, then it is $(p,q)$-polynomial.
\end{proposition}

To prove Proposition \ref{pq poly implies more poly} we will use the following proposition. 

\begin{proposition}\label{prop: weiss 4.3}
Let $\res_\R$ and $\res_{\Rdelta}$ be the restriction maps
\begin{align*}
    &\res_\R: \Jzero (\R\oplus V,W)\rightarrow \Jzero(V,W)\\
    &\res_{\Rdelta}:\Jzero(\Rdelta\oplus V,W)\rightarrow\Jzero(V,W).
\end{align*}
There exist $\C$-equivariant homeomorphisms 
\begin{align*}
    &\res^*_\R\C\gamma_{p,q}(V,W)\cong \epsilon^{p,q}_\R\oplus \C\gamma_{p,q}(\R\oplus V,W)\\
    &\res^*_{\Rdelta}\C\gamma_{p,q}(V,W)\cong \epsilon^{q,p}_{\Rdelta}\oplus \C\gamma_{p,q}(\Rdelta\oplus V,W),
\end{align*}
where $\epsilon^{m,n}_X$ is the total space of the trivial bundle 
\begin{equation*}
    \R^{m,n}\times \Jzero(X\oplus V,W)\rightarrow \Jzero(X\oplus V,W),
\end{equation*}
and $\res^*_X \C\gamma_{p,q}(V,W)$ is the pullback of the following diagram
% https://q.uiver.app/#q=WzAsMyxbMCwwLCJDXzJcXG1hdGhjYWx7Sn1fezAsMH0oWFxcb3BsdXMgVixXKSJdLFs0LDAsIkNfMlxcZ2FtbWFfe3AscX0oVixXKV8rIl0sWzIsMCwiQ18yXFxtYXRoY2Fse0p9X3swLDB9KFYsVykiXSxbMSwyLCJwXzEiLDJdLFswLDIsIlxccmVzX1giXV0=
\[\begin{tikzcd}
	{C_2\mathcal{J}_{0,0}(X\oplus V,W)} && {C_2\mathcal{J}_{0,0}(V,W)} && {C_2\gamma_{p,q}(V,W)_+}
	\arrow["{p_1}"', from=1-5, to=1-3]
	\arrow["{\res_X}", from=1-1, to=1-3]
\end{tikzcd}\]
with $p_1$ the fibre bundle map (projection onto first factor). 
\end{proposition}

\begin{proof}
The space $\res^*_\R\C\gamma_{p,q}(V,W)$ consists of pairs $(f,x)$, where $f$ is a linear isometry $\R\oplus V\rightarrow W$ and $x\in \R^{p,q}\otimes f(V)^\perp$. The space $\epsilon^{p,q}_\R\oplus \C\gamma_{p,q}(\R\oplus V,W)$ consists of triples $(f,x,y)$, where $f$ is a linear isometry $\R\oplus V\rightarrow W$, $x\in \R^{p,q}\otimes f(\R\oplus V)^\perp$ and $y\in \R^{p,q}$. We claim that the map 
\begin{align*}
    t: \epsilon^{p,q}_\R\oplus \C\gamma_{p,q}(\R\oplus V,W) &\rightarrow \res^*_\R\C\gamma_{p,q}(V,W)\\
    (f,x,y)&\mapsto (f, x+(y\otimes f(1,0))
\end{align*}
is a $\C$-equivariant homeomorphism. Indeed, the map $t$ is a homeomorphism since it is continuous and has continuous inverse 
\begin{align*}
    t^{-1}: \res^*_\R\C\gamma_{p,q}(V,W)&\rightarrow \epsilon^{p,q}_\R\oplus \C\gamma_{p,q}(\R\oplus V,W)\\
    (f,x)&\mapsto (f,x',y')
\end{align*}
where 
\begin{align*}
x'&=x-\left( \Sigma_{i=1}^{p+q}\langle x,e_i\otimes f(1,0)\rangle e_i \otimes f(1,0) \right)    \\
y'&= \Sigma_{i=1}^{p+q}\langle x,e_i\otimes f(1,0)\rangle e_i
\end{align*}
and $e_i$ is the $i$-th unit vector in $\R^{p,q}$. 

It remains to show that $t$ is $\C$-equivariant. 
\begin{align*}
    \sigma(t(f,x,y))&=\sigma (f,x+(y\otimes f(1,0))\\
    &=(\sigma *f, \sigma x + (\sigma y \otimes \sigma f(1, 0))\\
    &=(\sigma *f, \sigma x + (\sigma y \otimes (\sigma *f) (1,0))\\
    &=t(\sigma *f, \sigma x, \sigma y)\\
    &=t(\sigma (f,x,y))
\end{align*}

We now prove the second homeomorphism. The space $\res^*_{\Rdelta}\C\gamma_{p,q}(V,W)$ consists of pairs $(f,x)$, where $f$ is a linear isometry $\Rdelta\oplus V\rightarrow W$ and $x\in \R^{p,q}\otimes f(V)^\perp$. The space $\epsilon^{q,p}_{\Rdelta}\oplus \C\gamma_{p,q}(\Rdelta\oplus V,W)$ consists of triples $(f,x,y)$, where $f$ is a linear isometry $\Rdelta\oplus V\rightarrow W$, $x\in \R^{p,q}\otimes f(\Rdelta\oplus V)^\perp$ and $y\in \R^{q,p}\cong \R^{p,q}\otimes \Rdelta$. We claim that the map 
\begin{align*}
    s: \epsilon^{q,p}_{\Rdelta}\oplus \C\gamma_{p,q}(\Rdelta\oplus V,W) &\rightarrow \res^*_{\Rdelta}\C\gamma_{p,q}(V,W)\\
    (f,x,y)&\mapsto (f, x+((y\otimes 1_\delta)\otimes f(1_\delta,0))
\end{align*}
is a $\C$-equivariant homeomorphism, where $1_\delta$ is the identity element in $\Rdelta$. 

Indeed, the map $s$ is a homeomorphism since it is continuous and has continuous inverse 
\begin{align*}
    s^{-1}: \res^*_{\Rdelta}\C\gamma_{p,q}(V,W)&\rightarrow \epsilon^{q,p}_{\Rdelta}\oplus \C\gamma_{p,q}(\Rdelta\oplus V,W)\\
    (f,x)&\mapsto (f,x',y')
\end{align*}
where 
\begin{align*}
x'&=x-\left( \Sigma_{i=1}^{p+q}\langle x,e_i\otimes f(1_\delta,0)\rangle e_i \otimes f(1_\delta,0) \right)    \\
y'&= \Sigma_{i=1}^{p+q}\langle x,e_i\otimes f(1_\delta,0)\rangle e_i\otimes 1_\delta
\end{align*}
and $e_i$ is the $i$-th unit vector in $\R^{p,q}$. It remains to show that $s$ is $\C$-equivariant. 
\begin{align*}
    \sigma(s(f,x,y))&=\sigma (f,x+((y\otimes 1_\delta)\otimes f(1_\delta,0))\\
    &=(\sigma *f, \sigma x + (\sigma (y\otimes 1_\delta) \otimes \sigma f(1_\delta, 0))\\
    &=(\sigma *f, \sigma x + ((\sigma y\otimes -1_\delta) \otimes (\sigma *f)(-1_\delta, 0))\\
    &=(\sigma *f, \sigma x + (-(\sigma y\otimes 1_\delta) \otimes \sigma f (1_\delta, 0))\\
    &=(\sigma *f, \sigma x + ((\sigma y\otimes 1_\delta) \otimes -\sigma f (1_\delta, 0))\\
    &=(\sigma *f, \sigma x + ((\sigma y\otimes 1_\delta) \otimes (\sigma *f)(1_\delta, 0))\\
    &=s(\sigma *f, \sigma x, \sigma y)\\
    &=s(\sigma (f,x,y))\qedhere
\end{align*}
\end{proof}

\begin{proof}[Proof of Proposition \ref{pq poly implies more poly}]
Since the weak equivalences in $(p,q)$-poly-$\Ezero^S$ are the $S_{p,q}$-local equivalences (see Corollary \ref{tpqisspq}), it suffices to show that $S_{p+1,q}$-equivalences and $S_{p,q+1}$-equivalences are $S_{p,q}$-equivalences. This is sufficient, since it proves the existence of Quillen adjunctions
\begin{align*}
    &\Id:(p+1,q)\poly\Ezero^S \rightleftarrows (p,q)\poly\Ezero^S:\Id\\
      &\Id:(p,q+1)\poly\Ezero^S \rightleftarrows (p,q)\poly\Ezero^S:\Id,
\end{align*}
and $\Id$ being right Quillen yields the desired result. 

We start by showing that an $S_{p+1,q}$-equivalence is an $S_{p,q}$-equivalence. The proof is analogous to \cite[Proposition 6.7]{BO13}. 

We must prove that the sphere bundle map $S\C\gamma_{p+1,q}(V,-)_+\rightarrow \Jzero(V,)$ is an $S_{p,q}$-equivalence for any $V\in\Jzero$. There is a map of unit sphere bundles 
\begin{equation*}
    \alpha:S\C\gamma_{p,q}(V,-)_+\rightarrow S\C\gamma_{p+1,q}(V,-)_+
\end{equation*}
induced by the standard inclusion $\R^{p,q}\rightarrow \R^{p+1,q}$, $(x,y)\mapsto (x,0,y)$. It then suffices to show that $\alpha$ is an $S_{p,q}$-equivalence, since there is a commutative diagram 
% https://q.uiver.app/#q=WzAsMyxbMCwwLCJTXFxnYW1tYV97cCxxfShWLC0pIl0sWzIsMCwiU1xcZ2FtbWFfe3ArMSxxfShWLC0pIl0sWzIsMSwiXFxtYXRoY2Fse0p9X3swLDB9KFYsLSkiXSxbMCwxLCJcXGFscGhhIl0sWzEsMl0sWzAsMl1d
\[\begin{tikzcd}
	{S\C\gamma_{p,q}(V,-)_+} && {S\C\gamma_{p+1,q}(V,-)_+} \\
	&& {\C\mathcal{J}_{0,0}(V,-)}
	\arrow["\alpha", from=1-1, to=1-3]
	\arrow[from=1-3, to=2-3]
	\arrow[from=1-1, to=2-3]
\end{tikzcd}\]
where the diagonal map is an $S_{p,q}$-equivalence, since it is an element of $S_{p,q}$. 

The vector bundle $\C\gamma_{p+1,q}(V,-)_+$ is $\C$-equivariantly homeomorphic to the Whitney sum of vector bundles $\C\gamma_{p,q}(V,-)_+\oplus\C\gamma_{1,0}(V,-)_+$. Since the unit sphere of a Whitney sum of $\C$-vector bundles is $\C$-equivariantly equivalent to the fibrewise join of the unit sphere bundles, we know that $S\C\gamma_{p+1,q}(V,-)_+$ is the homotopy pushout in the following diagram, where $\boxtimes$ denotes the fibrewise product. 

% https://q.uiver.app/#q=WzAsNCxbMCwwLCJTXFxnYW1tYV97cCxxfShWLC0pXytcXGJveHRpbWVzIFNcXGdhbW1hX3sxLDB9KFYsLSlfKyJdLFsyLDAsIlNcXGdhbW1hX3sxLDB9KFYsLSlfKyJdLFswLDIsIlNcXGdhbW1hX3twLHF9KFYsLSlfKyJdLFsyLDIsIlNcXGdhbW1hX3twKzEscX0oViwtKV8rIl0sWzAsMSwicF8yIl0sWzEsM10sWzIsMywiXFxhbHBoYSIsMl0sWzAsMiwicF8xIiwyXV0=
\[\begin{tikzcd}
	{S\C\gamma_{p,q}(V,-)_+\boxtimes S\C\gamma_{1,0}(V,-)_+} && {S\C\gamma_{1,0}(V,-)_+} \\
	\\
	{S\C\gamma_{p,q}(V,-)_+} && {S\C\gamma_{p+1,q}(V,-)_+}
	\arrow["{p_2}", from=1-1, to=1-3]
	\arrow[from=1-3, to=3-3]
	\arrow["\alpha"', from=3-1, to=3-3]
	\arrow["{p_1}"', from=1-1, to=3-1]
\end{tikzcd}\]
Therefore, since homotopy pushouts preserve $S_{p,q}$-equivalences, it suffices to show that $p_2$ is an $S_{p,q}$-equivalence. 

By Proposition \ref{prop: weiss 4.3}, there is a pullback square 
% https://q.uiver.app/#q=WzAsNCxbMCwwLCIoXFx2YXJlcHNpbG9uXntwLHF9XFxvcGx1c1xcZ2FtbWFfe3AscX0oXFxtYXRoYmJ7Un1cXG9wbHVzIFYsLSkpXysiXSxbMiwwLCJcXGdhbW1hX3twLHF9KFYsLSlfKyJdLFswLDIsIlxcbWF0aGNhbHtKfV97MCwwfShcXG1hdGhiYntSfVxcb3BsdXMgViwtKSJdLFsyLDIsIlxcbWF0aGNhbHtKfV97MCwwfShWLC0pIl0sWzIsMywiXFxyZXNfXFxtYXRoYmJ7Un0iLDJdLFswLDFdLFsxLDNdLFswLDJdXQ==
\[\begin{tikzcd}
	{(\epsilon^{p,q}_\R\oplus\C\gamma_{p,q}(\mathbb{R}\oplus V,-))_+} && {\C\gamma_{p,q}(V,-)_+} \\
	\\
	{\C\mathcal{J}_{0,0}(\mathbb{R}\oplus V,-)} && {\C\mathcal{J}_{0,0}(V,-)}
	\arrow["{\res_\mathbb{R}}"', from=3-1, to=3-3]
	\arrow[from=1-1, to=1-3]
	\arrow[from=1-3, to=3-3]
	\arrow[from=1-1, to=3-1]
\end{tikzcd}\]
where $\epsilon^{p,q}_\R$ is the total space of the trivial bundle
\begin{equation*}
    \R^{p,q}\times \Jzero(\R\oplus V,-)\rightarrow \Jzero(\R\oplus V,-)
\end{equation*}
and $\res_\R$ is the restriction map. Note that $\Jzero(\R\oplus V,-)$ is $\C$-equivariantly homeomorphic to $S\C\gamma_{1,0}(V,-)_+$ by Proposition \ref{hocolimprop}. Hence, the diagram 
\[\begin{tikzcd}
	{S(\epsilon^{p,q}_\R\oplus\C\gamma_{p,q}(\mathbb{R}\oplus V,-))_+} && {S\C\gamma_{p,q}(V,-)_+} \\
	\\
	{S\C\gamma_{1,0}(V,-)_+} && {S\C\mathcal{J}_{0,0}(V,-)}
	\arrow["{\res_\mathbb{R}}"', from=3-1, to=3-3]
	\arrow[from=1-1, to=1-3]
	\arrow[from=1-3, to=3-3]
	\arrow[from=1-1, to=3-1]
\end{tikzcd}\]
is a homotopy pullback square. The homotopy pullback is $S\C\gamma_{p,q}(V,-)_+\boxtimes S\C\gamma_{1,0}(V,-)_+$, by the definition of the fibrewise product. Therefore, the map $p_2$ can be $\C$-equivariantly identified as the sphere bundle map 
\begin{equation*}
    S(\epsilon^{p,q}_\R\oplus\C\gamma_{p,q}(\mathbb{R}\oplus V,-))_+\rightarrow \Jzero(\R\oplus V,-).
\end{equation*}

Since the unit sphere of a Whitney sum of $\C$-vector bundles is $\C$-equivariantly equivalent to the fibrewise join of the unit sphere bundles, the following diagram is a homotopy pushout.   
% https://q.uiver.app/#q=WzAsNCxbMCwwLCJTXntwLTEscX1fKyBcXHdlZGdlIFNcXGdhbW1hX3twLHF9KFxcbWF0aGJie1J9XFxvcGx1cyBWLC0pXysiXSxbMiwwLCJTXntwLTEscX1fK1xcd2VkZ2UgXFxtYXRoY2Fse0p9X3swLDB9KFxcbWF0aGJie1J9XFxvcGx1cyBWLC0pIl0sWzAsMiwiU1xcZ2FtbWFfe3AscX0oXFxtYXRoYmJ7Un1cXG9wbHVzIFYsLSlfKyJdLFsyLDIsIlMoXFxlcHNpbG9uXntwLHF9XFxvcGx1c1xcZ2FtbWFfe3AscX0oXFxtYXRoYmJ7Un1cXG9wbHVzIFYsLSkpXysiXSxbMiwzLCJcXG11IiwyXSxbMCwxLCJcXGRlbHRhIl0sWzEsM10sWzAsMl1d
\[\begin{tikzcd}
	{S^{p-1,q}_+ \wedge S\C\gamma_{p,q}(\mathbb{R}\oplus V,-)_+} && {S^{p-1,q}_+\wedge \C\mathcal{J}_{0,0}(\mathbb{R}\oplus V,-)} \\
	\\
	{S\C\gamma_{p,q}(\mathbb{R}\oplus V,-)_+} && {S(\epsilon^{p,q}_\R\oplus\C\gamma_{p,q}(\mathbb{R}\oplus V,-))_+}
	\arrow["\mu"', from=3-1, to=3-3]
	\arrow["\delta", from=1-1, to=1-3]
	\arrow[from=1-3, to=3-3]
	\arrow[from=1-1, to=3-1]
\end{tikzcd}\]
The map $\delta$ is an $S_{p,q}$-equivalence, since it is an element of $S_{p,q}$ smashed with a $\C$-CW complex. Therefore, the pushout $\mu$ is also an $S_{p,q}$-equivalence. 

There is a commutative diagram 
% https://q.uiver.app/#q=WzAsMyxbMCwwLCJTXFxnYW1tYV97cCxxfShcXG1hdGhiYntSfVxcb3BsdXMgViwtKV8rIl0sWzIsMCwiUyhcXGVwc2lsb25ee3AscX1cXG9wbHVzXFxnYW1tYV97cCxxfShcXG1hdGhiYntSfVxcb3BsdXMgViwtKSlfKyJdLFsyLDEsIlxcbWF0aGNhbHtKfV97MCwwfShcXG1hdGhiYntSfVxcb3BsdXMgViwtKSJdLFsxLDIsInBfMiJdLFswLDEsIlxcbXUiXSxbMCwyXV0=
\[\begin{tikzcd}
	{S\C\gamma_{p,q}(\mathbb{R}\oplus V,-)_+} && {S(\epsilon^{p,q}_\R\oplus\C\gamma_{p,q}(\mathbb{R}\oplus V,-))_+} \\
	&& {\C\mathcal{J}_{0,0}(\mathbb{R}\oplus V,-)}
	\arrow["{p_2}", from=1-3, to=2-3]
	\arrow["\mu", from=1-1, to=1-3]
	\arrow[from=1-1, to=2-3]
\end{tikzcd}\]
where $\mu$ is an $S_{p,q}$-equivalence, by above, and the diagonal map is an $S_{p,q}$-equivalence, since it is an element of $S_{p,q}$. Hence, by the two-out-of-three property, the map $p_2$ is an $S_{p,q}$-equivalence as required.

Proving that an $S_{p,q+1}$-equivalence is an $S_{p,q}$-equivalence is similar, and uses the other case of Proposition \ref{prop: weiss 4.3}.\end{proof}

\begin{corollary}\label{Tpq Tmn adjunction}
There exists a Quillen adjunction 
\begin{equation*}
      \Id:(p,q)\poly\Ezero^S\rightleftarrows (m,n)\poly\Ezero^S;\Id
\end{equation*}
for $m<p$ and $n<q$.
\end{corollary}

Combining Proposition \ref{pq poly implies more poly} with Theorem \ref{weiss6.3.1} and Lemma \ref{weiss6.3.2}, gives an important result about how the functors $\Tpq$ interact. 

\begin{corollary}\label{tlmtpq=tpq}
Let $E\in\Ezero$. If $l\geq p$ and $m\geq q$, then $$T_{p,q}T_{l,m}E\simeq T_{l,m}T_{p,q}E\simeq T_{p,q}E.$$
\end{corollary}

We now construct another model structure on $\Ezero$. We call this model structure the $(p,q)$-polynomial model structure, and denote it by $(p,q)$-poly-$\Ezero$. The fibrant objects of this model structure will be the functors which are $(p,q)$-polynomial. The construction of this model structure is done in a similar way as for the model structure $(p,q)$-poly-$\Ezero^S$. We use a Bousfield-Friedlander localisation to prove the existence of the model structure and then use a left Bousfield localisation to show that it is cellular. This ensures that we can construct a right Bousfield localisation of the $(p,q)$-polynomial model structure and, by choosing the right set of objects to localise at, this right Bousfield localisation has fibrant-cofibrant objects which are the $(p,q)$-homogeneous functors, see Section \ref{sec:homog ms}. 

\begin{proposition}\label{polynomial model structure}\index{$(p,q)\poly\Ezero$}
There is a proper model structure on $\Ezero$ such that a morphism $f$ is a weak equivalence if and only if it is a $T_{p+1,q}T_{p,q+1}$-equivalence. The cofibrations are the same as for the projective model structure. The fibrant objects are the functors that are $(p,q)$-polynomial. A morphism $f$ is a fibration if and only if it is an objectwise fibration and the diagram 
% https://q.uiver.app/#q=WzAsNCxbMCwwLCJYIl0sWzIsMCwiWSJdLFswLDIsIlRfe3ArMSxxfVRfe3AscSsxfVgiXSxbMiwyLCJUX3twKzEscX1UX3twLHErMX1ZIl0sWzIsMywiVF97cCsxLHF9VF97cCxxKzF9ZiIsMl0sWzAsMSwiZiJdLFsxLDMsIlxcZXRhIl0sWzAsMiwiXFxldGEiLDJdXQ==
\[\begin{tikzcd}
	X && Y \\
	\\
	{T_{p+1,q}T_{p,q+1}X} && {T_{p+1,q}T_{p,q+1}Y}
	\arrow["{T_{p+1,q}T_{p,q+1}f}"', from=3-1, to=3-3]
	\arrow["f", from=1-1, to=1-3]
	\arrow["\eta", from=1-3, to=3-3]
	\arrow["\eta"', from=1-1, to=3-1]
\end{tikzcd}\] is a homotopy pullback square in $\Ezero$. Denote this model structure by $(p,q)\poly\Ezero$.
\end{proposition}

The proof is similar to Proposition \ref{Tpq model structure}, using $T_{p+1,q}T_{p,q+1}$ in place of $\Tpq$, so we omit it here. 

\begin{remark}\label{ezero and poly adjunction}
Since the cofibrations of $\Ezero$ and $(p,q)$-poly-$\Ezero$ are the same and $T_{p+1,q}T_{p,q+1}$ preserves objectwise weak equivalences, there is again a Quillen adjunction 
\begin{equation*}
    \Id:\Ezero\rightleftarrows (p,q)\poly\Ezero:\Id,
\end{equation*}
and similarly to Corollary \ref{Tpq Tmn adjunction}, there is a Quillen adjunction 
\begin{equation*}
    \Id:(p,q)\poly\Ezero\rightleftarrows (m,n)\poly\Ezero:\Id
\end{equation*}
where $m\leq p$ and $n\leq q$.

\noindent Additionally, the two localisations are related by the following Quillen adjunction.
\begin{equation*}
    \Id:(p,q)\poly\Ezero\rightleftarrows (p,q)\poly\Ezero^S:\Id
\end{equation*}

\end{remark}

We will now construct this model structure again, this time by a left Bousfield localisation. 

\begin{proposition}
The model category $(p,q)\poly\Ezero$ is the left Bousfield localisation of $\Ezero$ with respect to the class of maps $S_{p+1,q}\coprod S_{p,q+1}$ where 
\begin{align*}
    S_{p+1,q}&=\{S\Gmor{p+1,q}{V}{-}_+\rightarrow\Jzero(V,-):V\in\Jzero\}\\
    S_{p,q+1}&=\{S\Gmor{p,q+1}{V}{-}_+\rightarrow\Jzero(V,-):V\in\Jzero\}.
\end{align*}
\end{proposition}

\begin{proof}The proof follows in the same way as for Proposition \ref{TpqEzero as left bousfield localisation}. 
\end{proof}

\chapter{Equivariant homogeneous functors}
\label{ch:EquivHomogFunct}

\fancyhf{}
\fancyhead[C]{\rightmark}
\fancyhead[R]{\thepage}

In orthogonal calculus, the homotopy fibre of the map $T_n F\rightarrow T_{n-1} F$ is $n$-polynomial and has trivial $(n-1)$-polynomial approximation. Functors of this type are called $n$-homogeneous, and are completely determined by orthogonal spectra with $O(n)$-action, see \cite[Theorem 7.3]{Wei95}. In this section we define a new class of $(p,q)$-homogeneous functors in the $\C$-equivariant input category. The main goal will be to classify such functors by a category of spectra, as is done in the underlying calculus. The relation between the $(p,q)$-homogeneous model structure and the stable model structure in Proposition \ref{prop: stable ms}, forms one half of the zig-zag of equivalences that gives this classification, see Theorem \ref{zigzagclassification}.

\section{Homogeneous functors}\label{sec: homog functors}

We want to use the properties of the homotopy fibre $$D_{p,q} F\rightarrow T_{p+1,q}T_{p,q+1}F\rightarrow \Tpq F$$to define a class of $(p,q)$-homogeneous functors in the input category $\Ezero$ (see Definition \ref{jzero and ezero def}). First, we will need to determine what these properties are, and to do so we will need to give some important and useful properties of polynomial functors. Many of these properties will also be needed when constructing the homogeneous model structure in Section \ref{sec:homog ms}. 

The following is the $\C$-version of \cite[Lemma 5.5]{Wei95}.

\begin{lemma}\label{lem: c2 wes 5.5}
Let $g:E\rightarrow G$ in $\Ezero$ be such that $\ind_{0,0}^{p,q}G$ is objectwise contractible and $E$ is strongly $(p,q)$-polynomial, then the homotopy fibre of $g$ is strongly $(p,q)$-polynomial. 
\end{lemma}

\begin{proof}
This follows, as in \cite{Wei95}, from the homotopy fibre sequence in Lemma \ref{hofiblemma}.  
\end{proof}

\begin{corollary}\label{cor: c2 weiss 5.5}
Let $g:E\rightarrow F$ in $\Ezero$ be such that $\ind_{0,0}^{p+1,q}F$ and $\ind_{0,0}^{p,q+1}F$ are both objectwise contractible and $E$ is $(p,q)$-polynomial, then the homotopy fibre of $g$ is $(p,q)$-polynomial. \qed
\end{corollary}

The following is a $\C$-generalisation of \cite[Corollary 5.6]{Wei95}, which is an instant consequence Corollary \ref{cor: c2 weiss 5.5} by setting $E=*$. 
\begin{corollary}\label{cor: F delooping}
Let $F\in\Ezero$ be such that $\ind_{0,0}^{p+1,q}F$ and $\ind_{0,0}^{p,q+1}F$ are both objectwise contractible, then the functor 
\begin{equation*}
    V\mapsto \Omega F(V)
\end{equation*}
is $(p,q)$-polynomial. \qed
\end{corollary}

Now we can determine the properties of the homotopy fibre $D_{p,q}F$. 

\begin{definition}\label{def: pq reduced}
Let $E\in \Ezero$. Define $E$ to be \emph{$(p,q)$-reduced} if $\Tpq E$ is objectwise contractible. 
\end{definition}

\begin{remark}
    Note that if $E\in\Ezero$ is $(p,q)$-reduced, then $E$ is also $(a,b)$-reduced for all pairs $(a,b)$ with $0\leq a\leq p$ and $0\leq b\leq q$. This follows from Corollary \ref{tlmtpq=tpq}.
\end{remark}

\begin{theorem}\label{thm: DYpq is homog}\index{$D_{p,q}$}
The homotopy fibre $D_{p,q} F$ of the map $r_{p,q}: T_{p+1,q}T_{p,q+1} F\rightarrow T_{p,q} F$ is $(p,q)$-polynomial and $(p,q)$-reduced.
\end{theorem}
\begin{proof}
Indeed, applying $\Tpq$ to the homotopy fibre sequence and applying Corollary \ref{tlmtpq=tpq} gives
% https://q.uiver.app/#q=WzAsNixbMCwwLCJUX3twLHF9RF97cCxxfV5ZIEYiXSxbMSwwLCJUX3twLHF9VF97cCsxLHF9VF97cCxxKzF9RiJdLFsyLDAsIlRfe3AscX1UX3twLHF9RiJdLFsxLDEsIlRfe3AscX1GIl0sWzIsMSwiVF97cCxxfUYiXSxbMCwxLCIqIl0sWzAsMV0sWzEsMl0sWzEsMywiXFxzaW1lcSIsMyx7InN0eWxlIjp7ImJvZHkiOnsibmFtZSI6Im5vbmUifSwiaGVhZCI6eyJuYW1lIjoibm9uZSJ9fX1dLFsyLDQsIlxcc2ltZXEiLDMseyJzdHlsZSI6eyJib2R5Ijp7Im5hbWUiOiJub25lIn0sImhlYWQiOnsibmFtZSI6Im5vbmUifX19XSxbMyw0XSxbNSwzXV0=
\[\begin{tikzcd}
	{T_{p,q}D_{p,q} F} & {T_{p,q}T_{p+1,q}T_{p,q+1}F} & {T_{p,q}T_{p,q}F} \\
	& {T_{p,q}F} & {T_{p,q}F}
	\arrow[from=1-1, to=1-2]
	\arrow["r_{p,q}", from=1-2, to=1-3]
	\arrow["\simeq"{marking}, draw=none, from=1-2, to=2-2]
	\arrow["\simeq"{marking}, draw=none, from=1-3, to=2-3]
	\arrow["\id"',from=2-2, to=2-3]
\end{tikzcd}\]
Therefore, $\Tpq D_{p,q}F\simeq \hofibre[\Tpq F\overset{\id}{\rightarrow} \Tpq F]\simeq *$. That is, $ D_{p,q}F$ is $(p,q)$-reduced (see Definition \ref{def: pq reduced}). 

The functors $T_{p+1,q}T_{p,q+1}F$ and $\Tpq F$ are both $(p,q)$-polynomial (see Definition \ref{def: polynomial}). This follows from Theorem \ref{weiss6.3.1} and Proposition \ref{pq poly implies more poly}. Therefore, $D_{p,q} F$ is $(p,q)$-polynomial, by Corollary \ref{cor: c2 weiss 5.5}.
\end{proof}

\begin{definition}
Let $E\in \Ezero$. $E$ is defined to be \emph{$(p,q)$-homogeneous} if it is $(p,q)$-polynomial and $(p,q)$-reduced.
\end{definition}

We now define the equivariant generalisation of the term `connected at infinity', this is analogous to \cite[Definition 5.9]{Wei95}. 

\begin{definition}
$E\in\Ezero$ is defined to be \emph{connected at infinity} if the $\C$-space 
\begin{equation*}
    \underset{a,b}{\hocolim} E(\R^{a,b})=:E(\R^{\infty,\infty})
\end{equation*}
is connected (meaning that the equivariant homotopy groups $\pi_0^H E(\R^{\infty,\infty})$ are trivial for all closed subgroups $H$ of $\C$). 
\end{definition}

\begin{lemma}\label{lem: homog implies connected at infinity}
If $E$ is $(p,q)$-homogeneous for pairs $(p,q)$ where at least one of $p,q$ is greater than zero, then $E$ is connected at infinity. 
\end{lemma}
\begin{proof}
This is a straightforward calculation that follows from Example \ref{ex: 00-poly approx}. 
\end{proof}
 
The next result is \cite[Proposition 5.10]{Wei95}. It is used in conjunction with Lemma \ref{lem: homog implies connected at infinity} to construct the $(p,q)$-homogeneous model structure. The proof follows as in \cite[Lemma 5.10]{BO13}, replacing $T_n$ by $T_{p,q}$, so we omit it here.

\begin{lemma}\label{conn at infinity lemma}
Let $g:E\rightarrow F$ be a map between strongly $(p,q)$-polynomial objects such that the homotopy fibre of $g$ is objectwise contractible and $F$ is connected at infinity. Then $g$ is an objectwise weak equivalence. 
\end{lemma}

Finally we will prove a version of \cite[Corollary 5.12]{Wei95}. We will use this result later, along with Lemma \ref{omega spectra lemma}, to show that the induction functor takes objects that are $(p,q)$-polynomial to $(p,q)\Omega$-spectra. This guarantees that induction is a right Quillen functor from the $(p,q)$-polynomial model structure on $\Ezero$ to the $(p,q)$-stable model structure on $\Epq$.

\begin{proposition}\label{5.12}
Let $E\in \Ezero$ be $(p,q)$-polynomial. Then for all $V\in\Jzero$ there exist weak equivalences of $\C$-spaces
\begin{align*}
    \ind_{0,0}^{p,q}E(V)&\rightarrow\Omega^{p,q\R}\ind_{0,0}^{p,q}E(V\oplus\R)\\
    \ind_{0,0}^{p,q}E(V)&\rightarrow\Omega^{p,q\Rdelta}\ind_{0,0}^{p,q}E(V\oplus\Rdelta).
\end{align*}
\end{proposition}

\begin{proof}
We will prove the first weak equivalence, since the second follows by a similar argument.

If $p=q=0$, then there is nothing to prove, since $E$ is constant (see Remark \ref{rem: (1,0)/(0,1)-polynomial is constant}) and $\ind_{0,0}^{0,0}E\simeq E$ by the enriched Yoneda lemma. 

Let $p,q$ be such that at least one of $p,q$ is non-zero. By Proposition \ref{loops fibre sequence}, there exists a $\C$-homotopy fibre sequence
\begin{equation*}
    \res_{p,q}^{p+1,q}\ind_{0,0}^{p+1,q}E(V)\rightarrow \ind_{0,0}^{p,q}E(V)\rightarrow\Omega^{p,q\R}\ind_{0,0}^{p,q}E(V\oplus\R).
\end{equation*}
We know that $\ind_{0,0}^{p+1,q}E(V)$ is a contractible $\C$-space, since $E$ is strongly $(p+1,q)$-polynomial and by Corollary \ref{poly implies ind contractible}. Thus, if we can show that both $\ind_{0,0}^{p,q}E$ and $$F:V\mapsto \Omega^{p,q\R}\ind_{0,0}^{p,q}E(V\oplus\R)$$ are strongly $(p+1,q)$-polynomial and $F$ is connected at infinity, then Lemma \ref{conn at infinity lemma} gives the weak equivalence. Showing this is the same as for \cite[Proposition 5.12]{BO13}.
\end{proof}

\section{The $(p,q)$-homogeneous model structure}\label{sec:homog ms}
We now construct a right Bousfield localisation of the $(p,q)$-polynomial model structure in order to build a model structure on $\Ezero$ analogous to the $n$-homogeneous structure in \cite[Proposition 6.9]{BO13}. The cofibrant-fibrant objects of this model structure are the projectively cofibrant $(p,q)$-homogeneous functors and the weak equivalences are detected by derivatives. This model structure allows for the classification of $(p,q)$-homogeneous functors in terms of a Quillen equivalence, see Theorem \ref{boclassification}. 

\begin{proposition}\index{$(p,q)\homog\Ezero$}
There exists a model structure on $\Ezero$ whose cofibrant-fibrant objects are the $(p,q)$-homogeneous functors that are cofibrant in the projective model structure on $\Ezero$. Fibrations are the same as $(p,q)\poly\Ezero$ and weak equivalences are morphisms $f$ such that $\res_{0,0}^{p,q}\ind_{0,0}^{p,q}T_{p+1,q}T_{p,q+1}f$ is an objectwise weak equivalence. 
We call this the $(p,q)$-homogeneous model structure on $\Ezero$ and denote it by $(p,q)\homog\Ezero$. 

\noindent There is a Quillen adjunction 
\begin{equation*}
    \Id: (p,q)\homog\Ezero\rightleftarrows (p,q)\poly\Ezero:\Id
\end{equation*}
\end{proposition}

\begin{proof}We show that right Bousfield localisation of $(p,q)$-poly-$\Ezero$ with respect to the set of objects 
\begin{equation*}
    K_{p,q}=\{\Jpq(V,-):V\in \Jzero\}
\end{equation*}
yields the desired model structure. 

Since $(p,q)$-poly-$\Ezero$ is proper and cellular, we know that the right Bousfield localisation $R_{K_{p,q}}((p,q)\poly\Ezero)$ exists by \cite[Theorem 5.1.1]{Hir03}. 

It suffices (by the same argument used in the proof of Proposition \ref{TpqEzero as left bousfield localisation}) to use $$\Nat_{0,0}(\hat{c}X,T_{p+1,q}T_{p,q+1}Y),$$ where $\hat{c}$ denotes cofibrant replacement in $(p,q)\poly\Ezero$, as a homotopy mapping object $X\rightarrow Y$ in $(p,q)\poly\Ezero$, since $(p,q)\poly\Ezero$ is enriched over $\CTop_*$ and $T_{p+1,q}T_{p,q+1}$ is the fibrant replacement. Note that all elements of $K_{p,q}$ are already cofibrant by Lemma \ref{sphereandmorpharecofibrant}, since the cofibrant object of $(p,q)$-poly-$\Ezero$ are the same as for $\Ezero$. 

The weak equivalences of $R_{K_{p,q}}((p,q)\poly\Ezero)$ are the $K_{p,q}$-colocal equivalences. That is, $f:X\rightarrow Y$ is a weak equivalence in $R_{K_{p,q}}((p,q)\poly\Ezero)$ if and only if 
\begin{equation*}
    \Nat_{0,0}(\Jpq(V,-),T_{p+1,q}T_{p,q+1}X)\rightarrow \Nat_{0,0}(\Jpq(V,-),T_{p+1,q}T_{p,q+1}Y)
\end{equation*}
is a weak equivalence in $\CTop_*$, for all $V\in\Jzero$. By Definition \ref{induction definition}, we see that this map is exactly $\res_{0,0}^{p,q}\ind_{0,0}^{p,q}T_{p+1,q}T_{p,q+1}f(V)$. Therefore, $f$ is a weak equivalence if and only if $\res_{0,0}^{p,q}\ind_{0,0}^{p,q}T_{p+1,q}T_{p,q+1}f$ is an objectwise weak equivalence as desired. 

The fibrations are the same as for $(p,q)$-poly-$\Ezero$, since fibrations are unchanged by right Bousfield localisation. Thus, the fibrant objects are the $(p,q)$-polynomial functors. 

The cofibrant objects are those functors $X$ such that $X$ is cofibrant in $\Ezero$ and such that for all $K_{p,q}$-colocal equivalences $f:A\rightarrow B$ the map 
\begin{equation*}
    \Nat_{0,0}(X,T_{p+1,q}T_{p,q+1}A)\rightarrow \Nat_{0,0}(X,T_{p+1,q}T_{p,q+1}B)
\end{equation*}
is a weak equivalence in $\CTop_*$. 

We now show that the cofibrant-fibrant objects of $R_{K_{p,q}}((p,q)\poly\Ezero)$ are the projectively cofibrant $(p,q)$-homogeneous functors. 

Let $A$ be a cofibrant-fibrant object of $R_{K_{p,q}}((p,q)\poly\Ezero)$. Then $A$ is $(p,q)$-polynomial and projectively cofibrant by above. There is a $K_{p,q}$-colocal equivalence $*\rightarrow \Tpq A$, since
\begin{align*}
\res_{0,0}^{p,q}\ind_{0,0}^{p,q}T_{p+1,q}T_{p,q+1} \Tpq A (V) &\simeq \res_{0,0}^{p,q}\ind_{0,0}^{p,q}\Tpq A (V) \\
&\simeq * \\
&\simeq \res_{0,0}^{p,q}\ind_{0,0}^{p,q}T_{p+1,q}T_{p,q+1} * (V)
\end{align*}
where the first equivalence is by Corollary \ref{tlmtpq=tpq} and the second equivalence is by Theorem \ref{weiss6.3.1} and Corollary \ref{poly implies ind contractible}. Therefore, there is a weak equivalence of $\C$-spaces 
\begin{equation*}
    \Nat_{0,0}(A,*)\rightarrow\Nat_{0,0}(A,\Tpq A)
\end{equation*}
Using this map and Remark \ref{Tpq homotopy category isoms}, there are isomorphisms 
\begin{equation*}
    0=[A,\Tpq A]\cong [A,A]^{\Tpq}\cong [\Tpq A, \Tpq A]^{\Tpq}\cong [\Tpq A, \Tpq A].
\end{equation*}
Hence, $\Tpq A$ is objectwise contractible, and $A$ is $(p,q)$-homogeneous. 

Let $B$ be a $(p,q)$-homogeneous functor that is cofibrant in $\Ezero$. Then $B$ is fibrant in $R_{K_{p,q}}((p,q)\poly\Ezero)$, since it is $(p,q)$-polynomial. It is left to show that $B$ is cofibrant. 

Let $f:\hat{c}B\rightarrow B$ be the cofibrant replacement of $B$ in $(p,q)$-homog-$\Ezero$. Since $B$ is fibrant, so is $\hat{c}B$. Therefore, $\res_{0,0}^{p,q}\ind_{0,0}^{p,q}f$ is an objectwise weak equivalence. If we can show that $f$ is an objectwise weak equivalence, then it will follow that $B$ is cofibrant. 

Let $D$ be the homotopy fibre of $f$. Since both $B$ and $\hat{c}B$ are fibrant, so is $D$ (see Corollary \ref{cor: c2 weiss 5.5}. Since $\res_{0,0}^{p,q}\ind_{0,0}^{p,q}f$ is an objectwise weak equivalence, $\res_{0,0}^{p,q}\ind_{0,0}^{p,q}D$ is objectwise contractible. Hence, we have a homotopy fibre sequence of fibrant ($(p,q)$-polynomial) objects 
\begin{equation*}
    \res_{0,0}^{p,q}\ind_{0,0}^{p,q}D\rightarrow D\rightarrow \taupq D
\end{equation*}
whose homotopy fibre is objectwise contractible. By Lemma \ref{conn at infinity lemma}, if $\taupq D$ is connected at infinity, then $D\rightarrow \taupq D$ is an objectwise weak equivalence. This is true, since $\taupq$ commutes with sequential homotopy colimits and $D$ is connected at infinity as it is $(p,q)$-homogeneous (see Lemma \ref{lem: homog implies connected at infinity}). Therefore, $D$ is strongly $(p,q)$-polynomial.  

We know that $D$ is $(p,q)$-polynomial (fibrant) from above. It remains to check that $T_{p,q}D$ is objectwise contractible in order to confirm that $D$ is indeed $(p,q)$-homogeneous as stated above. The sequential homotopy colimit used to describe $\Tpq D$ commutes with homotopy pullbacks. Therefore, there is a homotopy fibre sequence 
\begin{equation*}
    \Tpq D\rightarrow \Tpq \hat{c}B\rightarrow \Tpq B.
\end{equation*}
The functors $\hat{c}B$ and $B$ are both $(p,q)$-homogeneous ($B$ by assumption and $\hat{c}B$ by cofibrant-fibrant). Therefore, $\Tpq\hat{c}B$ and $\Tpq B$ are both objectwise contractible, which implies that $\Tpq D$ is objectwise contractible as required. 

Since $D$ is strongly $(p,q)$-polynomial, $D$ is objectwise weakly equivalent to $\Tpq D$. Hence, $D$ is objectwise contractible, which by Lemma \ref{conn at infinity lemma} implies that $f$ is an objectwise weak equivalence as desired. 

If $f$ is a $T_{p+1,q}T_{p,q+1}$-equivalence, then $\res_{0,0}^{p,q}\ind_{0,0}^{p,q}T_{p+1,q}T_{p,q+1}f$ is an objectwise weak equivalence, since $\res_{0,0}^{p,q}\ind_{0,0}^{p,q}$ preserves objectwise weak equivalences. Therefore, $\Id$ is right Quillen, and the adjunction exists. 
\end{proof}

\begin{remark}
Detecting weak equivalences via $\ind_{0,0}^{p,q}T_{p+1,q}T_{p,q+1}$ can be difficult, since the induction functor $\ind_{0,0}^{p,q}$ is complex. In unitary calculus, Taggart shows that a map is an $\ind_{0}^n T_n$-equivalence if and only if it is a $D_n$-equivalence, where $D_nF$ is the homotopy fibre of $T_nF\rightarrow T_{n-1}F$ (see \cite[Proposition 8.2]{Tag22unit}). Via a similar proof, one can show that a map is an $\ind_{0,0}^{p,q}T_{p+1,q}T_{p,q+1}$-equivalence if and only if it is a $D_{p,q}$-equivalence.  
\end{remark}

\chapter{The equivariant classification theorem}
\label{ch:EquivClass}

\fancyhf{}
\fancyhead[C]{\rightmark}
\fancyhead[R]{\thepage}

The main result of orthogonal calculus is the classification of $n$-homogeneous functors, as functors fully determined by a category of spectra. This is given by the classification theorem of Weiss \cite[Theorem 7.3]{Wei95}, which states that an $n$-homogeneous functor is weakly equivalent to a functor of the form $V\mapsto \Omega^\infty [(S^{nV}\wedge \Psi)_{hO(n)}]$, where $\Psi$ is an orthogonal spectrum with an action of $O(n)$. In \cite{BO13}, this classification is derived as a Quillen equivalence on the model categories constructed, see Section \ref{sec: classification n-homog}. As a result, the homotopy fibres of the maps $T_{n+1}X\rightarrow T_nX$ for an input functor $X$, which are $n$-homogeneous, stand a chance of begin computed. 

In this chapter, we construct two Quillen equivalences. These Quillen equivalences form a zig-zag of equivalences between the $(p,q)$-homogeneous model structure on $\Ezero$ and the stable model structure on $\C Sp^O[O(p,q)]$. 
% https://q.uiver.app/?q=WzAsMyxbMCwwLCIocCxxKVxcdGV4dHstaG9tb2ctfUNfMlxcbWF0aGNhbHtFfV97MCwwfSJdLFs0LDAsIkNfMlNwXlxcbWF0aGNhbHtPfVtPKHArcSldIl0sWzIsMCwiTyhwK3EpQ18yXFxtYXRoY2Fse0V9X3twLHF9XnMiXSxbMCwyLCJcXHRleHR7aW5kfV97MCwwfV57cCxxfVxcdmFyZXBzaWxvbl4qIiwyLHsib2Zmc2V0IjoyfV0sWzIsMCwiXFx0ZXh0e3Jlc31fezAsMH1ee3AscX0vTyhwK3EpIiwyLHsib2Zmc2V0IjoyfV0sWzIsMSwiKFxcYWxwaGFfe3AscX0pXyEiLDAseyJvZmZzZXQiOi0yfV0sWzEsMiwiXFxhbHBoYV97cCxxfV4qIiwwLHsib2Zmc2V0IjotMn1dXQ==
\[\begin{tikzcd}
	{(p,q)\homog C_2\mathcal{E}_{0,0}} && {O(p,q)C_2\mathcal{E}_{p,q}^s} && {C_2Sp^O[O(p,q)]}
	\arrow["{\ind_{0,0}^{p,q}\varepsilon^*}"', shift right=2, from=1-1, to=1-3]
	\arrow["{\res_{0,0}^{p,q}/O(p,q)}"', shift right=2, from=1-3, to=1-1]
	\arrow["{(\alpha_{p,q})_!}", shift left=2, from=1-3, to=1-5]
	\arrow["{\alpha_{p,q}^*}", shift left=2, from=1-5, to=1-3]
\end{tikzcd}\]

In the same way as for Barnes and Oman in \cite[Section 10]{BO13}, this leads to a classification theorem for $(p,q)$-homogeneous functors (Theorem \ref{weissclassification}), as functors fully determined by genuine orthogonal $\C$-spectra with an action of $O(p,q)$.

\section{The intermediate category as a category of spectra}

In this section we show that the $(p,q)$-th intermediate category $\OEpq$ is Quillen equivalent to the category of $O(p,q)$-equivariant objects in orthogonal $\C$-spectra. Thus, the $(p,q)$-derivative of an input functor (see Definition \ref{jzero and ezero def}) can be described in terms of these spectra. This section is analogous to \cite[Section 8]{BO13}, where the only differences are due to the additional $\C$-action, which does not effect the $O(p,q)$-equivariance of maps considered. The resulting Quillen equivalence 
\begin{equation*}
    (\alpha_{p,q})_{!}: \OEpq \rightleftarrows \C Sp^O [O(p,q)] : \alpha_{p,q}^*
\end{equation*}   
forms one half of the zig-zag of equivalences which gives the classification of $(p,q)$-homogeneous functors, see Theorem \ref{zigzagclassification}. The spectrum $(\alpha_{p,q})_{!}F$ for a functor $F\in \OEpq$ is the categorification of the spectrum $\Theta F$ constructed in \cite[Section 2]{Wei95}.

We begin by describing the category of orthogonal $\C$-spectra. Details of these constructions have been discussed by Mandell and May in \cite[Section II.4]{MM02}. 

\begin{definition}\label{def: orth c2 spectra}\index{$\C Sp^O$}
The \emph{category of orthogonal $\C$-spectra} $\C Sp^O$ is the category $\C\mathcal{E}_{1,0}$.

\noindent This category has a cofibrantly generated proper stable model structure where the cofibrations are $q$-cofibrations and the weak equivalences are the $(1,0)\pi_*$-equivalences (see Definition \ref{def: pq pi equiv}). That is, $f$ if a weak equivalence if $(1,0)\pi_k^H f$ is an isomorphism for all closed subgroups $H\leq \C$ and integers $k$. It is cofibrantly generated by the following sets of generating cofibrations and generating acyclic cofibrations respectively
\begin{align*}
 \{&F_Vi:i\in I_{\C}\}\\
 \{&F_Vj:j\in J_{\C}\}
\end{align*}
where $F_V:\CTop_*\rightarrow \C Sp^O$ is the left adjoint to evaluation at $V$, and $I_{\C},J_{\C}$ are the generating cofibrations and acyclic cofibrations for $\CTop_*$ (see Proposition \ref{finemodelstructure}). 
\end{definition}

\begin{remark}
Sometimes in the literature, an orthogonal $G$-spectrum is defined as a collection of based spaces $\{X_n\}_{n\in \mathbb{N}}$ with an action of $G\times O(n)$ on each $X_n$. The structure maps are $G$-equivariant maps 
\begin{equation*}
    X_n\wedge S^1\rightarrow X_{n+1},
\end{equation*}
such that the iterated structure maps 
\begin{equation*}
    X_n\wedge S^m \rightarrow X_{n+m}
\end{equation*}
are $O(m)\times O(n)$-equivariant, where $G$ acts trivially on $S^m$. 

This is what we refer to as a naive orthogonal $G$-spectrum. Genuine $G$-spectra are indexed on a complete $G$-universe of all $G$-representations, whereas naive $G$-spectra are indexed on the trivial $G$-universe. That is, naive $G$-spectra are just spectra with an action of $G$ on each level, and $G$-equivariant structure maps. These two descriptions of $G$-spectra are categorically equivalent, however they are not homotopically equivalent (see \cite[Section 9.3]{HHR21}), in that the most natural model structures associated to these categories are not Quillen equivalent. In what follows, we consider orthogonal spectra that are genuine with respect to $\C$ and naive with respect to $O(p,q)$. 
\end{remark}

\begin{definition}\index{$\C Sp^O [O(p,q)]$}
The category of \emph{$O(p,q)$-objects in orthogonal $\C$-spectra}, $\C Sp^O [O(p,q)]$, is the category of $O(p,q)$-objects in $\C \mathcal{E}_{1,0}$ and $O(p,q)$-equivariant maps. An $O(p,q)$-object in $\C \mathcal{E}_{1,0}$ is a $\C\Top_*$-enriched functor from $\C\mathcal{J}_{1,0}$ to $\COTop_*$. 
\end{definition}

\begin{theorem}\label{thm: stable ms on C2O-spectra}
The category of genuine orthogonal $\C$-spectra with an action of $O(p,q)$, $\C Sp^O [O(p,q)]$, has a cofibrantly generated proper stable model structure where fibrations and weak equivalences are defined by the underlying model structure on $\C Sp^O$ above.
\end{theorem}

One can prove the existence of this model structure in a similar way as for Remark \ref{remark: model structure on semi direct}, using the adjunction
% https://q.uiver.app/#q=WzAsMixbMCwwLCIoTyhwK3EpXFxydGltZXMgXFxDKVxcd2VkZ2Vfe1xcQ30oLSk6XFxDIFNwXk8iXSxbMSwwLCJcXEMgU3BeT1tPKHArcSldOiBpXioiXSxbMSwwLCIiLDAseyJvZmZzZXQiOi0yfV0sWzAsMSwiIiwwLHsib2Zmc2V0IjotMn1dXQ==
\[\begin{tikzcd}
	{(O(p,q)\rtimes \C)\wedge_{\C}(-):\C Sp^O} & {\C Sp^O[O(p,q)]: i^*}
	\arrow[shift left=2, from=1-1, to=1-2]
	\arrow[shift left=2, from=1-2, to=1-1]
\end{tikzcd}\]
where $i^*$ is the restriction functor.

We want to define a functor $\C Sp^O [O(p,q)]\rightarrow \OEpq$. This functor will be called $\alpha_{p,q}^*$, and it is analogous to the $\alpha_n^*$ functor built in \cite[Section 8]{BO13}.

\begin{definition}\index{$\alpha_{p,q}^*$}
Define the functor $\alpha_{p,q}:\Jpq\rightarrow\C\mathcal{J}_{1,0}$ by 
\begin{equation*}
    U\mapsto \mathbb{R}^{p,q}\otimes U = (p,q)U
\end{equation*}
on objects, and 
\begin{equation*}
    (f,x)\mapsto(\mathbb{R}^{p,q}\otimes f, x)
\end{equation*}
on morphisms. 
\end{definition}

The map on morphisms is clearly $\left(O(p,q)\rtimes \C\right)$-equivariant, and thus $\alpha_{p,q}$ is enriched over $\left(O(p,q)\rtimes \C\right)$-spaces. 

This induces a functor 
\begin{equation*}
    \alpha_{p,q}^*:\C Sp^O[O(p,q)]\rightarrow \OEpq
\end{equation*}
defined by precomposition. For $X\in\C Sp^O[O(p,q)]$ we define the $ (O(p,q)\rtimes \C)$- action on $(\alpha_{p,q}X)(V):=X((p,q)V)$ by 
\begin{equation*}
 X(g\sigma \otimes \sigma)\circ (g\sigma )_{X((p,q)V)}
\end{equation*}Here $ X(g\sigma \otimes \sigma)$ is the internal action on $X((p,q)V)$ induced by the action on $(p,q)V$, and $(g\sigma )_{X((p,q)V)}$ is the external action from $X((p,q)V)$ being an $\left( O(p,q)\rtimes \C\right)$-space. These two actions commute by construction. 

Checking that $\alpha_{p,q}^*X$ is well defined (i.e. that $\alpha_{p,q}^*X$ is $(O(p,q)\rtimes \C)\Top_*$-enriched) is the same as checking that the map 
\begin{equation*}
    \Jpq(U,V)\rightarrow \COTop_*(X((p,q)V),X((p,q)V))
\end{equation*}
is $\left(O(p,q)\rtimes \C\right)$-equivariant. 

To do this we consider the following commutative diagram. We use the notation $(-)^*$ to mean pre-composition and $(-)_*$ to mean post-composition. Let $s$ denote the map $((g\sigma )^{-1}\otimes \sigma)^*\circ (g\sigma  \otimes \sigma)_*$, and $t$ be the map $(X((g\sigma )^{-1}\otimes \sigma))^*\circ (X(g\sigma  \otimes \sigma))_*$. We have abbreviated $\COTop_*$ to $S\Top_*$ to save space ($S$ for semi-direct product). 
% https://q.uiver.app/#q=WzAsNixbMCwwLCJcXENcXG1hdGhjYWx7Sn1fe3AscX0oVSxWKSJdLFsxLDAsIlxcQ1xcbWF0aGNhbHtKfV97MSwwfSgocCxxKVUsKHAscSlWKSJdLFsyLDAsIihPKHArcSlcXHJ0aW1lc1xcQylcXFRvcF8qKFgoKHAscSlVKSxYKChwLHEpVikpIl0sWzAsMiwiXFxDXFxtYXRoY2Fse0p9X3twLHF9KFUsVikiXSxbMSwyLCJcXENcXG1hdGhjYWx7Sn1fezEsMH0oKHAscSlVLChwLHEpVikiXSxbMiwyLCIoTyhwK3EpXFxydGltZXNcXEMpXFxUb3BfKihYKChwLHEpVSksWCgocCxxKVYpKSJdLFsyLDUsInQiXSxbMSw0LCJzIl0sWzAsMywiZ1xcc2lnbWEiXSxbMCwxLCJcXGFscGhhX3twLHF9Il0sWzMsNCwiXFxhbHBoYV97cCxxfSIsMl0sWzEsMiwiWCJdLFs0LDUsIlgiLDJdXQ==
\[\begin{tikzcd}
	{\C\mathcal{J}_{p,q}(U,V)} & {\C\mathcal{J}_{1,0}((p,q)U,(p,q)V)} & {S\Top_*(X((p,q)U),X((p,q)V))} \\
	\\
	{\C\mathcal{J}_{p,q}(U,V)} & {\C\mathcal{J}_{1,0}((p,q)U,(p,q)V)} & {S\Top_*(X((p,q)U),X((p,q)V))}
	\arrow["t", from=1-3, to=3-3]
	\arrow["s", from=1-2, to=3-2]
	\arrow["g\sigma", from=1-1, to=3-1]
	\arrow["{\alpha_{p,q}}", from=1-1, to=1-2]
	\arrow["{\alpha_{p,q}}"', from=3-1, to=3-2]
	\arrow["X", from=1-2, to=1-3]
	\arrow["X"', from=3-2, to=3-3]
\end{tikzcd}\]
Given a pair $(f,x)\in \Jmor{p,q}{U}{V}$, by applying $\alpha_{p,q}^* X$ we get a a $\left( O(p,q)\rtimes \C\right)$-equivariant map
\begin{equation*}
    X(\mathbb{R}^{p,q}\otimes f): X((p,q)V)\rightarrow X((p,q)V).
\end{equation*}Therefore, the following two expressions are equal 
\begin{equation*}
    X(g\sigma \otimes V)\circ X(\mathbb{R}^{p,q}\otimes f,x)\circ X((g\sigma )^{-1}\otimes U)
\end{equation*}
\begin{equation*}
   (g\sigma )_{X(U)}\circ X(g\sigma \otimes V)\circ X(\mathbb{R}^{p,q}\otimes f,x)\circ X((g\sigma )^{-1}\otimes U)\circ (g\sigma )^{-1}_{X(V)}
\end{equation*}which by comparison to the commutative diagram tells us exactly that the map 
\begin{equation*}
    \Jpq(U,V)\rightarrow \COTop_*(X((p,q)V),X((p,q)V))
\end{equation*}
is $\left(O(p,q)\rtimes \C\right)$-equivariant, and hence that $\alpha_{p,q}^*X$ is well defined. 

\begin{remark}
Note that any other choice of internal action on $X((p,q)V)$ results in the failure of the diagram being commutative. For example, taking $X(g\sigma \otimes \id)$ as the internal action means that the left square in the diagram commutes only if $f$, from the pair $(f,x)$, is $\C$-equivariant, which is not necessarily the case.
\end{remark}

The left Kan extension of $X\in \C Sp^O[O(p,q)]$ along $\alpha_{p,q}$ can be described by the following $\COTop_*$-enriched coend. If we use the notation $(\alpha_{p,q})_!$ to denote taking the left Kan extension along $\alpha_{p,q}$, then 
\begin{equation*}
    ((\alpha_{p,q})_! (X))(V)=\int\limits^{U\in \Jpq} \CJ{1}{0} ((p,q)U,V)\wedge X(U).
\end{equation*}
We make this functor suitable enriched by `twisting' the action as in \cite[Definition 8.2]{BO13}. That is, we let $\C\mathcal{J}_{1,0}$ act on $\CJ{1}{0}((p,q)U,V)$ on the left by composition, and let $\C\mathcal{J}_{p,q}$ act on $\CJ{1}{0}((p,q)U,V)$ on the right by composition. It is not hard to show, by an argument of coends, that $(\alpha_{p,q})_!$ forms a left adjoint to $\alpha_{p,q}^*$. We now prove that this adjunction is indeed a Quillen equivalence. 

\begin{theorem}\label{QEstabletospectra}
The adjoint pair 
\begin{equation*}
    (\alpha_{p,q})_{!}: \OEpq^s\rightleftarrows \C Sp^O [O(p,q)] : \alpha_{p,q}^*
\end{equation*}
is a Quillen equivalence, where both categories are equipped with their stable model structures (see Proposition \ref{prop: stable ms} and Theorem \ref{thm: stable ms on C2O-spectra}). 
\end{theorem}

\begin{proof}
The proof follows as in \cite[Proposition 8.3]{BO13}.
Since $\alpha_{p,q}^*$ is defined by precomposition, it preserves objectwise fibrations and objectwise acyclic fibrations. It can easily be shown that $\alpha_{p,q}^*$ also preserved homotopy pullbacks. Hence, it preserves stable fibrations. An argument using finality shows that $\alpha_{p,q}^*$ preserves and reflects weak equivalences, and therefore preserves stable acyclic fibrations, making the adjunction Quillen.

It remains to show that the Quillen adjunction is a Quillen equivalence. By Hovey \cite[Theorem 1.3.16]{Hov99}, it suffices to show that the derived unit is a weak equivalence. Since the categories are stable and $\alpha_{p,q}^*$ preserves coproducts, it suffices to do this for the generators of $\OEpq$. This can be done by plugging in the generators $O(p,q)_+\wedge \pqs$ and $(\C\times O(p,q))_+\wedge \pqs$ into the formula for the unit, where $(p,q)\mathbb{S}$ sends $V\in\Jzero$ to the one point compactification of $\R^{p+q\delta}\otimes V$, denoted $S^{(p,q)V}$. 
\end{proof}

\section{Induction as a Quillen functor}

In this section we construct a Quillen adjunction between the $(p,q)$-homogeneous model structure and the stable model structure on $\Epq$. The right adjoint of this adjunction will be the differentiation functor $\ind_{0,0}^{p,q}\varepsilon^*$. Moreover, we will show that this Quillen adjunction is in fact a Quillen equivalence between these categories. Combined with the Quillen equivalence of Theorem \ref{QEstabletospectra}, this allows for the classification of $(p,q)$-homogeneous functors in terms of orthogonal $\C$-spectra with an action of $O(p,q)$. 

The steps taken to construct this Quillen equivalence are similar to \cite[Section 9 and Section 10]{BO13}. We begin by explicitly proving the existence of a Quillen adjunction between the projective model structures on $\OEpq$ and $\Ezero$, which can then be extended via properties of Quillen adjunctions and Bousfield localisations. 

\begin{lemma}\label{Epqlevel Ezero adjunction}
    There exists a Quillen adjunction 
\begin{equation*} \res_{0,0}^{p,q}/O(p,q):\OEpql\rightleftarrows\Ezero:\ind_{0,0}^{p,q}\varepsilon^*
\end{equation*}
\end{lemma}

\begin{proof}
The generating (acyclic) cofibrations of $\OEpq^l$ are of the form 
\begin{equation*}
\Jpq(U,-)\wedge O(p,q) \wedge i,    
\end{equation*}
where $i$ is a generating (acyclic) cofibration of the fine model structure on $\CTop_*$, see Theorem \ref{finemodelstructure}. 

Applying the left adjoint gives $\res_{0,0}^{p,q}\Jpq(U,-)\wedge i$ in $\Ezero$. We know that the functor $\res_{0,0}^{p,q}\Jpq(U,-)$ is cofibrant in $\Ezero$ by Lemma \ref{sphereandmorpharecofibrant}, hence $\res_{0,0}^{p,q}\Jpq(U,-)\wedge i$ is a (acyclic) cofibration in $\Ezero$. Therefore, the left adjoint preserves (acyclic) cofibrations, which by Hovey \cite[Lemma 2.1.20]{Hov99} shows that the adjunction is Quillen. 
\end{proof}

This adjunction can be extended to the $(p,q)$-polynomial model structure, via a composition of Quillen adjunctions. Furthermore, since the stable model structure on $\Epq$ is a left Bousfield localisation of the projective model structure, the Theorems of Hirschhorn \cite[Theorem 3.1.6, Proposition 3.1.18]{Hir03} can be used to additionally extend the adjunction to the stable model structure.  In particular, this Quillen adjunction implies that the derivative of a $(p,q)$-polynomial functor is a $(p,q)\Omega$-spectrum. 

\begin{lemma}
There exists a Quillen adjunction 
\begin{equation*} \res_{0,0}^{p,q}/O(p,q):\OEpq^s\rightleftarrows (p,q)\poly\Ezero:\ind_{0,0}^{p,q}\varepsilon^*
\end{equation*}
\end{lemma}

\begin{proof}
We know by Remark \ref{ezero and poly adjunction} that there exists a Quillen adjunction 
\begin{equation*}
    \Id:\Ezero\rightleftarrows (p,q)\poly\Ezero:\Id
\end{equation*}
Since the composition of Quillen adjunctions is a Quillen adjunction, combining the above adjunction with Lemma \ref{Epqlevel Ezero adjunction} gives a Quillen adjunction 
\begin{equation*}
    \res_{0,0}^{p,q}/O(p,q):\OEpq^l\rightleftarrows (p,q)\poly\Ezero:\ind_{0,0}^{p,q}\varepsilon^*
\end{equation*}
We now use \cite[Theorem 3.1.6, Proposition 3.1.18]{Hir03} to show that this Quillen equivalence passes to $\OEpq^s$. That is, we show that $\ind_{0,0}^{p,q}\varepsilon^*$ takes objects that are $(p,q)$-polynomial to $(p,q)\Omega$-spectra. This has been done in Proposition \ref{5.12} and Lemma \ref{omega spectra lemma}.
\end{proof}

Since the $(p,q)$-homogeneous model structure is a right Bousfield localisation of the $(p,q)$-polynomial model structure, the Theorems \cite[Theorem 3.1.6, Proposition 3.1.18]{Hir03} can again be used to extend this Quillen adjunction to the $(p,q)$-homogeneous model structure. 

\begin{lemma}\label{QAstabletohomog}
There exists a Quillen adjunction 
\begin{equation*} \res_{0,0}^{p,q}/O(p,q):\OEpq^s\rightleftarrows (p,q)\homog\Ezero:\ind_{0,0}^{p,q}\varepsilon^*
\end{equation*}
\end{lemma}
\begin{proof}
Let $f:X\rightarrow Y$ be a weak equivalence between fibrant objects in the $(p,q)$-homogeneous model structure. Then the map 
\begin{equation*}
    f^*:\Nat_{0,0}(\Jpq(V,-),X)\rightarrow \Nat_{0,0}(\Jpq(V,-),Y)
\end{equation*}
is a weak equivalence of $\C$-spaces for all $V\in\Jzero$, by definition of the right Bousfield localisation. Therefore, $\ind_{0,0}^{p,q}\varepsilon^*f$ is an objectwise weak equivalence. An application of \cite[Lemma 3.1.6, Proposition 3.1.18]{Hir03} now gives that $\ind_{0,0}^{p,q}\varepsilon^*$ is right Quillen. 
\end{proof}

The Quillen adjunctions between the model categories constructed are summarised in the following diagram (which we do not claim commutes). 

% https://q.uiver.app/?q=WzAsNSxbMCwwLCJPKHAscSlDXzJcXG1hdGhjYWx7RX1fe3AscX1ebCJdLFswLDIsIk8ocCxxKUNfMlxcbWF0aGNhbHtFfV97cCxxfV5zIl0sWzIsMCwiQ18yXFxtYXRoY2Fse0V9X3swLDB9Il0sWzQsMCwiKHAscSlcXHRleHR7LXBvbHktfUNfMlxcbWF0aGNhbHtFfV97MCwwfSJdLFs0LDIsIihwLHEpXFx0ZXh0ey1ob21vZy19Q18yXFxtYXRoY2Fse0V9X3swLDB9Il0sWzEsMCwiXFx0ZXh0e2lkfSIsMix7Im9mZnNldCI6NH1dLFsyLDMsIlxcdGV4dHtpZH0iLDAseyJvZmZzZXQiOi0yfV0sWzMsMiwiXFx0ZXh0e2lkfSIsMCx7Im9mZnNldCI6LTJ9XSxbNCwzLCJcXHRleHR7aWR9IiwwLHsib2Zmc2V0IjotNH1dLFszLDQsIlxcdGV4dHtpZH0iLDAseyJvZmZzZXQiOi00fV0sWzEsNCwiXFx0ZXh0e3Jlc31fezAsMH1ee3AscX0vTyhwLHEpIiwwLHsib2Zmc2V0IjotMn1dLFs0LDEsIlxcdGV4dHtpbmR9X3swLDB9XntwLHF9XFxlcHNpbG9uXioiLDAseyJvZmZzZXQiOi0yfV0sWzAsMiwiXFx0ZXh0e3Jlc31fezAsMH1ee3AscX0vTyhwLHEpIiwwLHsib2Zmc2V0IjotMn1dLFsyLDAsIlxcdGV4dHtpbmR9X3swLDB9XntwLHF9XFxlcHNpbG9uXioiLDAseyJvZmZzZXQiOi0yfV0sWzAsMSwiXFx0ZXh0e2lkfSIsMix7Im9mZnNldCI6NH1dXQ==
\[\begin{tikzcd}
	{O(p,q)C_2\mathcal{E}_{p,q}^l} && {C_2\mathcal{E}_{0,0}} && {(p,q)\poly C_2\mathcal{E}_{0,0}} \\
	\\
	{O(p,q)C_2\mathcal{E}_{p,q}^s} &&&& {(p,q)\homog C_2\mathcal{E}_{0,0}}
	\arrow["{\Id}"', shift right=4, from=3-1, to=1-1]
	\arrow["{\Id}", shift left=2, from=1-3, to=1-5]
	\arrow["{\Id}", shift left=2, from=1-5, to=1-3]
	\arrow["{\Id}", shift left=4, from=3-5, to=1-5]
	\arrow["{\Id}", shift left=4, from=1-5, to=3-5]
	\arrow["{\res_{0,0}^{p,q}/O(p,q)}", shift left=2, from=3-1, to=3-5]
	\arrow["{\ind_{0,0}^{p,q}\varepsilon^*}", shift left=2, from=3-5, to=3-1]
	\arrow["{\res_{0,0}^{p,q}/O(p,q)}", shift left=2, from=1-1, to=1-3]
	\arrow["{\ind_{0,0}^{p,q}\varepsilon^*}", shift left=2, from=1-3, to=1-1]
	\arrow["{\Id}"', shift right=4, from=1-1, to=3-1]
\end{tikzcd}\]

\section{The classification of $(p,q)$-homogeneous functors}\label{sec: classification}

We now aim to show that the Quillen adjunction of Theorem \ref{QAstabletohomog}, between the $(p,q)$-stable model structure and the $(p,q)$-homogeneous model structure, is in fact a Quillen equivalence. To do this, we take the same approach as Barnes and Oman \cite[Section 10]{BO13}, but first we will need to generalise a few more results from the underlying calculus to the $\C$-equivariant setting. The following is a $\C$-generalisation of \cite[Lemma 9.3]{BO13}.

\begin{lemma}\label{BO9.3}
The left derived functor of $\res_{0,0}^{p,q}/O(p,q)$ is objectwise weakly equivalent to $EO(p,q)_+\wedge_{O(p,q)} \res_{0,0}^{p,q}(-)$.    
\end{lemma}

\begin{proof}
 Let $X\in\OEpq^s$ and denote the cofibrant replacement of $X$ in the projective model structure $\OEpq^l$ by $\hat{c}X$. Then $\hat{c}X$ is in particular $O(p,q)$-free. Hence, there are objectwise weak equivalences
\begin{align*}
EO(p,q)_+\wedge_{O(p,q)}\res_{0,0}^{p,q}(\hat{c}X)&\rightarrow EO(p,q)_+\wedge_{O(p,q)}\res_{0,0}^{p,q}(X)\\
EO(p,q)_+\wedge_{O(p,q)}\res_{0,0}^{p,q}(\hat{c}X)&\rightarrow\res_{0,0}^{p,q}(\hat{c}X)/O(p,q)
\end{align*}
induced by the maps $\hat{c}X\rightarrow X$ and $EO(p,q)_+\rightarrow S^0$ respectively. The result follows directly from this. 
\end{proof}

The following two examples play a key role in classifying homogeneous functors. These examples generalise \cite[Example 5.7 and Example 6.4]{Wei95} respectively. Example \ref{5.7} proves that the functor $F(V)=\Omega^\infty [(S^{(p,q)V}\wedge \Theta)_{hO(p,q)}]$ is $(p,q)$-homogeneous. Alternatively, one can take the perspective that $(p,q)$-homogeneous functors are defined using the properties of this functor $F$. If one were to chose a different functor $F$, still in terms of spectra, it may be possible to classify a different class of input functors. However, we choose to work with the functor $F$ defined above,  since it is analogous to the functor $F(V)=\Omega^\infty [(S^{nV}\wedge \Theta)_{hO(n)}]$ used in the underlying calculus, see \cite[Example 5.7]{Wei98}.

The subscript $(-)_{hO(p,q)}$\index{$(-)_{hO(p,q)}$} denotes taking homotopy orbits. For $X\in \C Sp^O[O(p,q)]$, $X_{hO(p,q)}$ is the genuine orthogonal $\C$-spectrum defined by $$X_{hO(p,q)}:=EO(p,q)_+\wedge_{O(p,q)}X,$$ where $EO(p,q)$ is the universal space of $O(p,q)$. This universal space has a $\C$-action, which is given in Section \ref{sec: BO}. Hence, the homotopy orbits $X_{hO(p,q)}$ has a diagonal $\C$-action. The notation $\Omega^\infty$\index{$\Omega^\infty$} denotes taking the infinite loop space. For a genuine orthogonal $\C$-spectrum $X\in \C Sp^O$, $\Omega^\infty X$ is the $\C$-space defined by $$\Omega^\infty X := \hocolim_V \Omega^V X(V),$$ where the homotopy colimit is taken over $V\in\Jzero$. This homotopy colimit has a natural $\C$-action induced by the actions on each $X(V)$.

%Weiss Example 5.7 
\begin{example}\label{5.7}
Let $\Theta$ be an orthogonal $\C$-spectrum with $O(p,q)$-action and $p,q\geq 1$. The functor $F\in\Ezero$ defined by 
\begin{equation*}
    F:V\mapsto \Omega^\infty [(S^{(p,q)V}\wedge \Theta)_{hO(p,q)}]
\end{equation*}
is $(p,q)$-homogeneous. 
\end{example}

\begin{proof}
Since $F$ has a delooping, by Corollary \ref{cor: F delooping}, in order to show that $F$ is $(p,q)$-polynomial it suffices to show that $F^{(p+1,q)}$ and $F^{(p,q+1)}$ are both objectwise contractible (where $F^{(m,n)}$ denotes the $(m,n)$-derivative of $F$, $\ind_{0,0}^{m,n}F$).

Recall (see Proposition \ref{loops fibre sequence}) that $F^{(p+1,q)}(V)$ is the homotopy fibre of 
\begin{equation*}
    F^{(p,q)}(V)\rightarrow \Omega^{p,q\R}F^{(p,q)}(V\oplus \R)
\end{equation*}
and that $F^{(p,q+1)}(V)$ is the homotopy fibre of 
\begin{equation*}
    F^{(p,q)}(V)\rightarrow \Omega^{p,q\Rdelta}F^{(p,q)}(V\oplus \Rdelta).
\end{equation*}
Iterating this process gives a lattice of derivatives 

\[\begin{tikzcd}
	{F^{(p,q)}} & {F^{(p-1,q)}} & \dots & {F^{(0,q)}} \\
	{F^{(p,q-1)}} & {F^{(p-1,q-1)}} \\
	\vdots && \ddots \\
	{F^{(p,0)}} &&& F
	\arrow[from=1-1, to=1-2]
	\arrow[from=1-1, to=2-1]
	\arrow[from=1-2, to=1-3]
	\arrow[from=1-3, to=1-4]
	\arrow[from=2-1, to=3-1]
	\arrow[from=3-1, to=4-1]
	\arrow[from=1-2, to=2-2]
	\arrow[from=2-1, to=2-2]
	\arrow[dotted, from=4-1, to=4-4]
	\arrow[dotted, from=1-4, to=4-4]
\end{tikzcd}\]

We will ``identify" this lattice with another lattice
% https://q.uiver.app/?q=WzAsMTAsWzAsMCwiRltwLHFdIl0sWzEsMCwiRltwLTEscV0iXSxbMiwwLCJcXGRvdHMiXSxbMywwLCJGWzAscV0iXSxbMCwxLCJGW3AscS0xXSJdLFswLDIsIlxcdmRvdHMiXSxbMCwzLCJGW3AsMF0iXSxbMywzLCJGIl0sWzEsMSwiRltwLTEscS0xXSJdLFsyLDIsIlxcZGRvdHMiXSxbMCwxXSxbMCw0XSxbMSwyXSxbMiwzXSxbNCw1XSxbNSw2XSxbMSw4XSxbNCw4XSxbNiw3LCIiLDIseyJzdHlsZSI6eyJib2R5Ijp7Im5hbWUiOiJkb3R0ZWQifX19XSxbMyw3LCIiLDAseyJzdHlsZSI6eyJib2R5Ijp7Im5hbWUiOiJkb3R0ZWQifX19XV0=
\[\begin{tikzcd}
	{F[p,q]} & {F[p-1,q]} & \dots & {F[0,q]} \\
	{F[p,q-1]} & {F[p-1,q-1]} \\
	\vdots && \ddots \\
	{F[p,0]} &&& F
	\arrow[from=1-1, to=1-2]
	\arrow[from=1-1, to=2-1]
	\arrow[from=1-2, to=1-3]
	\arrow[from=1-3, to=1-4]
	\arrow[from=2-1, to=3-1]
	\arrow[from=3-1, to=4-1]
	\arrow[from=1-2, to=2-2]
	\arrow[from=2-1, to=2-2]
	\arrow[dotted, from=4-1, to=4-4]
	\arrow[dotted, from=1-4, to=4-4]
\end{tikzcd}\]
\noindent where $F[i,j](V)=\Omega^\infty[(S^{(p,q)V}\wedge \Theta)_{hO(p-i,q-j)}]$. That is, we will verify that $F[i,j]^{(1,0)}$ is objectwise equivalent to $F[i+1,j]$ and $F[i,j]^{(0,1)}$ is objectwise equivalent to $F[i,j+1]$, as is true for the functors $F^{(i,j)}$.

Here $O(p-i,q-j)$ is the subgroup of $O(p,q)$ that fixes the first $i$ coordinates and the $(p+1)^\text{st}$ to $(p+j)^\text{th}$ coordinates. That is, for all $g$ in $O(p-i,q-j)$ and all $(x_1,...,x_{p+q})$ in $\R^{p+q\delta}$, if $g (x_1,...,x_{p+q})= (y_1,...,y_{p+q})$, then $x_n=y_n$ for all $n\leq i$ and all $p+1\leq n\leq p+j$. 

Each $F[i,j]$ is an element of $\C \mathcal{E}_{i,j}$. The structure maps are defined by the following series of maps
\begin{align*}
    S^{(i,j)U}\wedge F[i,j](V)&=S^{(i,j)U}\wedge \Omega^\infty[(S^{(p,q)V}\wedge \Theta)_{hO(p-i,q-j)}]\\
    &\rightarrow \Omega^\infty[S^{(i,j)U}\wedge(S^{(p,q)V}\wedge \Theta)_{hO(p-i,q-j)}]\\
    &=\Omega^\infty[(S^{(i,j)U}\wedge S^{(p,q)V}\wedge \Theta)_{hO(p-i,q-j)}]\\
    &\rightarrow \Omega^\infty[(S^{(p,q)U}\wedge S^{(p,q)V}\wedge \Theta)_{hO(p-i,q-j)}]\\
    &\simeq \Omega^\infty[(S^{(p,q)(U\oplus V)}\wedge \Theta)_{hO(p-i,q-j)}]\\
    &=F[i,j](U\oplus V)
\end{align*}
where the second equality holds since $O(p-i,q-j)$ fixes $\R^{i,j}$. Moreover, $F[p,q]$ is a $(p,q)\Omega$-spectrum (by substituting $F[p,q]$ into the adjoint structure map). 

We now want to show that $F[i,j]^{(1,0)}$ is objectwise equivalent to $F[i+1,j]$ and $F[i,j]^{(0,1)}$ is objectwise equivalent to $F[i,j+1]$. Then, since $F[p,q]$ is a $(p,q)\Omega$-spectrum
\begin{align*}
*\simeq F[p,q]^{(1,0)}\equiv (F^{(p,q)})^{(1,0)}=F^{(p+1,q)}\\
*\simeq F[p,q]^{(0,1)}\equiv (F^{(p,q)})^{(0,1)}=F^{(p,q+1)}
    \end{align*}
as desired. We do this by calculating $F[i,j]^{(1,0)}$ and $F[i,j]^{(0,1)}$ using Proposition \ref{loops fibre sequence}. 

First let $0\leq i<p$ and $0\leq j \leq q$ 
\begin{align*}
&F[i,j]^{(1,0)}(V)= \hofibre\left[ F[i,j](V) \rightarrow \Omega^{(i,j)\R} F[i,j](\R\oplus V) \right]\\
=& \hofibre\left[ \Omega^\infty[(S^{(p,q)V}\wedge \Theta)_{hO(p-i,q-j)}]\rightarrow \Omega^{(i,j)\R} \Omega^\infty[(S^{(p,q)(\R\oplus V)}\wedge \Theta)_{hO(p-i,q-j)}]\right] \\
=&\Omega^\infty  \hofibre\left[ (S^{(p,q)V}\wedge \Theta)_{hO(p-i,q-j)}\rightarrow \Omega^{(i,j)\R} [(S^{(p,q)(\R\oplus V)}\wedge \Theta)_{hO(p-i,q-j)}]\right] \\
\simeq & \Omega^\infty  \hofibre\left[ (S^{(p,q)V}\wedge \Theta)_{hO(p-i,q-j)}\rightarrow \Omega^{(i,j)\R} [(S^{(p,q)\R}\wedge S^{(p,q)V}\wedge \Theta)_{hO(p-i,q-j)}]\right] \\
\simeq & \Omega^\infty  \hofibre\left[ (S^{(p,q)V}\wedge \Theta)_{hO(p-i,q-j)}\rightarrow \Omega^{(i,j)\R} [(S^{(i,j)\R}\wedge S^{(p-i,q-j)\R}\wedge S^{(p,q)V}\wedge \Theta)_{hO(p-i,q-j)}]\right] \\
\simeq & \Omega^\infty  \hofibre\left[ (S^{(p,q)V}\wedge \Theta)_{hO(p-i,q-j)}\rightarrow \Omega^{(i,j)\R}\Sigma^{(i,j)\R} [(S^{(p-i,q-j)\R}\wedge S^{(p,q)V}\wedge \Theta)_{hO(p-i,q-j)}]\right] \\
\simeq &\Omega^\infty  \hofibre\left[ (S^{(p,q)V}\wedge \Theta)_{hO(p-i,q-j)}\rightarrow (S^{(p-i,q-j)\R}\wedge S^{(p,q)V}\wedge \Theta)_{hO(p-i,q-j)}\right] 
\end{align*}
where the last weak equivalence is by the $\pi_*$-equivalence $\Omega^V\Sigma^VX\rightarrow X$ for $\C Sp^O[O(p,q)]$. 

The map $S^{(p,q)V}\wedge \Theta \rightarrow S^{(p-i,q-j)\R}\wedge S^{(p,q)V}\wedge \Theta$ is given by 
\begin{equation*}
    S^{0,0}\wedge S^{(p,q)V}\wedge \Theta \xrightarrow[]{\mu\wedge \id\wedge\id} S^{(p-i,q-j)\R}\wedge S^{(p,q)V}\wedge \Theta,
\end{equation*}
where $\mu:S^{0,0}\rightarrow S^{(p-i,q-j)\R}$ is the canonical inclusion, which has stable homotopy fibre $S^{(p-i-1,q-j)\R}_+$. Therefore, the homotopy fibre of $\mu\wedge \id\wedge\id$ is $S^{(p-i-1,q-j)\R}_+\wedge S^{(p,q)V}\wedge \Theta$. Here $O(p-i,q-j)$ acts on $S^{(p-i-1,q-j)\R}_+$ by identifying $S^{(p-i-1,q-j)\R}$ with the unit sphere $S(\R^{p-i,q-j})$ in $\R^{p-i,q-j}$. Since taking homotopy orbits preserves fibre sequences, the homotopy fibre of the map $(S^{(p,q)V}\wedge \Theta)_{hO(p-i,q-j)}\rightarrow (S^{(p-i,q-j)\R}\wedge S^{(p,q)V}\wedge \Theta)_{hO(p-i,q-j)}$ is $(S^{(p-i-1,q-j)\R}_+\wedge S^{(p,q)V}\wedge \Theta)_{hO(p-i,q-j)}$.

Then we conclude as follows, where the second weak equivalence is described below. 
\begin{align*}
F[i,j]^{(1,0)}(V)&\simeq \Omega^\infty \left[ (S^{(p-i-1,q-j)\R}_+\wedge S^{(p,q)V}\wedge \Theta)_{hO(p-i,q-j)}\right]\\
&=\Omega^\infty \left[ (S(\R^{p-i,q-j})_+\wedge S^{(p,q)V}\wedge \Theta)_{hO(p-i,q-j)}\right]\\
&\simeq \Omega^\infty \left[ (S^{(p,q)V}\wedge \Theta)_{hO(p-i-1,q-j)}\right]\\
&=F[i+1,j]
\end{align*}

The second weak equivalence holds by Proposition \ref{sphere as quotient of orthogonal groups prop} and that for $X$ a $\C$-spectrum with $O(m,n)$-action, $m> 0$ and $n\geq 0$,
\begin{align*}
(S(\R^{m,n})_+\wedge X)_{hO(m,n)}&=EO(m,n)_+\wedge_{O(m,n)} (S(\R^{m,n})_+\wedge X)\\
&=EO(m,n)_+\wedge_{O(m,n)} (O(m,n)/O(m-1,n)_+\wedge X)\\
&=EO(m,n)_+\wedge_{O(m-1,n)} X\\
&\simeq EO(m-1,n)_+\wedge_{O(m-1,n)} X\\
&=X_{hO(m-1,n)}
\end{align*}since the map $t^*:(O(m,n)\rtimes \C)\Top_*\rightarrow (O(m-1,n)\rtimes \C)\Top_*$, induced by the $\C$-equivariant subgroup inclusion map $t:O(m-1,n)\rightarrow O(m,n)$, exhibits $t^*EO(m,n)$ as a model for $EO(m-1,n)$. 

Now let $0\leq i\leq p$ and $0\leq j<q$. A similar calculation shows that 
\begin{align*}
F[i,j]^{(0,1)}(V)&\simeq \Omega^\infty \left[ (S^{(p-i,q-j-1)\Rdelta}_+\wedge S^{(p,q)V}\wedge \Theta)_{hO(p-i,q-j)}\right]\\
&\simeq \Omega^\infty \left[ (S^{(p,q)V}\wedge \Theta)_{hO(p-i,q-j-1)}\right]\\
&=F[i,j+1]
\end{align*}
As above, the second weak equivalence holds by Proposition \ref{sphere as quotient of orthogonal groups prop} and that 
\begin{align*}
(S(\R^{n,m})_+\wedge X)_{hO(m,n)}&=EO(m,n)_+\wedge_{O(m,n)} (S(\R^{n,m})_+\wedge X)\\
&=EO(m,n)_+\wedge_{O(m,n)} (O(n,m)/O(n-1,m)_+\wedge X)\\
&\cong EO(m,n)_+\wedge_{O(m,n)} (O(m,n)/O(m,n-1)_+\wedge X)\\
&=EO(m,n)_+\wedge_{O(m,n-1)} X\\
&\simeq EO(m,n-1)_+\wedge_{O(m,n-1)} X\\
&=X_{hO(m,n-1)}
\end{align*}
where the third step uses the $\C$-equivariant group isomorphism $O(m,n)\rightarrow O(n,m)$ defined for $m,n\geq 1$ by 
\begin{equation*}
A\rightarrow\begin{pmatrix}
    0& \Id_n\\\Id_m &0
\end{pmatrix}A \begin{pmatrix}
    0 & \Id_m\\ \Id_n & 0
\end{pmatrix}.
\end{equation*}

What remains to show is that $F$ is $(p,q)$-reduced ($\Tpq F(V) \simeq *$ for all $V\in \Jzero$). The notions of dimension functions for $G$-spaces discussed in Section \ref{sec: EFST} can also be used to describe the connectivity of a $G$-spectrum. Since the category of spectra $\C Sp^O[O(p,q)]$ is stable with respect to $\C$-representations, the spectrum $X=S^{(p,q)V}\wedge \Theta$ has equivariant connectivity given by the dimension function $c^*(X)=|(p,q)V^*|+c^*(\Theta)$ (see Examples \ref{ex: dim function examples} for similar notation). 

That is, $\pi_n^H X$ is trivial for $n\leq c^H(X)$, where
\begin{align*}
    c^e(X)&=(p+q)\Dim(V) + c^e(\Theta) \\
    c^{\C}(X)&=p\Dim(V^{\C}) +q\Dim((V^{\C})^{\perp}) +c^{\C}(\Theta),
\end{align*}
$c^H(\Theta)$ denotes the connectivity of $\Theta$ with respect to its equivariant homotopy groups $\pi_n^H \Theta$ and $(V^{\C})^\perp$ denotes the orthogonal complement of $V^{\C}$ (see Remark \ref{rem: ^h^perp notation}). 

Therefore, the map $$F(V)\rightarrow *(V)=*$$ is $(c^*(X)+1)$-connected, since homotopy orbits and $\Omega^\infty$ do not decrease connectivity. Repeated application of Lemma \ref{e.3} gives that this map is a $\Tpq$-equivalence. That is, 
\begin{equation*}
    \Tpq F(V)\simeq \Tpq * (V)=*\qedhere
\end{equation*}
\end{proof}

\begin{remark}
If one replaces the group $O(p,q)$ with the group $O(p)\times O(q)$ of $\C$-equivariant linear isometries on $\mathbb{R}^{p+q\delta}$, then a generalisation like Example \ref{5.7} cannot be achieved. This is because there is no equivariant description of the sphere $S(\R^{p,q})$ as a quotient of these groups like there is for the groups $O(p,q)$ (see Proposition \ref{sphere as quotient of orthogonal groups prop}).
\end{remark}

Before we give the $\C$-generalisation of \cite[Example 6.4]{Wei95}, we first prove an equivariant version of a key result used in the proof of the underlying example. The theorem describes how close the map $[\Omega^\infty X]_{hL} \rightarrow \Omega^\infty [X_{hL}]$ is to being an equivalence, where $X$ is a $G$-spectrum with an action of a compact Lie group $L$. Here homotopy orbits $(-)_{hL}$ and the infinite loop space $\Omega^\infty$ of $X\in G Sp^O[L]$ are defined in a similar way to $(-)_{hO(p,q)}$ and $\Omega^\infty $ of $X\in \C Sp^O[O(p,q)]$ in Example \ref{5.7}. For a $(L\rtimes G)$-space $X$, $X_{hL}$ is defined to be the space
\begin{equation*}
    X_{hL}:=EL\wedge_L X,
\end{equation*}
which has a $G$-action induced by the $G$-actions on $EL$ (the universal space of $L$) and $X$.

\begin{theorem}\label{loopsholimthm}
For a finite group $G$ that acts on $L$, let $X$ be a $G$-spectrum with an action of $L$, and $\R[G]$ be the regular representation of $G$. Then the canonical map 
\begin{equation*}
    \xi:[\Omega^\infty X]_{hL} \rightarrow \Omega^\infty [X_{hL}]
\end{equation*}
is $v$-connected for any dimension function $v$ satisfying 
\begin{itemize}
    \item $v(H)\leq 2c^H(X)+1\quad $
    \item $v(H)\leq c^K(X) \quad $
\end{itemize}
for all closed subgroups $H\leq G$ and subgroup pairs $K<H$, where $c^H(X):=\conn(X^H)$.
\end{theorem} 

\begin{proof}
To show that the map $\xi: [\Omega^\infty X]_{hL}\rightarrow \Omega^\infty [X_{hL}]$ is $v$-connected it will suffice to show that the maps $f,g$ in the following commutative diagram are $v$-connected, 
\[\begin{tikzcd}
	{[\Omega^\infty X]_{hL}} & {Q([\Omega^\infty X]_{hL})} \\
	& {\Omega^\infty[X_{hL}]}
	\arrow["g", from=1-1, to=1-2]
	\arrow["f", from=1-2, to=2-2]
	\arrow["\xi"', from=1-1, to=2-2]
\end{tikzcd}\]
where $QY:=\hocolim_n \Omega^{n\R[G]}\Sigma^{n\R[G]} Y$. 

$Q(\Omega^\infty X)$ is the homotopy colimit of the sequential diagram 
\begin{equation*}
    \Omega^\infty X \rightarrow \Omega^{\R[G]}\Sigma^{\R[G]} (\Omega^\infty X)\rightarrow \Omega^{2\R[G]}\Sigma^{2\R[G]} (\Omega^\infty X)\rightarrow \dots
\end{equation*}
where each map $\Omega^{n\R[G]}\Sigma^{n\R[G]} (\Omega^\infty X)\rightarrow \Omega^{(n+1)\R[G]}\Sigma^{(n+1)\R[G]} (\Omega^\infty X)$ is obtained by sending a map $S^{n\R[G]}\rightarrow S^{n\R[G]}\wedge \Omega^\infty X$ to its smash product with the identity map on $S^{\R[G]}$. 

By the Equivariant Freudenthal Suspension Theorem (see Section \ref{sec: EFST}), the first of these maps is $v$-connected for $v$ satisfying\begin{align*}
    v(H)&\leq 2c^H(X)+1 \\
    v(H)&\leq c^K(X) 
\end{align*}
for all closed subgroups $H\leq G$ and subgroup pairs $K<H$, where we have used that $c^H(\Omega^\infty X)=c^H(X)$. 
Moreover, each successive map in the homotopy colimit is at least as connected as the first (this can be seen by applying the equivariant Freudenthal suspension to the maps in the colimit). Since equivariant homotopy groups commute with sequential homotopy colimits, we conclude that the unit map $i:\Omega^\infty X\rightarrow Q(\Omega^\infty X)$ is $v$-connected. 

Since taking homotopy orbits does not decrease connectivity, the map $g$ is as connected as the map $i$. Therefore the map $g$ is $v$-connected as required.

There is a commutative diagram of equivariant homotopy groups ($*\geq 0$)
% https://q.uiver.app/?q=WzAsNSxbMCwwLCJcXHBpXypeSChcXE9tZWdhXlxcaW5mdHkgWCkiXSxbMSwwLCJcXHBpXypeSChRXFxPbWVnYV5cXGluZnR5IFgpIl0sWzEsMSwiXFxwaV8qXkgoXFxPbWVnYV5cXGluZnR5IFgpIl0sWzIsMCwiXFxwaV8qXkgoXFxTaWdtYV5cXGluZnR5IFxcT21lZ2FeXFxpbmZ0eSBYKSJdLFsyLDEsIlxccGlfKl5IIFgiXSxbMCwxLCJpIl0sWzEsMiwicCJdLFswLDIsIlxcdGV4dHtpZH0iLDJdLFs0LDIsIlxcY29uZyIsMl0sWzMsMSwiXFxjb25nIiwyXSxbMyw0LCJwJyJdXQ==
\[\begin{tikzcd}
	{\pi_*^H(\Omega^\infty X)} & {\pi_*^H(Q\Omega^\infty X)} & {\pi_*^H(\Sigma^\infty \Omega^\infty X)} \\
	& {\pi_*^H(\Omega^\infty X)} & {\pi_*^H X}
	\arrow["i", from=1-1, to=1-2]
	\arrow["p", from=1-2, to=2-2]
	\arrow["{\id}"', from=1-1, to=2-2]
	\arrow["\cong"', from=2-3, to=2-2]
	\arrow["\cong"', from=1-3, to=1-2]
	\arrow["{p'}", from=1-3, to=2-3]
\end{tikzcd}\]
In particular, the connectivity of the map $p'$ is the same as the connectivity of the map $i$. Hence, $p'$ is also $v$-connected. Since taking homotopy orbits does not decrease connectivity and $\Sigma^\infty$ commutes with homotopy orbits, the composition map 
\begin{equation*}
\Sigma^\infty ([\Omega^\infty X]_{hL})\cong [\Sigma^\infty \Omega^\infty X]_{hL}\overset{p'_{hL}}{\longrightarrow} X_{hL}   
\end{equation*}
is also $v$-connected. Call this composition $p''_{hL}$. The map $f$ is $\Omega^\infty$ of the map $p''_{hL}$ above. Therefore the map $f$ is $v$-connected as required, since $\Omega^\infty$ preserves connectivity.
\end{proof}

\begin{example}\label{6.4}
Let $\Theta$ be a $\C$-spectrum with an action of $O(p,q)$, where $p,q \geq 1$. Then the input functors 
\begin{align*}
    &E:V\mapsto [\Omega^\infty(S^{(p,q)V}\wedge \Theta)]_{hO(p,q)}    \\
    &F:V\mapsto \Omega^\infty [(S^{(p,q)V}\wedge \Theta)_{hO(p,q)}]
\end{align*}
are $T_{p+1,q}T_{p,q+1}$-equivalent under the canonical subfunctor inclusion map $r:E\rightarrow F$. 
\end{example}

\begin{proof}
The spectrum $X=S^{(p,q)V}\wedge \Theta$ has equivariant connectivity given by the dimension function $c^*(X)=|(p,q)V^*|+c^*(\Theta)$, see Section \ref{sec: EFST}. That is, $\pi_n^H X$ is trivial for all $n\leq c^H(X)$, where
\begin{align*}
    c^e(X)&=(p+q)\Dim(V) + c^e(\Theta) \\
    c^{\C}(X)&=p\Dim(V^{\C}) +q\Dim((V^{\C})^{\perp}) +c^{\C}(\Theta),
\end{align*}
$c^H(\Theta)$ denotes the connectivity of $\Theta$ with respect to its equivariant homotopy groups $\pi_n^H \Theta$ and $(V^{\C})^\perp$ denotes the orthogonal complement of $V^{\C}$ (see Remark \ref{rem: ^h^perp notation}). 

Applying Theorem \ref{loopsholimthm} yields that the map $r(V):E(V)\rightarrow F(V)$ is $v$-connected, where 
\begin{align*}
    v(e)&=2(p+q)\Dim(V) + 2c^e(\Theta) +1\\
    v(\C)&=\text{min}\{2p\Dim(V^{\C}) +2q\Dim((V^{\C})^{\perp}) +2c^{\C}(\Theta)+1, (p+q)\Dim(V)+c^e(\Theta)\}
\end{align*}

Corollary \ref{erratum corollary} implies that $\tau_{p+1,q}\tau_{p,q+1}r(V):\tau_{p+1,q}\tau_{p,q+1}E(V)\rightarrow \tau_{p+1,q}\tau_{p,q+1}F(V)$ is at least $(v+1)$-connected. Repeated application of Corollary \ref{erratum corollary} yields that the connectivity of $\tau_{p+1,q}^l\tau_{p,q+1}^lr(V)$ tends to infinity as $l$ tends to infinity. Hence $T_{p+1,q}T_{p,q+1}r$ is an objectwise weak equivalence. \qedhere

\end{proof}

We can now prove that the Quillen adjunction of Theorem \ref{QAstabletohomog} is a Quillen equivalence. The proof resembles that of Barnes and Oman in \cite[Theorem 10.1]{BO13} and Taggart in \cite[Theorem 7.5]{Tag22unit}, using the new $\C$-equivariant versions of \cite[Example 5.7 and Example 6.4]{Wei95} given in Examples \ref{5.7} and \ref{6.4} respectively.

\begin{theorem}\label{boclassification}
For all $p,q\geq 1$, the Quillen adjunction 
\begin{equation*} \res_{0,0}^{p,q}/O(p,q):\OEpq^s\rightleftarrows (p,q)\homog\Ezero:\ind_{0,0}^{p,q}\varepsilon^*
\end{equation*}
is a Quillen equivalence. 
\end{theorem}

\begin{proof}
Let $f:A\rightarrow B$ be a map of fibrant objects in $(p,q)$-homog-$\Ezero$ such that $\ind_{0,0}^{p,q} \varepsilon^* f$ is a weak equivalence in $\OEpq^s$. Then $\ind_{0,0}^{p,q}\varepsilon^* T_{p+1,q}T_{p,q+1}f$ is a weak equivalence of fibrant objects in $\OEpq^s$, and by Lemma \ref{BOLem7.10} it is also an objectwise weak equivalence. That is, $f$ is a weak equivalence in $(p,q)$-homog-$\Ezero$, so the right adjoint $\ind_{0,0}^{p,q}\varepsilon^*$ reflects weak equivalences of fibrant objects. 

By Hovey \cite[Theorem 1.3.16]{Hov99}, what remains to check is that the derived unit is a weak equivalence of $\OEpq^s$ on cofibrant objects. 

Let $X$ be cofibrant in $\OEpq^s$. There is a Quillen equivalence 
\begin{equation*}
(\alpha_{p,q})_{!}: \OEpq^s \rightleftarrows \C Sp^O [O(p,q)] : \alpha_{p,q}^*
\end{equation*}
by Theorem \ref{QEstabletospectra}. Therefore, an application of \cite[Theorem 1.3.16]{Hov99} says that there exists a $(p,q)\pi_*$-isomorphism 
\begin{equation*}
    X\rightarrow \alpha_{p,q}^* \hat{f} (\alpha_{p,q})_{!} X,
\end{equation*}
where $\hat{f}$ represents fibrant replacement in $\C Sp^O[O(p,q)]$. 

Let $\hat{c}$ represent cofibrant replacement in $\OEpq^s$, and denote the $\Omega$-spectrum with $O(p,q)$-action $\hat{f}(\alpha_{p,q})_!X$ by $\Psi$. There is a commutative diagram in $\OEpq^s$ 
% https://q.uiver.app/#q=WzAsNCxbMCwwLCJYIl0sWzAsMSwiXFx0ZXh0e2luZH1fezAsMH1ee3AscX1cXHZhcmVwc2lsb25eKlRfe3ArMSxxfVRfe3AscSsxfVxcdGV4dHtyZXN9X3swLDB9XntwLHF9WC9PKHArcSkiXSxbMSwxLCJcXHRleHR7aW5kfV97MCwwfV57cCxxfVxcdmFyZXBzaWxvbl4qVF97cCsxLHF9VF97cCxxKzF9XFx0ZXh0e3Jlc31fezAsMH1ee3AscX0oXFxoYXR7Y31cXGFscGhhX3twLHF9XipcXFBzaSkvTyhwK3EpIl0sWzEsMCwiXFxoYXR7Y31cXGFscGhhX3twLHF9XipcXFBzaSJdLFswLDFdLFswLDNdLFszLDIsImgiXSxbMSwyXV0=
\[\begin{tikzcd}
	X & {\hat{c}\alpha_{p,q}^*\Psi} \\
	{\ind_{0,0}^{p,q}\varepsilon^*T_{p+1,q}T_{p,q+1}\res_{0,0}^{p,q}X/O(p,q)} & {\ind_{0,0}^{p,q}\varepsilon^*T_{p+1,q}T_{p,q+1}\res_{0,0}^{p,q}(\hat{c}\alpha_{p,q}^*\Psi)/O(p,q)}
	\arrow["\eta", swap, from=1-1, to=2-1]
	\arrow[from=1-1, to=1-2]
	\arrow[from=1-2, to=2-2]
	\arrow[from=2-1, to=2-2]
\end{tikzcd}\]
The map $\eta$ is the derived unit, which we want to show is a $(p,q)\pi_*$-isomorphism. 

The top horizontal map is a $(p,q)\pi_*$-isomorphism, since it is the $(p,q)\pi_*$-isomorphism given above composed with cofibrant replacement. The bottom horizontal map is a $(p,q)\pi_*$-isomorphism, since the top one is and derived functors preserve weak equivalences. The same method as \cite[Theorem 10.1]{BO13} shows that the right vertical map is also a $(p,q)\pi_*$-isomorphism. This argument uses the $\C$-equivariant generalisations of \cite[Examples 5.7 and 6.4]{Wei95}, which are given by Example \ref{5.7} and Example \ref{6.4} respectively. Therefore, $\eta$ is also a $(p,q)\pi_*$-isomorphism as required. 
\end{proof}

\begin{corollary}\label{zigzagclassification}
There is an equivalence of homotopy categories 
\begin{equation*}
    \Ho(\C Sp^O[O(p,q)])\rightleftarrows \Ho((p,q)\homog\Ezero)
\end{equation*}
for $p,q\geq 1$. 
\end{corollary}

\begin{proof}
Let $p,q\geq 1$. The adjunctions 
\begin{equation*}
\res_{0,0}^{p,q}/O(p,q):\OEpq^s\rightleftarrows (p,q)\homog\Ezero:\ind_{0,0}^{p,q}\varepsilon^*
\end{equation*}
and 
\begin{equation*}
(\alpha_{p,q})_{!}: \OEpq^s \rightleftarrows \C Sp^O [O(p,q)] : \alpha_{p,q}^*
\end{equation*}
are Quillen equivalences. Composition of the left and right derived functors gives the desired zig-zag of equivalences. 
\end{proof}

Rephrasing this classification using the derived adjunctions, we can explicitly describe how $(p,q)$-homogeneous functors are completely determined by genuine orthogonal $\C$-spectra with an action of $O(p,q)$. The following classification is a $\C$-equivariant generalisation of \cite[Theorem 7.3]{Wei95}. We will denote the image $\mathbb{L}(\alpha_{p,q})_!\mathbb{R}\ind_{0,0}^{p,q}\varepsilon^*F\in \C Sp^O[O(p,q)]$ of a $F\in\Ezero$ under the derived zig-zag of Quillen equivalences by $\Theta_F^{p,q}$\index{$\Theta_F^{p,q}$}. That is, $\Theta_F^{p,q}$ is a specific $\C$-spectrum with an action of $O(p,q)$, which is determined by the functor $F$. The proof follows the method used by Taggart \cite[Theorem 8.1]{Tag22unit}.

\begin{theorem}\label{weissclassification}
Let $p,q\geq 1$. If $F\in\Ezero$ is a $(p,q)$-homogeneous functor, then $F$ is objectwise weakly equivalent to 
\begin{equation*}
    V\mapsto \Omega^\infty[(S^{(p,q)V}\wedge\Theta_F^{p,q})_{hO(p,q)}]
\end{equation*}
Conversely, every functor of the form 
\begin{equation*}
    V\mapsto \Omega^\infty[(S^{(p,q)V}\wedge\Theta)_{hO(p,q)}]
\end{equation*}
where $\Theta\in \C Sp^O[O(p,q)]$ is $(p,q)$-homogeneous. 
\end{theorem}
\begin{proof}
The proof of the converse statement is exactly Example \ref{5.7}. 

Let $F$ be cofibrant-fibrant in $(p,q)$-homog-$\Ezero$. That is, $F$ is $(p,q)$-homogeneous and cofibrant in the projective model structure. Define functors $E,G\in\Ezero$ by 
\begin{align*}
    E(V)&=(\ind_{0,0}^{p,q}\varepsilon^* F(V))_{hO(p,q)}\\
    G(V)&=\Omega^\infty[(S^{(p,q)V}\wedge \Theta_F^{p,q})_{hO(p,q)}]
\end{align*}
The functors $E$ and $G$ are $T_{p+1,q}T_{p,q+1}$-equivalent, since 
\begin{equation*}
    \ind_{0,0}^{p,q}\varepsilon^* F(V)=\alpha_{p,q}^*\Theta_F^{p,q}(V)\simeq\Omega^\infty(S^{(p,q)V}\wedge\Theta_F^{p,q})
\end{equation*}
and $[\Omega^\infty(S^{(p,q)V}\wedge\Theta_F^{p,q})]_{hO(p,q)}$ is $T_{p+1,q}T_{p,q+1}$-equivalent to $G$ by Example \ref{6.4}. 

Since $G$ is $(p,q)$-polynomial by Example \ref{5.7}, Lemma \ref{weiss6.3.2} implies that $G$ is objectwise weakly equivalent to $T_{p+1,q}T_{p,q+1}E$. Therefore, there is an objectwise weak equivalence between $\ind_{0,0}^{p,q}\varepsilon^*T_{p+1,q}T_{p,q+1}E$ and $\ind_{0,0}^{p,q}\varepsilon^* G$, since $\ind_{0,0}^{p,q}\varepsilon^*$ is right Quillen and preserves weak equivalences of fibrant objects. 

Using the ``identification" $\equiv$ from Example \ref{5.7}, we get the following.
\begin{equation*}
    \ind_{0,0}^{p,q}\varepsilon^*G(V)\equiv G[p,q](V):=\Omega^\infty(S^{(p,q)V}\wedge \Theta_F^{p,q})\simeq \ind_{0,0}^{p,q}\varepsilon^*F(V)
\end{equation*}
Therefore, there is a zig-zag of objectwise weak equivalences between $\ind_{0,0}^{p,q}\varepsilon^*T_{p+1,q}T_{p,q+1}E$ and $\ind_{0,0}^{p,q}\varepsilon^*F(V)$. 

Since $F$ is $(p,q)$-homogeneous by assumption, it is in particular $(p,q)$-polynomial. Then by an application of Lemma \ref{weiss6.3.2}, there is a zig-zag of objectwise weak equivalences 
\begin{equation*}
\ind_{0,0}^{p,q}\varepsilon^*T_{p+1,q}T_{p,q+1}E\simeq\ind_{0,0}^{p,q}\varepsilon^*T_{p+1,q}T_{p,q+1}F   
\end{equation*}
That is, $E$ and $F$ are weakly equivalent in the $(p,q)$-homogeneous model structure.

Since both $E$ and $F$ are cofibrant in $(p,q)$-homog-$\Ezero$, an application of \cite[Theorem 3.2.13 (2)]{Hir03} implies that $E$ and $F$ are weakly equivalent in the $(p,q)$-polynomial model structure. Since both $E$ and $F$ are fibrant in $(p,q)$-poly-$\Ezero$, an application of \cite[Theorem 3.2.13 (1)]{Hir03} implies that $E$ and $F$ are weakly equivalent in the projective model structure. Hence, there are objectwise weak equivalences
\begin{equation*}
    G\simeq T_{p+1,q}T_{p,q+1}E\simeq T_{p+1,q}T_{p,q+1}F\simeq F
\end{equation*}
since $T_{p+1,q}T_{p,q+1}$ preserves objectwise equivalences and $F$ is $(p,q)$-homogeneous. 

For general $(p,q)$-homogeneous $F$ the result follows by cofibrantly replacing $F$ in the projective model structure and then applying the argument above. 
\end{proof}

The final theorem is an application of the classification Theorem \ref{weissclassification}. This theorem describes how the classification is actually used, in order to study one of the input functors (see Definition \ref{jzero and ezero def}). In particular, the fibre $$D_{p,q}X(V) \rightarrow T_{p+1,q}T_{p,q+1}X(V)\rightarrow T_{p,q}X(V)$$ is determined by the $(p,q)$-derivative of $D_{p,q}X$. This is analogous to studying the layers of the Taylor tower of approximations in the underlying calculus, see Section \ref{sec: classification n-homog}, and the statement is similar to that of Weiss \cite[Theorem 9.1]{Wei95}.
\begin{theorem}
For all $X\in\Ezero$, $p,q\geq 1$ and $V\in\Jzero$, there exists homotopy fibre sequences 
\begin{equation*}
  \Omega^\infty[(S^{(p,q)V}\wedge\Theta_{D_{p,q}X}^{p,q})_{hO(p,q)}] \rightarrow T_{p+1,q}T_{p,q+1}X(V)\rightarrow T_{p,q}X(V)
\end{equation*}
\end{theorem}
\begin{proof}
The map $T_{p+1,q}T_{p,q+1}X\rightarrow T_{p,q}X$ is exactly $T_{p+1,q}T_{p,q+1}$ of the canonical inclusion map $X\rightarrow T_{p,q}X$, by Lemma \ref{weiss6.3.2}. 

Let $D_{p,q}X$ be the homotopy fibre of this map. Then $D_{p,q}X$ is $(p,q)$-homogeneous (see Theorem \ref{thm: DYpq is homog}). An application of the classification Theorem \ref{weissclassification} gives that 
\begin{equation*}
    D_{p,q}X(V)\simeq \Omega^\infty[(S^{(p,q)V}\wedge \Theta_{D_{p,q}X}^{p,q})_{hO(p,q)}].\qedhere
\end{equation*}

\end{proof}

\section{The functor $BO(-)$}\label{sec: BO}

The complexity of computations in orthogonal calculus is widely acknowledged. In \cite[Example 2.7]{Wei95}, Weiss gives calculations for low degree derivatives of the functor $$BO(-):V\rightarrow BO(V)$$ where $V\in \mathcal{J}_0$. In \cite[Theorem 2]{Aro02}, Arone gives a formula for the remaining higher derivatives of this functor. 

In this section, we find some good candidates for the first derivatives of an input functor for $\C$-equivariant orthogonal calculus, $BO(-)\in \Ezero$. These calculations are analogous to that of \cite[Example 2.7]{Wei95} for $BO^{(1)}(-)$. 

We define an input functor to $\C$-equivariant orthogonal calculus $BO(-)$ by 
\begin{equation*}
    BO(-):V\rightarrow BO(V)
\end{equation*}
where $V\in \Jzero$. The space $BO(V)$ is the classifying space of $O(V)$, with the $\C$-action inherited from the $\C$-action on $V$. 

More specifically, since $O(V)$ has a $\C$-action, which is conjugation by the matrix $A$ defined in \ref{def:O(p,q) and matrix A}, finite products $O(V)\times \dots \times O(V)$ can be equipped with the diagonal $\C$-action. The universal space $EO(V)$ is the geometric realisation of a simplicial space whose $n$-simplices are $(n+1)$-tuples of elements of $O(V)$, see \cite[Example 1B.7]{Hat02}. As such, $EO(V)$\index{$EO(V)$} inherits a $\C$-action, which is the diagonal action on simplices, and so does the orbit space $EO(V)/O(V)=:BO(V)$\index{$BO(V)$}. 

Note that the universal space $EO(V)$ is a contractible $\C$-space. That is, the equivariant homotopy groups $\pi_n^H EO(V)$ are trivial for all $n$ and for each closed subgroup $H$ of $\C$. This follows from the fact that taking fixed points commutes with taking the geometric realisation (see \cite[Section V.1]{May96}), so that in particular
\begin{equation*}
    \left(EO(p,q)\right)^{\C} \cong E\left(O(p,q)^{\C}\right)= E\left(O(p)\times O(q)\right)\cong EO(p)\times EO(q),
\end{equation*}
which is a contractible space, since $EO(p)$ and $EO(q)$ are both contractible. 

Consider the long exact sequence of $\C$-equivariant homotopy groups $\pi_n^H$ on the $\C$-equivariant fibre sequence $O(V)\rightarrow EO(V)\rightarrow BO(V)$. Since the homotopy groups $\pi_n^H EO(V)$ are all trivial, we see that there is a weak equivalence of $\C$-spaces $$\Omega BO(V)\simeq O(V).$$ Here the $\C$-action on $S^1$ is trivial and so $\C$ acts on $f\in \Omega BO(V)$ by $\sigma (f):= \sigma \circ f$. 

\begin{remark}
    There are other descriptions for the classifying space of $O(V)$, which may lead to more interesting calculations. For example, in \cite{GMM17}, Guillou, May and Merling define a classifying space for principal $(G,\Pi)$-bundles, where $G$ and $\Pi$ are topological groups and $G$ acts on $\Pi$ so that
    \begin{equation*}
        1\rightarrow \Pi \rightarrow \Pi\rtimes G\rightarrow G\rightarrow 1
    \end{equation*}
    is a split extension. Since the conjugation $\C$-action on $O(V)$ induces a split exact sequence identical to that above (see Definition \ref{def:O(p,q) and matrix A}), this theory of $(\C, O(V))$-bundles is applicable. The description of $E(G,\Pi)$ is given in terms of the classifying space of a category of functors $\mathscr{C}(\tilde G, \tilde \Pi)$, and $B(G,\Pi)=E(G,\Pi)/\Pi$ is the orbit space. In \cite[Chapter VII]{May96}, May presents a more general theory of $(\Pi:\Gamma)$-bundles that is applicable to non-split extensions
        \begin{equation*}
        1\rightarrow \Pi \rightarrow \Gamma\rightarrow G\rightarrow 1.
    \end{equation*}
\end{remark}

The `calculation' of the derivatives of $BO(-)\in\Ezero$ uses the iterative description of derivatives given by the homotopy fibre sequence of Proposition \ref{loops fibre sequence}.

Let $BO(-)\in \Ezero$ be defined as above. By Proposition \ref{loops fibre sequence}, there exist $\C$-equivariant homotopy fibre sequences
\begin{equation*}
    BO^{(1,0)}(V)=\ind_{0,0}^{1,0}BO(V)\rightarrow BO(V)\rightarrow BO(V\oplus \mathbb{R})
\end{equation*}
and 
\begin{equation*}
    BO^{(0,1)}(V)=\ind_{0,0}^{0,1}BO(V)\rightarrow BO(V)\rightarrow BO(V\oplus \mathbb{R}^\delta) .
\end{equation*}

The subgroup inclusion map $i:O(V)\hookrightarrow O(V\oplus X)$ induces a $\C$-equivariant homotopy fibre sequence $$O(V\oplus X)/O(V)\simeq \Omega B(O(V\oplus X)/O(V))\rightarrow BO(V)\rightarrow BO(V\oplus X)$$where $X=\R$ or $X=\Rdelta$. 

That is, there exist $\C$-equivariant homotopy fibre sequences
\begin{equation*}
   O(V\oplus \mathbb{R})/O(V) \rightarrow BO(V)\rightarrow BO(V\oplus \mathbb{R})
\end{equation*}
and 
\begin{equation*}
    O(V\oplus \mathbb{R}^\delta)/O(V)\rightarrow BO(V)\rightarrow BO(V\oplus \mathbb{R}^\delta)
\end{equation*}

Therefore, the functors $E\in\C\mathcal{E}_{1,0}$ and $F\in\C\mathcal{E}_{0,1}$ defined by
\begin{align*}
    E&:V\mapsto O(V\oplus \mathbb{R})/O(V)\cong S^V\\
    F&:V\mapsto O(V\oplus\Rdelta)/O(V)\cong S^{V\otimes\mathbb{R}^\delta}
\end{align*}
are good candidates for $BO^{(1,0)}(-)$ and $BO^{(0,1)}(-)$ respectively. 

\begin{conjecture}\label{ex: BO}
The $(1,0)$-derivative of $BO(-)$ is the orthogonal sphere spectrum $\mathbb{S}:V\mapsto S^V$ and the $(0,1)$-derivative of $BO(-)$ is the twisted orthogonal sphere spectrum $\mathbb{S}^{\otimes \delta}:V\mapsto S^{V\otimes \Rdelta}$.
\end{conjecture}

\begin{remark}
In \cite{Wei95}, see also \cite{Tag}, the second derivative $BO^{(2)}(-)$ is calculated using spectral sequence methods. Trying to generalise the method to calculate $BO^{(2,0)}(-)$, $BO^{(1,1)}(-)$ and $BO^{(0,2)}(-)$ would be an interesting task. 
\end{remark}

% To add more chapters, uncomment the following. You can also add more

%\input{./Chapters/Chapter3}

%\input{./Chapters/Chapter4}

\clearpage

%----------------------------------------------------------------------------------------
%	THESIS CONTENT - APPENDICES (uncomment to use)
%----------------------------------------------------------------------------------------

%\addtocontents{toc}{\vspace{2em}} % Add a gap in the Contents, for aesthetics

%\appendix % Cue to tell LaTeX that the following 'chapters' are Appendices

% Include the appendices of the thesis as separate files from the Appendices folder
% Uncomment the lines as you write the Appendices

%\input{./Appendices/AppendixA}
%\input{./Appendices/AppendixB}
%\input{./Appendices/AppendixC}

%\addtocontents{toc}{\vspace{2em}} % Add a gap in the Contents, for aesthetics

%\backmatter

%----------------------------------------------------------------------------------------
%	BIBLIOGRAPHY
%----------------------------------------------------------------------------------------

\label{Bibliography}

\clearpage 
\phantomsection
\addcontentsline{toc}{chapter}{Bibliography}

\fancyhf{}
\fancyhead[C]{Bibliography}
\fancyhead[R]{\thepage} % Change the page header to say "Bibliography"

\bibliographystyle{alpha} % Use the "unsrtnat" BibTeX style for formatting the Bibliography
\normalsize
\bibliography{Thesismain.bib} % The references (bibliography) information are stored in the file named "Thesis main.bib"

\clearpage

%----------------------------------------------------------------------------------------
%   Index - see the makeindex package documentation, note the examples of use given in the sample chapters
%----------------------------------------------------------------------------------------

\clearpage
\phantomsection %this command ensures that the hyperref package makes the cickable links go to the correct place
\addcontentsline{toc}{chapter}{Index}

\fancyhf{}
\fancyhead[C]{Index}
\fancyhead[R]{\thepage}

\printindex

\clearpage

%----------------------------------------------------------------------------------------
%   Nomenclature - see the nomenclature package documentation, comment out if not required, note the examples of use given in the sample chapters
%----------------------------------------------------------------------------------------

\clearpage
\phantomsection
\addcontentsline{toc}{chapter}{Index of notation}

\fancyhf{}
\fancyhead[C]{Index of notation}
\fancyhead[R]{\thepage}

\printnomenclature

\end{document}